\documentclass[10pt]{amsart}
\usepackage[margin=1.3in,marginparsep=0.1in]{geometry}
\usepackage{marginnote}
\usepackage{tikz-cd}
\usepackage[hidelinks]{hyperref}
\usepackage{setspace}
\theoremstyle{plain}
\newtheorem{thm}{Theorem}[section]
\newtheorem{lem}[thm]{Lemma}
\newtheorem{prop}[thm]{Proposition}
\newtheorem{si}[thm]{Situation}
\newtheorem{cor}[thm]{Corollary}
\theoremstyle{definition}
\newtheorem{defn}[thm]{Definition}
\theoremstyle{remark}
\newtheorem{rem}[thm]{Remark}
\theoremstyle{definition}
\newtheorem{ex}[thm]{Example}

\newcommand{\PP}{\mathbb{P}}
\newcommand{\ZZ}{\mathbb{Z}}

\newcommand{\la}{\lambda}
\newcommand{\aaa}{\alpha}
\newcommand*{\vv}[1]{\vec{\mkern0mu#1}}
\newcommand{\ma}{$\mathcal{G}^{k,\text{EHT}}_{2,\omega_g,d_{\bullet},a^{\Gamma}}(X_g)$}
\newcommand{\dis}{\displaystyle}

\newcommand{\ep}{\epsilon}
\newcommand{\fa}{\forall}
\newcommand{\ti}{\tilde}
\newcommand{\mF}{\mathscr{F}}
\newcommand{\mA}{\mathscr{A}}
\newcommand{\mE}{\mathscr{E}}
\newcommand{\mO}{\mathscr{O}}

\newcommand{\mL}{\mathscr{L}}
\newcommand{\GG}{\mathcal{G}}

\newcommand{\rr}{\rho_{g,k}}

\newcommand{\ot}{\otimes}
\newcommand{\op}{\oplus}
\newcommand{\oo}{\omega}
\newcommand{\cd}{\cdot}

\newcommand{\tb}{\textbf}

\newcommand{\cm}{\mathcal{M}}
\usepackage{amsfonts}
\usepackage{pdflscape}
\usepackage{amsmath,amssymb,latexsym,amsthm}
\usepackage{mathtools}
\usepackage{graphicx}
\usepackage{mathrsfs}
\usepackage{amscd}
\usepackage{color}

\newenvironment{st}[2][Step]{\begin{trivlist}
\item[\hskip \labelsep {\bfseries #1}\hskip \labelsep {\bfseries #2}]}{\end{trivlist}}
\DeclareMathOperator\Tor{Torsion}
\DeclareMathOperator\rk{rank}
\DeclareMathOperator\Gr{Gr}
\DeclareMathOperator\im{Im}

\DeclareMathOperator\rev{Reverse}
\DeclareMathOperator\supp{Supp}
\DeclareMathOperator\Pic{Pic}
\DeclareMathOperator\aut{Aut}
\DeclareMathOperator\hilb{Hilb}
\DeclareMathOperator\sy{Sym}
\DeclareMathOperator\s{span}
\DeclareMathOperator\spec{spec}
\DeclareMathOperator\iso{Iso}

\DeclareMathOperator\van{Van}
\DeclareMathOperator\ord{ord}
\DeclareMathOperator\Ob{Ob}
\DeclareMathOperator\h{Hom}

\input xy
\xyoption{all}
\DeclareRobustCommand{\SkipTocEntry}[5]{}
\begin{document}
\title{Towards the Bertram-Feinberg-Mukai Conjecture}
\author{Naizhen Zhang}
\maketitle
\begin{abstract}
In this paper, we prove the existence portion of the Bertram-Feinberg-Mukai Conjecture for an infinite family of new cases using degeneration technique. This not only leads to a substantial improvement of known results but also develops finer tools for analyzing the moduli of rank two limit linear series which should be useful for other applications to other higher-rank Brill-Noether Problems.
\end{abstract}
\tableofcontents
\section{Introduction}
Brill-Noether theory studies the moduli space of vector bundles of fixed rank and degree, with given amount of global sections, over algebraic curves of a given genus. It relates the intrinsic geometry of a curve to its extrinsic geometry. For example, in the classical case, we study the extrinsic data of maps from a curve to projective spaces through studying the moduli of \textit{linear series}, i.e. pairs of the form $(\mL,V)$, where $\mL$ is a line bundle and $V$ is a space of sections. There, the basic questions are non-emptiness, dimension count, connectedness and smoothness of the moduli space and all of them were answered in early 1980's by Kempf, Kleiman-Laksov, Griffiths-Harris, and Gieseker. (\cite{ACGH},\cite{Gieseker},\cite{GH})

Brill-Noether theory has been shown to be useful in studying the geometry of the moduli space of genus $g$ curves $\overline{\mathcal{M}}_g$ via effective divisors. The Brill-Noether Theorem indicates that the space of degree $d$, dimension $k$ linear series on a \textit{general} genus $g$ curve is empty, when $\rho=k(d-k+1)-(k-1)g<0$. For $(k,d,g)$ such that $\rho=-1$, consider all genus $g$ curves with degree $d$, dimension $k$ linear series, then one gets a natural candidate for an effective divisor in $\overline{\mathcal{M}}_g$. With this observation, Eisenbud and Harris showed that $\mathcal{M}_g$ is of general type for $g\ge24$. (\cite{EHgen})

Higher-rank Brill-Noether theory is a natural generalization of the classical case, which studies \textit{high-rank linear series}. Geometrically, they correspond to maps from curves to Grassmannians. High-rank Brill-Noether theory is also useful in studying the birational geometry of $\mathcal{M}_g$. One well-known application due to Farkas and Popa produced a counterexample to the slope conjecture. (\cite{FPSlo})

Following Narasimhan and Seshadri, another viewpoint towards high-rank Brill-Noether theory is related to the representation theory of the fundamental group of the underlying Riemann surface. For example, rank 2 linear series with canonical determinant of certain dimension correspond to some type of irreducible $SU(2)$-representations of the fundamental group. (See \cite{Mnonab}.) Other interesting applications lie in the study of Fano 3-folds, vector bundles on K3 surfaces and classification of curves of low genera, to name a few. (See \cite{MK3}, \cite{MCGrass}, \cite{MCsymm}, \cite{MCsymm2}.) Recently, Bhosle, Brambila-Paz, and Newstead used high-rank Brill-Noether theory to solve a conjecture by Butler related to classical linear series. (See \cite{BBN}.)

It was conjectured by Bertram, Feinberg (\cite{Ber}) and Mukai (\cite{MCan}) that the moduli space of rank two vector bundles with canonical determinant and at least $k$ sections on a genus $g$ curve has expected dimension $\rr=3g-3-\binom{k+1}{2}$. In these pioneering papers, the moduli count was verified for some lower genuses. Later, in \cite{M1}, the  conjecture was partially verified by Teixidor i Bigas for the following cases: When $k=2k_1+1$ is odd, it was verified for all $g\ge k_1^2+k_1+1$; and when $k=2k_1$ is even, it was verified for all $g\ge k_1^2,k_1>2$, together with the cases $(g,k)=(5,4),(3,2)$. This was achieved using degeneration technique, involving limit linear series $((\mE_j,V_j),(\phi_j))$ (see section \ref{definition}) on a chain of elliptic curves, $X_g=C_1\cup...\cup C_g$, with the following property: for every component $C_j$, $\mE|_{C_j}$ is semi-stable over $C_j$. More recently, Lange, Newstead and Park showed non-emptiness of the moduli space for $g\ge {k(k-1)\over4}+1$, $k\ge8$, where $g$ is an odd prime number.

For these known results, the parameters $(g,k)$ should satisfy some inequality of the form $g\ge k_1^2+\aaa k_1+\beta$, where $k_1=\left\lfloor{k\over2}\right\rfloor$. Note that the Bertram-Feinberg-Mukai Conjecture (existence portion) is made for all $(g,k)$ satisfying an inequality of the form ${2\over3}k_1^2+\aaa k_1+\beta$.  In this paper, we improve the known results so that asymptotically $g\sim {11\over12}k_1^2$. To the author's knowledge, this is the first known result reducing the quadratic coefficient of the cut-off inequality to a number strictly smaller than 1. To do so, we utilize limit linear series $((\mE_j,V_j)(\phi_j))$ (see Section 5) such that not all vector bundles $\mE_j$ are semi-stable, but nevertheless the resulting bundle on the whole curve $X_g$ still satisfies certain semi-stability condition. Our main result is the following:

\begin{thm}\label{mresult}
 For the pairs of $(g,k)$ satisfying one of the following:
$$\begin{cases}
g\ge k_1^2+k_1-\lfloor{(k_1-2)^2+3\over12}\rfloor, k=2k_1+1\ge5\\
g\ge k_1^2-\lfloor{(k_1-4)^2\over12}\rfloor-1, k=2k_1\ge8,
\end{cases}$$
such that $\rr=3g-3-\binom{k+1}{2}\ge0$, in general the rank two Brill-Noether locus with canonical determinant has at least one component of the expected dimension $\rr$.
\end{thm}
Technically, we relate various limit linear series moduli stacks via naturally defined stack morphisms and thus reduce the combinatorial complexity in the analysis. We analyze fibers of these stack morphisms, introduce the notion of configuration at a point, which together lead to a purely combinatorial moduli counting procedure. These aspects should be useful when we apply the same technique to other Brill-Noether Problems.

A brief structure of the paper is as follows: In section 3 and 8, we give numerical descriptions of semi-stability and canonical determinant conditions for vector bundles on the degenerated curve. Instead of the usual notion of $\mu$-semi-stability after Seshadri, we adopt the notion of $\ell$-semi-stability define by Osserman in \cite{Ostab}. This greatly simplifies the verification of the stability condition. Section 4, 5, 6, 7 provide basic terminologies and preliminary results on limit linear series, configurations and vanishing sequences. Some result from section 4 is also used in \cite{Osp} in a more general context. Section 9 and 10 contain the main constructions and proofs, while we leave some of the technical results in the appendix. We derive our final conclusion in section 11.\\
\ \\
\noindent\textbf{Acknowledgments.} The author would like to thank his advisor Brian Osserman for introducing this problem and his tireless instructions, without which this paper would not have come into existence.
\section{Notations and Convention}
\noindent\textbf{Notation 1}
$\rr=3g-3-\binom{k+1}{2}$.\\
\ \\
\textbf{Notation 2}
$L(g,k)=\left\{\begin{aligned} &2(g-k_1^2)&\text{ when $k=2k_1$ is even}\\ &2(g-(k_1^2+k_1+\frac{1}{2}))&\text{ when $k=2k_1+1$ is odd.} \end{aligned}\right.$\\
\ \\
\noindent\tb{Notation 3\ }Denote $|A|$ to be the cardinality of a finite set $A$.\\
\ \\
\noindent\tb{Notation 4\ }Given a sequence $(a_1,...,a_n)$, denote $(a_1,a_2,...,a_n)^{\rev}=(a_n,a_{n-1},...,a_1)$.\\
\ \\
\noindent\tb{Notation 5\ }Let $s$ be a section of a vector bundle $\mE$ on a smooth algebraic curve $C$. Let $P$ be a point on $C$. We denote $\ord_P(s)$ to be the order of vanishing of $s$ at $P$. Let $V$ be a $k$-dimensional subspace of $\Gamma(\mE)$. We denote $\van_P(V)$ to be the vanishing sequence of $V$ at $P$.\\
\ \\
\noindent\tb{Notation 6\ }Let $\mE$ be a vector bundle on a smooth algebraic curve $C$. Let $P$ be a point on $C$, and $V$ be a subspace of $\Gamma(\mE)$. We denote $V(-nP)=V\cap\Gamma(\mE(-nP))$ to be the subspace of sections in $V$, vanishing to order at least $n$ at $P$.\\
\ \\
\noindent\tb{Notation 7\ }Let $C$ be an elliptic curve, and $P,Q\in C$ be two general points. Write $\mO(a,b)\op \mO(c,d)$ for the locally-free sheaf $\mO(aP+bQ)\op\mO(cP+dQ)$.\\
\ \\
\noindent\tb{Notation 8\ }Denote $X_g$ to be a chain of $g$ elliptic curves. When $g<g'$, we write $X_g\subset X_{g'}$ only when $X_g$ is the connected sub-curve consisting of the first $g$ irreducible components of $X_{g'}$. We also denote $\omega_g$ to be the dualizing sheaf on $X_g$.\\
\ \\
\noindent\tb{Notation 9\ }Let $\mE$ be a rank-$r$ vector bundle on a smooth algebraic curve $C$. Denote $Aut^0(\mE)$ to be the group of automorphisms in compatible with a fixed determinant map $\psi:\Lambda^r\mE\to L$.\\
\ \\
\noindent\tb{Notation 10\ }Let $\mE$ be a vector bundle on a smooth curve $C$, $P$ be a point on $C$. Suppose $s$ is a section of $\mE$ vanishing to order $a$ at $P$. Fix $t$ to be a uniformizer at $P$. Denote $s|_P$ to be the image of $t^{-a}s$ in $\PP \mE|_P$ under the projectivization map $\mE|_P\to \PP \mE|_P$. Notice that $s|_P$ is independent of the choice of the uniformizer.\\
\ \\
\noindent\tb{Notation 11\ }In a sequence of integers, we write $(...,[a]_n,..)$ to indicate $n$ consecutive $a$'s.\\
\ \\
\noindent\tb{Notation 12\ }Given \tb{Notation 11}, denote
$$a(k)=\begin{cases}
([0]_2,[1]_2,...,[k_1-1]_2,k_1)^T&\text{ for } k=2k_1+1 \text{ being odd}\\
([0]_2,[1]_2,...,[k_1-1]_2)^T&\text{ for } k=2k_1 \text{ being even}.
\end{cases}$$\\
\ \\
\noindent\tb{Notation 13\ }Denote $\iso(A,B)$ to be the set of isomorphisms between two objects $A,B$ in some category $\mathcal{C}$.\\
\ \\
\noindent\tb{Notation 14\ }Let $\GG$ be an algebraic stack. Denote $|\GG|$ to be its underlying topological space.\\
\ \\
\noindent\tb{Notation 15\ }Let $\psi:\GG\to\mathcal{H}$ be a 1-morphism between algebraic stacks. Denote $|\psi|$ to be the induced map between underlying topological spaces $|\GG|\to|\mathcal{H}|$.\\
\ \\
\noindent\tb{Notation 16\ }Denote $\GG^{k,\text{EHT}}_{2,\mL,d_{\bullet}}(X_g)$ to be the moduli stack of rank two, dimension $k$, Eisenbud-Harris-Teixidor limit linear series over a chain of $g$ curves, with component-wise degree $d_{\bullet}=(d_1,...,d_g)$ and determinant $\mL$. Denote $\GG^{k,\text{EHT}}_{2,\mL,d_{\bullet},a^{\Gamma}}(X_g)$ to be its locally-closed substack of limit linear series with prescribed vanishing sequences $a^{\Gamma}$ at the nodal points of $X_g$.\\
\ \\
\noindent\tb{Notation 17\ }When fixing $\GG=\GG^{k,\text{EHT}}_{2,\oo_g,d_{\bullet},a^{\Gamma}}(X_g)$ (\tb{Notation 16}), for $r<g$ we denote $\GG_r=\GG^{k,\text{EHT}}_{2,\oo_r(A_rP_{r+1}),d^r_{\bullet},a^{\Gamma}_r}(X_r)$, where $X_r$ is the sub-curve consisting of the first $r$ components, $d^r_{\bullet}=(d_1,...,d_r)$, $A_r=2g(g-r)-\dis\sum_{i=r+1}^{g}d_i$ and $a^{\Gamma}_r$ consists of the vanishing sequences at the first $r-1$ nodal points.

For $1<r_1<r_2<g$, denote $\GG_{r_1,r_2}=\GG^{k,\text{EHT}}_{2,\oo_{r_1,r_2}((2g-A_{r_1})P_{r_1}+A_{r_2}P_{r_2+1}),d^{r_1,r_2}_{\bullet},a^{\Gamma}_{r_1,r_2}}(X_{r_1,r_2})$, where $X_{r_1,r_2}$ is the cub-curve of $X_{r_2}$ consisting of its last $r_2-r_1+1$ components, $a^{\Gamma}_{r_1,r_2}$ consists of the vanishing sequences at the $r_1$-th,$(r_1+1)$-th,...,$(r_2+1)$-th nodal points, $\oo_{r_1,r_2}$ is the dualizing sheaf on $X_{r_1,r_2}$ and $d^{r_1,r_2}_{\bullet}=(d_{r_1},...,d_{r_2})$.\\
\ \\
\noindent\tb{Convention} Throughout the paper, we work over an algebraically closed field $K$, whose characteristic is zero.

Let $\mE$ be a rank two locally-free sheaf on a smooth projective curve $C$. Define $u(\mE)=\max\{\deg(\mL)-{\deg(\mE)\over2}\}$, where $\mL$ runs through all invertible sub-sheaves of $\mE$. Throughout this paper, when referring to moduli stacks of limit linear series of the form $\GG^{k,\text{EHT}}_{2,\oo_g,d_{\bullet},a^{\Gamma}}(X_g)$, we restrict ourselves to the open locus determined by the condition $u(\mE_j)\le{1\over2}$ for all $j$ (see also \ref{applc}), where $\mE_j$ is the underlying vector bundle of the limit linear series over the $j$-th component of $X_g$.
\section{$\ell$-Semi-stability Sheaves over Reducible Nodal Curves}
In this paper, we approach the Bertram-Feinberg-Mukai Conjecture using degeneration technique. To do so, we need to analyze ``limits" of stable vector bundles over the degenerated curve.

In \cite{M1}, Teixidor i Bigas also used degeneration technique to obtain a partial result on the existence portion of the conjecture. There, Teixidor i Bigas considered vector bundles $\mE$ on a reducible nodal curve $X$ with the following property: for every component $C_j$, $\mE|_{C_j}$ is semi-stable.

Beyond the cut-off line for this partial result (see Introduction), one can no longer utilize vector bundles on $X$ with such nice properties to tackle the problem: such vector bundles do not have sufficiently many sections. Therefore, to improve the result in \cite{M1}, some $\mE|_{C_j}$ have to be unstable. We use the quantity $L(g,k)$ (see \textbf{Notation 2}) to measure the complexity of construction for a certain pair $(g,k)$: the larger $|L|$ is, the more $C_j$'s there are on which $\mE|_{C_j}$ is unstable.

Yet still we always want to consider \textit{(semi)-stable} vector bundles on $X$. Therefore, the first thing to clarify is the notion of (semi)-stability on a reducible nodal curve. For reasons explained later, we consider the notion of $\mathbf{\ell}$\tb{-(semi)-stability} recently defined by Osserman in \cite{Ostab}:

\begin{defn}
 Let $\mE$ be a rank $r$ vector bundle on a nodal curve $X$. We say that $\mE$ is $\mathbf{\ell}$\tb{-semistable} (resp. $\mathbf{\ell}$\tb{-stable}) if for all proper subsheaves $\mF\subseteq\mE$ of constant rank $r'$, $\frac{\chi(\mF)}{r'}\le\frac{\chi(\mE)}{r}$ (resp. $\frac{\chi(\mF)}{r'}<\frac{\chi(\mE)}{r}$).
\end{defn}
This notion is useful in degeneration type argument involving vector bundles on curves, due to the following result in \cite{Ostab}:
\begin{prop}(Proposition 1.4 in \cite{Ostab})
$\ell$-(semi)stability is open in families.
\end{prop}
Consequently, if $\mE$ is a vector bundle on an irreducible family $\mathcal{X}\to S$ of curves which is generically smooth, and is $\ell$-(semi)stable on some nodal fiber $\mathcal{X}_s$, then it is generically (semi)stable.

On a reducible nodal curve, one also has the notion of $\mu$-(semi)stability for vector bundles (\cite{SESH}):
\begin{defn}
Let $X$ be a nodal curve with components $C_i$. A polarization $\mu$ of $X$ is a choice of rational numbers $w_i\in(0,1)$ such that $\sum w_1=1$. A depth one sheaf $\mE$ of rank $n$ on $X$ is said to be (semi)stable with respect to $\mu$ if for every torsion-free subsheaf $\mF$ of $\mE$ with rank $r_i$ on $C_i$,
$$\frac{\chi(\mF)}{\sum w_ir_i}(\le)<\frac{\chi(\mE)}{\rk(\mE)}.$$
\end{defn}
\begin{rem}
  When $\chi(\mE)=0$, this notion is independent of the choice of $\mu$.
\end{rem}
The notion of $\ell$-(semi)stability is weaker than $\mu$-(semi)stability. See the following example.

\begin{ex}\label{ex1} Let $X$ be a chain of two smooth curves, $C_1,C_2$. Denote $P=C_1\cap C_2$ to be the nodal point. Consider a rank two locally-free sheaf $\mE$ on $X$ such that such that $\mE|_{C_i}=\mL_{i,1}\op \mL_{i,2}$ is decomposable for $i=1,2$. Moreover, suppose $$\chi(\mL_{1,1})=0,\chi(\mL_{1,2})=2,\chi(\mL_{2,1})=\chi(\mL_{2,2})=0; \mL_{2,1}\not\cong \mL_{2,2};$$
and the gluing at $P$ is chosen so that $\mL_{1,2}$ is not glued to $\mL_{2,1}$ or $\mL_{2,2}$. Notice that $\chi(\mE)=0$.

Since there exists $\mF\subset\mE$ only supported on $C_1$ and $\chi(\mF)=\chi(\mL_{1,2}(-P))=1>0$, $\mE$ is not $\mu$-semistable. However, the assumption that $\mL_{2,1}\not\cong \mL_{2,2}$, and the assumption on the gluing at $P$ imply that any subsheaf $\mF$ of rank 1 on $X$ has $\chi(\mF)\le 0$. Hence, $\mE$ is $\ell$-semistable.
\end{ex}

On a reducible curve $X$, $\ell$-semistability is also more flexible than $\mu$-semistability in the following way: given some $\mE$ which is $\ell$-semistable, one may have some component $C_t$ of $X$ such that $\chi(\mE|_{C_t})$ is arbitrarily large or small. (This is impossible for $\mu$-semistable sheaves with some general $\mu$.)

\begin{ex}\label{ex2}
Let $\mE$ be an $\ell$-semi-stable locally-free sheaf on a reducible nodal curve $X$. Suppose $X_1,X_2$ are connected sub-curves of $X$ and $X=X_1\cup X_2$, $X_1\cap X_2=\{P\}$. Let $\mE'$ to be a sheaf on $X$ such that $\mE'|_{X_1}=\mE|_{X_1}(nP),\mE'|_{X_2}=\mE|_{X_2}(-nP)$ and the gluing at $P$ is induced by the gluing of $\mE$ at $P$.\footnote{$\mE|_{X_1}|_P$ and $\mE|_{X_1}(nP)|_P$ (similarly, $\mE|_{X_2}|_P$ and $\mE|_{X_2}(-nP)|_P$) are naturally identified up to a scalar. Given an isomorphism $\mE|_{X_1}|_P\stackrel{\sim}{\to}\mE|_{X_2}|_P$, one non-canonically gets an isomorphism $\mE|_{X_1}(nP)|_P\stackrel{\sim}{\to}\mE|_{X_2}(-nP)|_P$. However, the choice of scalar does not affect our discussion here.}  Then, $\mE'$ is also $\ell$-semistable. This is obvious, since if $\mF'\subset\mE'$ is a subsheaf of constant rank such that $\chi(\mF')>\chi(\mE')$, then there exists some $\mF\subset\mE$ such that $\mF|_{X_1}=\mF'|_{X_1}(-nP)$, $\mF|_{X_2}=\mF'|_{X_2}(nP)$ and $\chi(\mF)>\chi(\mE)$, which violates the $\ell$-semistability of $\mE$.
\end{ex}

One main motivation for applying the notion of $\ell$-semistability is the following (see \cite{Ostab}): $\ell$-semistability behaves well with respect to gluing two nodal curves at a (smooth) point.
\begin{prop}\label{blocks}
Let $X=Y\cup Z$ be a nodal curve, and the subcurves $Y$ and $Z$ meet at $P$. Given a vector bundle $\mE$ on $X$ of rank $r$, if $\mE|_Y$ and $\mE|_Z$ are $\ell$-semistable on $Y$ and $Z$ resp., then $\mE$ is $\ell$-semistable on $X$.
\end{prop}
This turns out to be very helpful in practice. For our application, what matters is to give a simple (sufficient) condition for a rank 2 vector bundle $\mE$ (with canonical determinant) over some reducible nodal curve $X$ to be $\ell$-semistable. Based on Proposition \ref{blocks}, we shall try to decompose $X$ into a union of connected sub-curves $X^i$, where $X^i$ and $X^{i+1}$ meet at a nodal point; it then suffices to check that every $\mE|_{X^i}$ is $\ell$-semistable over $X^i$. Moreover, given any $\mE$, we shall try to decompose $X$ in a way such that there is a simple criterion for $\mE|_{X^i}$ to be $\ell$-semistable.

Before we discuss the details, an important point to make is that one does not need to consider arbitrarily general $\ell$-semistable sheaves. Based on the work by Teixidor-i-Bigas in \cite{M2} (See Claim 2.3 in \cite{M2}), it is unnecessary to consider vector bundles on $X$ whose restriction to one component is ``too" unstable. (This applies generally, not just in the canonical determinant situation.) We therefore work under the following assumptions:
\begin{si}\label{s1}
\begin{enumerate}
  \item $X$ is a chain of $n$ smooth curves, $C_1,...,C_n$.
  \item $\mE$ is a rank 2 vector bundle over $X$ such that $\mE|_{C_t}=\mL_{t,1}\op\mL_{t,2}$ is decomposable and $|\chi(\mL_{t,1})-\chi(\mL_{t,2})|\le1$.
  \item $(\chi(\mL_{t,1}),\chi(\mL_{t,2}))\in\{(A,A+1),(A+1,A+1),(A+1,A+2),(A+2,A+2)\}$, for all $t$ and some integer $A$.
\end{enumerate}
\end{si}
\begin{rem}\label{similar}
  The first assumption reflects our choice of a particular kind of reducible nodal curves to degenerate to. The second assumption reflects the observation in \cite{M2} mentioned above.

  We justify the third assumption as follows: one can define an equivalence relation among locally-free sheaves on $X$, namely $\mF$ is \tb{similar to} $\mF'$ if and only if for every $t$, $\mF'|_{C_t}=\mF|_{C_t}(-a_tP_t+a_{t+1}P_{t+1})$ for integers $a_1,...,a_{n+1}$ ($a_1,a_{n+1}=0$), and the gluing for $\mF'$ at $P_t$ is induced by the gluing for $\mF$ at $P_t$. (See Example \ref{ex2}.) It is easy to see that any $\mF$ satisfying the second assumption is similar to some $\mF'$ which satisfies assumption (3). Eventually, we want to consider sections of certain vector bundles $\mF$ on $X$. However, this is equivalent to considering corresponding sections of a similar vector bundle. This will be obvious once we describe our constructions.
\end{rem}
We want to understand which locally-free sheaves $\mE$ in Situation \ref{s1} are $\ell$-semistable.

In general, to check $\ell$-semistability in rank 2, one needs to check all subsheaves $\mF\subset\mE$ of rank 1. For any given $\mF$ of rank 1, one can decompose $X$ into a union of connected sub-curves $X^1,...,X^m$, where $X^i$ and $X^{i+1}$ meet at a nodal point, such that $\mF_i:=\mF|_{X^i}/\Tor$ is locally-free on $X^i$. It is not hard to see that $\chi(\mF)=\sum_{i=1}^m\chi(\mF_i)$. Moreover, $\mF_i$ is not saturated (in $\mE|_{X^i}$) at $P=X^i\cap X^{i+1}$ if $i\neq m$, and not saturated at $P'=X^{i-1}\cap X^i$ if $i\neq 1$. Furthermore, suppose $\mE$ is in Situation \ref{s1}, we define the following quantities:
\begin{enumerate}
  \item $f_j:=\mE|_{C_j}-2A$, and $f_j\in\{1,2,3,4\}$.
  \item Given a rank 1 subsheaf $\mF$ and suppose $C_j\subset X^i$,

  $\ep_j(\mF):=\begin{cases}
    \frac{f_j}{2}-1&\text{ if }\mE|_{C_j}\text{ is semistable and }\mF_i|_{C_j}\text{ not a summand of }\mE|_{C_j}\\
    \frac{f_j}{2}&\text{ if }\mE|_{C_j}\text{ is semistable and }\mF_i|_{C_j}\text{ a summand of }\mE|_{C_j}\\
    \frac{f_j+1}{2}&\text{ if }\mE|_{C_j}\text{ is unstable and }\mF_i|_{C_j}\text{ destabilizes }\mE|_{C_j}\\
    \frac{f_j-1}{2}&\text{ if }\mE|_{C_j}\text{ is unstable and }\mF_i|_{C_j}\text{ does not destabilize }\mE|_{C_j}
  \end{cases}$

  \item Given a rank 1 subsheaf $\mF$, $\nu_i(\mF):=\begin{cases}0\text{ if }i=1=m\\1\text{ if }i=1\text{ or }m, \text{ and }1\neq m\\2\text{ otherwise }\end{cases}$.
\end{enumerate}
Using these notations, one can express $\chi(\mF)$ as follows
\begin{equation}
  \chi(\mF_i)=\sum_{j:C_j\subset X^i}(A+\ep_j(\mF))-(m_i-1)-\nu_i(\mF)= m_i\cdot A+\sum_{j:C_j\subset X^i}\ep_j(\mF)-m_i+1-\nu_i(\mF)
\end{equation}
(Here, $m_i$ is the number of irreducible components of $X^i$.)

Consequently, $\chi(\mF)=\sum\chi(\mF_i)= n\cdot A+\sum_{j=1}^n\ep_j(\mF)-\sum_i(m_i-1+\nu_i(\mF))$.

A simple counting shows that $\sum_i(m_i-1+\nu_i(\mF))=m+n-2$. So, we get
\begin{equation}
  \chi(\mF)=n\cdot A+\sum_{j=1}^n\ep_j(\mF)-m-n+2
\end{equation}
Meanwhile, $\mu(\mE)=\frac{2nA+\sum_jf_j-2(n-1)}{2}=nA+\frac{\sum_jf_j}{2}-(n-1)$, with $f_j=1,2$ or 3. Therefore, the $\ell$-semistability condition reduces to the following inequality:
\begin{equation}\label{s2}
  \sum_{j=1}^n\ep_j(\mF)-(m-1)\le \frac{\sum_jf_j}{2},\text{ } \fa \mF\subset\mE\text{ of rank one}
\end{equation}
We now state a simple criterion for a particular type of locally-free sheaves on $X$ to be $\ell$-semistable.
\begin{prop}\label{ssimple}
Suppose $\mE$ is a locally-free sheaf on $X$ of Situation \ref{s1}. If $\mE|_{C_t}$ is unstable for at most two $t$, then $\mE$ is $\ell$-semistable if and only if there does not exist an invertible subsheaf $\mL\subset \mE$ such that $\mL|_{C_t}$ is a summand of $\mE|_{C_t}$ whenever $\mE|_{C_t}$ is semistable and $\mL|_{C_t}$ destabilizes $\mE|_{C_t}$ whenever $\mE|_{C_t}$ is unstable.\end{prop}
\begin{proof}
Define $\nu(\mE)=|\{t|\mE|_{C_t}\text{ is unstable}\}|$. By \ref{blocks}, if $\nu(\mE)=0$, then $\mE$ is $\ell$-semistable.

Suppose $\nu(\mE)=1$ and $\mE|_{C_t}$ is unstable. By inequality \eqref{s2}, a rank one subsheaf $\mF$ violates the $\ell$-semistability condition if and only if $\ep_t(\mF)=\frac{f_j+1}{2}$, $\ep_j(\mF)=\frac{f_j}{2}$ for all $j\neq t$ and $m=1$. This means that $\mF\subset\mE$ is invertible and $\mF|_{C_j}$ is a summand of $\mE|_{C_j}$ for all $j\neq t$, and $\mF|_{C_t}=\mL_{t,2}$ (the destabilizing summand of $\mE|_{C_t}$).

When $\nu(\mE)=2$ and $\mE|_{C_{t_1}}$, $\mE|_{C_{t_2}}$ are unstable, still using inequality \eqref{s2}, one can conclude that $\mE$ is $\ell$-semistable if and only if there does not exist an invertible subsheaf $\mF$ such that $\mF|_{C_j}$ is a summand of $\mE|_{C_j}$ for all $j\neq t_1,t_2$, and $\mF|_{C_{t_1}}=\mL_{t_1,2}$, $\mF|_{C_{t_2}}=\mL_{t_2,2}$ (destabilizing summands of $\mE|_{C_{t_1}}$, $\mE|_{C_{t_2}}$ resp.).
\end{proof}
For our application, this criterion is good enough: we shall consider vector bundles $\mE$ on $X$ of Situation \ref{s1}, such that one can decompose $X$ as a union of connected sub-curves $X^i$ and there are at most two components $C_t$ in every $X^i$ such that $\mE|_{C_t}$ is unstable. We use Proposition \ref{ssimple} to verify that $\mE|_{X^i}$ is $\ell$-semistable. It then follows from Proposition \ref{blocks} that $\mE$ is $\ell$-semistable.
\begin{rem}
  We briefly mention the benefit from adopting the notion of $\ell$-semistability instead of $\mu$-stability. On one hand, it greatly reduces the combinatorial complexity which occurs in our original approach using the latter notion. On the other, in situations where $\chi(\mE)\neq0$, one does not need to make choices of polarizations on the nodal curve $X$.
\end{rem}
We end up this section by clarifying a theoretical point: although the Bertram-Feinberg-Mukai conjecture is stated for stable vector bundles, it suffices to check semi-stability condition for our constructions. This is justified by the following lemma:
\begin{lem}\label{reduce}
 Let $(g,k)$ be a pair of positive integers. Suppose $L(g,k)<0$, then over a general smooth projective curve of genus $g$, there does not exist a strictly semi-stable, rank 2, vector bundle $\mE$ with canonical determinant such that $h^0(\mE)\ge k$.
\end{lem}
\begin{proof}
A strictly semi-stable, rank 2, vector bundle $\mE$ can be realized as $0\to \mL\to \mE\to \mL'\to 0$. If $\det(\mE)\cong \omega$ is canonical, we further get $\mL\otimes \mL'=\omega$, and $\deg(\mL)=\deg(\mL')=g-1$.

By Riemann-Roch theorem, $h^0(\mL)-h^0(\omega\ot \mL^{-1})=h^0(\mL)-h^0(\mL')=1+(g-1)-g=0$, i.e. $h^0(\mL)=h^0(\mL')$. We shall show that $h^0(\mL)<\frac{k}{2}$.

First, suppose $k=2k_1+1$ is odd. $L<0$ implies $g\le k_1^2+k_1$. We compute the classical Brill-Noether number $\rho(k_1+1,g-1,g)=(k_1+1)(g-k_1-1)-k_1g=g-k^2_1-2k_1-1$. When $g\le k_1^2+k_1$, this number is negative. By Brill-Noether theorem, the classical Brill-Noether space $\mathcal{G}^{k_1}_{g-1}$ is empty on a general curve.

Similarly, when $k=2k_1$ is even, the classical Brill-Noether number $\rho(k_1,g-1,g)=k_1(g-k_1)-(k_1-1)g=-k_1^2+g$. $L<0$ implies $g<k_1^2$, i.e. $\rho(k_1,g-1,g)<0$. Again, the classical Brill-Noether space is empty on a general curve.

Combine the above two paragraphs, and assume the curve is general. We see that when $k=2k_1+1$ is odd, $h^0(\mL)+h^0(\mL')=2h^0(\mL)\le 2k_1<k$; and, when $k=2k_1$ is even, $h^0(\mL)+h^0(\mL')<2k_1=k$. Therefore, $h^0(\mE)<k$ in both cases. This proves the claim.
\end{proof}
Therefore, in verifying the existence portion of the conjecture for $(g,k)$ such that $L(g,k)<0$, it suffices to check $\ell$-semi-stability condition.
\section{Configurations of Fibers of Sections}
In order to prove existence results in Brill-Noether Theory via degeneration, we need to analyze sections of vector bundles $\mE$ on a reducible nodal curve $X$. When doing so, one often comes across the following questions: when do two sections $s,s'\in\Gamma(\mE)$ give the same line inside a fiber $\mE|_P$? Given two general points $P,Q$, if $s,s'$ give the same line inside $\mE|_P$, what can be said about their fibers at $Q$? In this section, we focus on this topic and develop some results applicable to our application. We shall answer this question in the case of rank two for vector bundles on an elliptic curve $C$ of the form $\mO(a_1P+(b-a_1)Q)\op \mO(a_2P+(b-a_2)Q)$ under some relevant assumptions. These questions also come up in the broader context of rank two Brill-Noether Problems with fixed special determinant and our results are useful in greater generality (see \cite{Osp}).
\begin{rem}\label{csec}
  Let $\mE=\mO(Z_1)\op \mO(Z_2)\op...\op \mO(Z_r)$ be a decomposable, semi-stable, rank $r$ vector bundle over an elliptic curve $C$, where $Z_1,...,Z_r$ are pairwise non-linearly equivalent effective divisors of $C$. Then, up to scalar multiplication, there exist unique sections $T_1,...,T_r$ of $\mE$ such that $T_i$ is a section of the summand $\mO(Z_i)$ and the divisor it induces is $Z_i$. Hereafter, when $\mE$ is of this form, we fix an $r$-tuple of such sections and refer to them as the \tb{canonical sections} of $\mE$.
\end{rem}
\begin{lem}\label{aux}
Let $C$ be an elliptic curve, and $P,Q$ be two general points on $C$. Suppose $\mE=\mO(a_1P+(b-a_1)Q)\op \mO(a_2P+(b-a_2)Q)$ ($0\le a_1,a_2\le b$) and $\mO(a_1P+(b-a_1)Q)\not\cong \mO(a_2P+(b-a_2)Q)$. Consider pairs $(s_1,s_2)\in\Gamma(\mE)\times\Gamma(\mE)$ satisfying:
\begin{enumerate}
\item  $\ord_P(s_i)=e_i,\ord_Q(s_i)=b-e_i-1$ ($i=1,2$);
\item  $s_i|_P\neq T_j|_P\in\PP\mE|_P$, $s_i|_Q\neq T_j|_Q\in\PP\mE|_Q$ for $i=1,2$, $j=1,2$ (see  \tb{Notation 12});
\item $s_1,s_2$ are linearly independent.
\item $s_1|_P=s_2|_P\in \PP\mE|_P$, $s_1|_Q=s_2|_Q\in \PP\mE|_Q$.
\end{enumerate}
Then, there exists such a pair $(s_1,s_2)$ in $\Gamma(\mE)\times\Gamma(\mE)$ if and only if $$e_1+e_2=a_1+a_2-1, a_i-e_i\neq 0,1, \text{and }e_1\neq e_2.$$
\end{lem}
\begin{proof}
Given $T_1,T_2$ as in \ref{csec}, any section $s$ of $\mE$ can be written as $s=gT_1+hT_2$, where $g,h$ are two rational functions on $C$.

Assume $e_1+e_2=a_1+a_2-1$ holds.

Since $\deg(\mO((a_1-e_1)P+(1-(a_1-e_1))Q))=1$, by Riemann-Roch Theorem, there exists (up to a scalar) a unique $g_1\in K(C)$ such that $\ord_P(g_1)\ge-(a_1-e_1),\ord_Q(g_1)\ge-(1-(a_1-e_1))$, having no poles elsewhere. Since $P,Q$ are general, as long as $a_1-e_1\neq 0,1$, one has
$$h^0(\mO((a_1-e_1-1)P+(1-(a_1-e_1))Q))=h^0(\mO((a_1-e_1)P+(-(a_1-e_1))Q))=0.$$
Hence, one gets $\ord_P(g_1)=-(a_1-e_1), \ord_Q(g_1)=-(1-(a_1-e_1)).$

Similarly, one gets (up to scalar) a unique $h_1\in K(C)$ such that $\ord_P(h_1)=-(a_2-e_1)$,
$\ord_Q(h_1)=-(1-(a_2-e_1))$, having no poles elsewhere.

Utilizing the group structure on $C$, we consider $\la\in \aut(C)$ given by $\la:C\to C:x\mapsto -x+P+Q$. Note that $\la(P)=Q,\la(Q)=P$. We define $g_2:=h_1\circ\la,h_2:=g_1\circ\la$, and take $s_1:=g_1T_1+h_1T_2,s_2:=g_2T_1+h_2T_2$.

We now check that $(s_1,s_2)$ is a pair with the desired properties. One can directly compute $\ord_P(s_i)=e_i,\ord_Q(s_i)=b-e_i-1$ for $i=1,2$.

The second condition follows from the fact that $g_1,h_1,g_2,h_2$ are all non-zero.

Linear independence of $s_1,s_2$ follows trivially from the assumption $e_1\neq e_2$.

Lastly, we get
$$\begin{aligned}
s_1|_P=m_1\cdot T_1|_P+n_1\cdot T_2|_P,\\
s_2|_P=m_2\cdot T_1|_P+n_2\cdot T_2|_P,
\end{aligned}
\begin{aligned}
s_1|_Q=n_2\cdot T_1|_Q+m_2\cdot T_2|_Q,\\
s_2|_Q=n_1\cdot T_1|_Q+m_1\cdot T_2|_Q,
\end{aligned}$$

\noindent where $m_1=(t_1^{a_1-e_1}g_1)(P),n_1=(t_1^{a_2-e_1}h_1)(P)$, $m_2=(t_2^{a_1-e_2}h_1)(Q),n_2=(t_2^{a_2-e_2}g_1)(Q)$ and $t_1,t_2$ are uniformizers at $P,Q$ resp. Hence, by taking suitable constant multiples of $g_1,h_1$ resp., we can get a pair $(s_1,s_2)$ such that $s_1|_P=s_2|_P$, $s_1|_Q=s_2|_Q$ with all the desired properties.\\
\ \\
\indent Conversely, suppose such a pair $(s_1,s_2)$ exists. First of all, it is clear that $a_i-e_i\neq 0,1$ must hold; otherwise, condition 2 would be violated. If $e_1=e_2$, condition 3 would be violated. Again, we have $s_i=g_iT_1+h_iT_2$ ($i=1,2$) and:
 $$(g_i)=(e_i-a_1)P+(a_1-e_i-1)Q+P_i,i=1,2$$
 $$(h_i)=(e_i-a_2)P+(a_2-e_i-1)Q+Q_i,i=1,2$$
\indent Denote $n=e_1+e_2-a_1-a_2+1$. Let $D=(1-n)Q+(n+1)P$. It is then not hard to see that $g_1h_2,g_2h_1\in\Gamma(\mO(D))$. Moreover, $\Gamma(\mO(D))=\s(g_1h_2,g_2h_1)$; otherwise $g_1h_2,g_2h_1$ would be dependent, and $\{P_1,Q_2\}=\{P_2,Q_1\}$, which would further imply $P,Q$ are not general.

Now, the conditions $s_1|_P=s_2|_P\in\PP \mE|_P$, $s_1|_Q=s_2|_Q\in\PP \mE|_Q$ can be interpreted as $\frac{g_1h_2}{g_2h_1}(P)=\frac{g_1h_2}{g_2h_1}(Q)=1$. In particular, $\ord_P(g_1h_2-g_2h_1)>e_1+e_2-a_1-a_2=n-1$, and $\ord_Q(g_1h_2-g_2h_1)>(a_1-e_1-1)+(a_2-e_2-1)+1=-n-1$. Therefore, $\ord_P(g_1h_2-g_2h_1)+\ord_Q(g_1h_2-g_2h_1)\ge 0$, and $g_1h_2-g_2h_1$ only has poles at $P$ or $Q$. By the degree count, it cannot have zeros away from $P,Q$ either. Since $P,Q$ are general points, it follows that $g_1h_2-g_2h_1$ is a non-zero constant. But $1\in\Gamma(\mO(D))$ only if $1-n\ge 0, n+1\ge 0$, which implies $n=0$ or $\pm 1$. And when $n=\pm 1$, either $\ord_P(g_1h_2-g_2h_1)>0$ or $\ord_Q(g_1h_2-g_2h_1)>0$, which is impossible.
\end{proof}
\begin{cor}\label{auxc}
Let $C,P,Q,\mE$ be as given in Lemma \ref{aux}. Suppose $(s_1,s_2)\in\Gamma(\mE)\times\Gamma(\mE)$ satisfies conditions 1-3 in \ref{aux}. Then, $s_1|_P=s_2|_P\in\PP\mE|_P$ if and only if $s_1|_Q=s_2|_Q\in \PP\mE|_Q$.
\end{cor}
\begin{proof}
Let $q$ be a point in $\PP\mE|_P$ not equal to $T_1|_P,T_2|_P$. For any $e:0\le e\le b-1$ and $e\neq a_1,a_2$, there exists (up to scalar) a unique section $s$ such that $s|_P$ vanishes to orders $e,b-e-1$ at $P,Q$ resp.

Given $s_1,s_2$ such that $s_1|_P=s_2|_P\in \PP\mE|_P$, by Lemma \ref{aux}, one can find some $s'_2$ such that $s_1|_P=s'_2|_P\in \PP\mE|_P$, $s_1|_Q=s'_2|_Q\in \PP\mE|_Q$ and $\ord_P(s_2')=\ord_P(s_2)=e_2,\ord_Q(s_2')=\ord_Q(s_2)=b-e_2-1$. By the observation in previous paragraph, since $s_2|_P=s'_2|_P=s_1|_P\in \PP\mE|_P$, $s_2$ and $s'_2$ must be linearly dependent, so that $s_2|_Q=s'_2|_Q\in \PP\mE|_Q$.

The other direction follows from same argument.
\end{proof}
On the other hand, if $\mE=\mO(aP+(b-a)Q)^{\op 2}$, the situation is more rigid:
\begin{lem}\label{sym}
Given $C,P,Q$ as in Lemma \ref{aux}, suppose $\mE=\mO(aP+(b-a)Q)^{\op 2}$. Let $s_1,s_2$ be two sections such that $\ord_P(s_i)+\ord_Q(s_i)\ge b-1$. Then, $$s_1|_Q=s_2|_Q\in \PP\mE|_Q \text{ if and only if }s_1|_P=s_2|_P\in \PP\mE|_P.$$
\end{lem}
\begin{proof}
 For $c:0\le c\le b-1,c\neq a,a-1$, there is (up to scalar) a unique section $s$ of $\mO(aP+(b-a)Q)$ vanishing to orders $c,d-1-c$ at $P,Q$ resp. Hence, any section of $\mE$ vanishing to such orders at $P,Q$ can be written as $c_1(s,0)+c_2(0,s)$, where $(s,0),(0,s)$ are the images of $s$ under the two natural injections, and $c_1,c_2\in K$.

 Similarly, any section vanishing to orders $a,b-a$ at $P,Q$ resp. can be written as $c_1(T,0)+c_2(0,T)$, where $T$ is (up to scalar) the canonical section of $\mO(aP+(b-a)Q)$.

 The lemma follows immediately from these descriptions.
\end{proof}
The reader may have noticed that we only consider sections of $\mE$ which only vanish at $P,Q$. In this paper, we shall focus on such sections. Indeed, we here make a precise definition:
\begin{defn}\label{max}
Let $\mE=\mO(a_1P+(b-a_1)Q)\op \mO(a_2P+(b-a_2)Q)$ ($0\le a_1,a_2\le b$) be a rank two vector bundle over an elliptic curve. We say $s$ is a section with \tb{maximal vanishing at} $\mathbf{P,Q}$, if $s$ is either a canonical section of $\mE$ or $\ord_{P}(s)+\ord_Q(s)=b-1$.
\end{defn}
This idea of studying sections with maximal vanishing at two points of each component of a chain of elliptic curves was first proposed by Teixidor i Bigas in \cite{M1}.

The rest of this section is devoted to proving a genericity type result. Roughly, it says that given $\mE$ is as in Lemma \ref{aux}, in general two sections $s_1,s_2$ (which are not canonical sections) with maximal vanishing at $P,Q$ such that $s_1|_P\neq s_2|_P$ should have $s_1|_Q\neq s_2|_Q$. This turns out to be crucial in justifying our general construction.
\begin{thm}\label{genericity0}
Let $(C,Q)$ be an elliptic curve. For $P\neq Q\in C$, define $\mE^P=\mO(a_1P+(b-a_1)Q)\op \mO(a_2P+(b-a_2)Q)$ ($0\le a_1,a_2\le b$) and fix a pair of canonical sections $(T^P_1,T^P_2)$ of $\mE^P$. Consider $s^P_1,s^P_2\in\Gamma(\mE^P)$ satisfying:
\begin{enumerate}
\item $\ord_P(s^P_i)=e_i,\ord_Q(s^P_i)=b-e_i-1$ ($i=1,2$);
\item $s^P_i|_P\neq T^P_j|_P\in\mE^P|_P$, $s^P_i|_Q\neq T^P_j|_Q\in\mE^P|_Q$ ($i,j=1,2$);
\item $s^P_1|_Q=s^P_2|_Q\in \PP\mE^P|_Q$;
\end{enumerate}
Suppose $n=e_1+e_2-a_1-a_2+1\neq 0$, $e_i\neq a_j,a_j-1$ and $e_2>e_1+1\ge0,a_2>a_1\ge0$. Then, for general choices of $P,P'\neq Q$, there does NOT exist an isomorphism $\phi:\mE^P|_P\to\mE^{P'}|_{P'}$ sending $T^P_1|_P,T^P_2|_P,s^P_1|_P,s^P_2|_P$ to $T^{P'}_1|_{P'},T^{P'}_2|_{P'},s^{P'}_1|_{P'},s^{P'}_2|_{P'}$ respectively.
\end{thm}
To prove this theorem, we reformulate our setup as follows: let $C$ be an elliptic curve with a chosen base point $Q$ and fix two non-negative integers $a_1,a_2$. Given any non-torsion point $P\in (C,Q)$ and integers $e_i\neq a_j,a_j-1$ ($i,j=1,2$), we get a quadruple of rational functions $(g_1,h_2,g_2,h_2)$ such that $$(g_i)=(e_i-a_1)P+(a_1-e_i-1)Q+P_i, (h_i)=(e_i-a_2)P+(a_2-e_i-1)Q+Q_i,\text{ }i=1,2.$$ (We always assume $e_2>e_1>0,a_2>a_1>0$.) For any fixed choice of $g_1,h_2,g_2,h_2$, $\frac{g_1h_2}{g_2h_1}$ defines a degree 2 map $C\to\PP^1$, where $P_1,Q_2$ are the pre-images of 0 and $Q_1,P_2$ are the pre-images of $\infty$.

We start by a simple lemma.
\begin{lem}\label{endo}
  Define $\la_{ij}:C\to C:P\mapsto P'$, where $P'\sim (a_j-e_i)P+(e_i-a_j+1)Q$. Then, $\la_{ij}$ is an endomorphism of $C$.
\end{lem}
\begin{proof}
  Notice that under our assumptions, $a_j-e_i,e_i-a_j+1\neq0,1$. Denote $m=e_i-a_j$. Then, $\la_{ij}$ is the composition of the following morphisms: first, take $C\to J(C):P\mapsto [Q-P]$; then, compose with $J(C)\to J(C):[Q-P]\mapsto m\cdot[Q-P]$; further, compose with $J(C)\to \Pic^1(C):[D]\mapsto [D+Q]$; eventually, compose with the natural isomorphism $\Pic^1(C)\to C$.
\end{proof}
We now state and prove the main technical result which leads to Theorem \ref{genericity0}:
\begin{prop}\label{genericity}
 Suppose $n=e_1+e_2-a_1-a_2+1\neq 0$. Fix $Q$ to be the base point of $C$. Let $P$ be any non-torsion point of $(C,Q)$, and denote $g^P_i,h^P_i$ ($i=1,2$) to be four rational functions determined (up to a scalar) by the following:
 $$(g^P_i)+(a_1-e_i)P+(e_i-a_1+1)Q\ge0,i=1,2$$
 $$(h^P_i)+(a_2-e_i)P+(e_i-a_2+1)Q\ge0,i=1,2$$
Then, the map $\tau:P\mapsto ({g^P_1h^P_2\over g^P_2h^P_1})(P)/({g^P_1h^P_2\over g^P_2h^P_1})(Q)$ extends to a rational function $r_{e_1,e_2}$ on $C$. If either $n>0$, or $n<0$ and $e_2>e_1+1$, then $r_{e_1,e_2}$ is non-constant.
\end{prop}
\begin{proof}
  Notice that although $g_i^P,h_i^P$ are only determined up to a scalar, $g_i^P(P)/g_i^P(Q)$ and $h_i^P(P)/h_i^P(Q)$ are well-defined. In particular, $\tau$ is well-defined over the locus of non-torsion points.

  Let $M:=\mathcal{M}or_2(C,\PP^1)$ be the quasi-projective scheme parametrizing degree 2 morphisms from $C$ to $\PP^1$. (See \cite{FGA} for Grothendieck's construction of Hom schemes. See also Theorem A.1 in \cite{SMP}.) Consider the map $\phi:C\times M\to \PP^1$ which sends a pair $(P,H)$ to $H(P)$. This is nothing but the universal object of the moduli scheme. We shall now construct a map $C\dashrightarrow\PP^1$ which restricts to $\tau$ over the locus of non-torsion points and factorizes as $C\dashrightarrow Z\stackrel{\phi|_Z}{\to}\PP^1$, where $Z\subset C\times M$ is birational to $C$.

  First, let $X$ be the locus of $(P,H)$ in $C\times M$ given by the condition $H(Q)=1$. $X$ is closed in $M$ since it is the pre-image of $M'$ under the projection $C\times M\to M$, where $M'$ is the fiber over $1$ under the evaluation map $ev_Q:M\to\PP^1$.

  Next, let $Y$ be the locus of $(P,H)$ in $C\times M$ determined by the condition $(H)=\la_{11}(P)+\la_{22}(P)-\la_{12}(P)-\la_{21}(P)$, where $\la_{ij}$ are as in Lemma \ref{endo}. We claim $Y$ is also closed in $C\times M$. There are two natural morphisms $p_1,p_2:M\to \sy^2C=\hilb_2(C)$, taking a degree 2 map to its zeros and poles respectively. Denote $Y_1$ to be the equalizer of the diagram $C\times M\rightrightarrows\hilb_2(C)\cong \sy^2C$, where the top arrow is $C\times M\to C\stackrel{(\la_{11},\la_{22})}{\to}C^2\to \sy^2C$, and the bottom arrow is $C\times M\to M\stackrel{p_1}{\to} \sy^2C$. Similarly, denote $Y_2$ to be the equalizer of a diagram of the same shape, with top arrow replaced by $C\times M\to C\stackrel{(\la_{12},\la_{21})}{\to}C^2\to \sy^2C$, and the bottom arrow replaced by $C\times M\to M\stackrel{p_2}{\to} \sy^2C$. Then, $Y=Y_1\cap Y_2$ is the intersection of two closed sub-schemes, which is closed.

  Define $Z=X\cap Y\subset C\times M$, which is closed in $C\times M$. Consider the first projection $\pi_1:C\times M\to C$. We claim that $\pi_1$ maps $Z$ bijectively to $C':=C\backslash E$, where $E$ is the finite set of points consisting of $(e_2-e_1)$-torsions, $|e_i-a_j|$-torsions ($i,j=1,2$) and $(a_2-a_1)$-torsions of $(C,Q)$.

  First of all, $H^P:=\frac{g_1^Ph_2^P}{g_2^Ph_1^P}$ fails to give a degree 2 map to $\PP^1$ if and only if $\{\la_{11}(P),\la_{22}(P)\}\cap\{\la_{12}(P),\la_{21}(P)\}\neq\emptyset$. But $\la_{ii}(P)=\la_{ij}(P)$ if and only if $P$ is an $(a_2-a_1)$-torsion, and $\la_{ii}(P)=\la_{ji}(P)$ if and only if $P$ is an $(e_2-e_1)$-torsion. Secondly, when $P$ is not $(a_2-a_1)$-torsion or $(e_2-e_1)$-torsion, $H^P(Q)=1$ holds (up to appropriate scaling) if and only if $Q\notin\{\la_{11}(P),\la_{12}(P),\la_{21}(P),\la_{22}(P)\}$, i.e. $P$ is not an $|e_i-a_j|$-torsion for any $i,j$.

  Thus, $\fa P\in C'$, there exists a unique $H\in M$ satisfying the conditions $H(Q)=1$, and $(H)=\la_{11}(P)+\la_{22}(P)-\la_{12}(P)-\la_{21}(P)$. In particular, $\pi_1|_Z:Z\to C'$ is a bijective morphism. Since $C'$ is a non-empty open sub-variety of $C$, $Z$ must be a quasi-projective curve and $\pi_1|_Z$ sends its generic point to the generic point of $C$. This induces an isomorphism $K(C)\to K(Z)$, since we are working over a base field of characteristic zero. Thus, $C$ and $Z$ are birational. One can then conclude that there is a rational map $C\dashrightarrow Z\stackrel{\phi|_Z}{\to}\PP^1$, which extends to a rational function $r_{e_1,e_2}$ on $C$; over the locus of non-torsion points, it restricts to the map $\tau$.

  It only remains to check non-constancy of $r_{e_1,e_2}$. It is not hard to see that over the locus $P$ is non-torsion, $r_{e_1,e_2}(P)\in K^*$. To show it is non-constant, it then suffices to show $r_{e_1,e_2}$ has zeros or poles. To do so, it further suffices to show that there exists some $P$ which is $|e_{i'}-a_{j'}+1|$-torsion for some $i',j'$, and is not $|e_i-a_j|$-torsion for any $i,j$, nor $(e_2-e_1)$-torsion, nor $(a_2-a_1)$-torsion.

  We first assume $n=e_1+e_2-a_1-a_2+1>0$. Based on our assumptions on $e_1,e_2,a_1,a_2$, one has the following possibilities:
  $$1.\ e_2>a_2>e_1>a_1;\hspace*{3cm} 2.\ e_2>e_1>a_2>a_1;$$
  $$3.\ e_2>a_2>a_1>e_1;\hspace*{3cm} 4.\ a_2>e_2>e_1>a_1.$$
  Notice that in the first two cases $e_2-a_1+1>2$ is strictly larger than any of $|e_i-a_j|$, or $a_2-a_1$, or $e_2-e_1$. Therefore, one can find some $P$ which is an $(e_2-a_1+1)$-torsion, but not $|e_i-a_j|$-torsion, nor $(e_2-e_1)$-torsion, nor $(a_2-a_1)$-torsion, such that $r_{e_1,e_2}(P)=\infty$, and we are done.

  In the third case, consider $e'_1$ such that $e'_1+e_1=a_1+a_2-1$. Then, $e'_1>a_2$ and one can check $e_2>e_1'>a_2>a_1$. By Lemma \ref{aux}, the rational function $r_{e_1,e_2}$ agrees with $r_{e_1',e_2}$ on the locus on non-torsion points. Hence, it suffices to see that $r_{e_1',e_2}$ is non-constant, which follows directly from our discussion of the second case.

  In the last case, given $n>0$, one gets $e_2-a_1+1>e_2-a_1\ge a_2-e_1>a_2-e_2$. Also, $e_2-a_1+1>e_2-e_1$ and $e_2-a_1+1>e_1-a_1+1>e_1-a_1$. Therefore, to show there exists some $(e_2-a_1+1)$-torsion which is not $|e_i-a_j|$-torsion for any $i,j$, nor $(e_2-e_1)$-torsion, nor $(a_2-a_1)$-torsion, it suffices to show $e_2-a_1+1\nmid a_2-a_1$. But given that $n>0$ and $e_2\neq a_2-1$, $e_2-a_1+1<a_2-a_1=(a_2-e_2)+(e_2-a_1)<2(e_2-a_1)$. The result then follows.

  Hence, when $n>0$, $r_{e_1,e_2}$ is always non-constant.

  When $n<0$, one can apply a similar case by case analysis to prove the result.
  \end{proof}
Eventually, we prove Theorem \ref{genericity0}:
\begin{proof}[Proof of Theorem \ref{genericity0}]
By our assumptions and Corollary \ref{auxc}, $T^P_1|_P,T^P_2|_P,s^P_1|_P,s^P_2|_P$ are four distinct points in $\PP\mE^P|_P$.

For any fixed $P$, $s^P_1,s^P_2$ determines a quadruple of rational functions $(g^P_1,h^P_1,g^P_2,h^P_2)$ on $C$, up to scalars.

Given the condition $s^P_1|_Q=s^P_2|_Q\in \PP\mE|_Q$, $({g^P_1h^P_2\over g^P_2h^P_1})(Q)=1$. Suppose we fix some $P$ and consider general choice of $P'$. The condition $\phi(T^P_1|_P)=T^{P'}_1|_{P'}$, $\phi(T^P_2|_P)=T^{P'}_2|_{P'}$, $\phi(s^P_1|_P)=s^{P'}_1|_{P'}$ determines $\phi$. If further $\phi(s^P_2|_P)=s^{P'}_2|_{P'}$, one gets $({g^P_1h^P_2\over g^P_2h^P_1})(P)=({g^{P'}_1h^{P'}_2\over g^{P'}_2h^{P'}_1})(P')=m\in K^*$.

By Proposition \ref{genericity}, for any given $m\in K^*$, equations $$\begin{cases}
({g^P_1h^P_2\over g^P_2h^P_1})(P)=m\\
({g^P_1h^P_2\over g^P_2h^P_1})(Q)=1
\end{cases}$$ hold only for finitely many $P\in C$. Hence, Theorem \ref{genericity0} follows.
\end{proof}
\section{Rank Two Limit Linear Series}\label{definition}
Limit linear series was first defined in rank one by Eisenbud and Harris in \cite{EHL}. In higher rank, the corresponding theory was developed by Teixidor i Bigas in \cite{M5}. Recently in a preprint (\cite{Ohrk}) Osserman gave a stack structure for moduli of objects considered by Eisenbud, Harris and Teixidor (which we shall refer to as \tb{Eisenbud-Harris-Teixidor limit linear series}\rm, or \tb{EHT limit linear series}\rm\ for short) and proved a smoothing theorem (see Theorem \ref{smoothing}) which theoretically justifies the degeneration approach towards higher rank Brill-Noether Problems.

In this section, we recall some key definitions in \cite{Ohrk} and elaborate on some of the technical issues. In particular, we shall introduce the notion of (rank two) EHT-limit linear series over a chain of $N$ smooth projective curves, which we denote as $X_N$.
\begin{defn}
  Let $\mE$ be a vector bundle on $X_N$, with irreducible components $C_1,..,C_N$. We call $\vv{d}=(\deg(\mE|_{C_1}),...,\deg(\mE|_{C_N}))$ the \tb{multi-degree} of $\mE$.
\end{defn}

Let $d,d_1,...,d_N$ be integers satisfying the identity $\sum_j d_j-2b(N-1)=d$ with respect to some $b$.\footnote{For a discussion on the parameter $b$, see Section 6.2 in \cite{Ohrk}.} Denote $\mathcal{G}^k_{2,d_j}(C_j)$ to be the stack of rank two linear series of degree $d_j$, dimension $k$ over the component $C_j$. Suppose $C_j\cap C_{j+1}=P_{j+1}$. Denote $\mathcal{M}^2(P_j)$ to be the stack of rank two vector bundles over $P_j$. Further denote $\mathcal{P}^{k}_{2,d_{\bullet}}$ to be the product of $\mathcal{G}^k_{2,d_j}(C_j)$'s fibered over the $\mathcal{M}^2(P_j)$'s.

\begin{rem}\label{twist}
  Following the convention in \cite{Ohrk}, instead of taking the the fibered product along the natural restriction maps $\mathcal{G}^k_{2,d_j}(C_j)\to \mathcal{M}^2(P_j),\mathcal{G}^k_{2,d_j}(C_j)\to \mathcal{M}^2(P_{j+1})$, we first choose some appropriate twisting of $\mE_j$, $\mE_j''=\mE_j(e_{j,1}P_j+e_{j,2}P_{j+1})$ such that $\sum\deg(\mE_j'')=d$ and then take the restrictions.

  The motivation is the following: any rank two, degree $d$ vector bundle $\mE$ on $X_N$ can be thought of as a collection of rank two vector bundles $(\mE_j'')_{j=1}^N$ over the components together with some chosen gluing isomorphisms at the nodes. The sections of $\mE$ can also be thought of as sections of $\mE_j''$ glued at the nodes. Start with a datum in $\mathcal{P}^{k}_{2,d_{\bullet}}$, we have vector bundles $\mE_j$ over $C_j$ with degree $d_j$, where $\sum d_j-2b(N-1)=d$. By choosing different twisting at the nodes, we obtain rank two vector bundles of different multi-degrees over $X_N$. In this convention of taking the fibered product, we essentially fix a multi-degree and take gluing data for the vector bundle of this multi-degree.
\end{rem}
Following notations in Remark \ref{twist}, let $\vv{d}=(\deg(\mE''_1),...,\deg(\mE''_N))$. Denote $\mathcal{M}_{2,\vv{d}}(X_N)$ to be the moduli stack of rank two, multi-degree $\vv{d}$ vector bundles on $X_N$. It is easy to see that there is a forgetful map $\pi:\mathcal{P}^{k}_{2,d_{\bullet}}\to\mathcal{M}_{2,\vv{d}}(X_N)$ (see also \cite{Ohrk}, remarks following Notation 4.1.1).
\begin{rem}\label{p}
There is a simple description of geometric fibers of $\pi$. Let $x:\spec(F)\to\mathcal{M}_{2,\vv{d}}(X_N)$ represent a point of $|\mathcal{M}_{2,\vv{d}}(X_N)|$ corresponding to some vector bundle $\mE''$ such that $\mE''|_{C_j}=\mE''_j$. The fiber $\pi_x$ is representable by a product of Grassmannians: $\mathbf{G}:=\prod\Gr(k,\Gamma(\mE_j))$, where $\mE_1,...,\mE_N$ are vector bundles which differ from $\mE''_1,...,\mE''_N$ by the fixed twisting in the definition of $\mathcal{P}^k_{2,d_{\bullet}}$.
\end{rem}
We now give the definition of rank two limit linear series on a chain $X_N$ of smooth curves:
\begin{defn}\label{Def}
Let $((\mE_j,V_j)_{i=1}^N,(\phi_j)_{j=2}^N)$ be a $K$-valued point of $\mathcal{P}_{2,d_{\bullet}}^{k}(X_N)$, where $(\mE_j,V_j)$ is the corresponding point in $\mathcal{G}^k_{2,d_j}(C_j)$ and $\phi_j:\mE''_{j-1}|_{P_j}\stackrel{\sim}{\to}\mE''_j|_{P_j}$ (See Remark \ref{twist}). Then, $((\mE_j,V_j)_{j=1}^N,(\phi_j)_{j=2}^N)$ is a rank two, degree $d$, dimension $k$ \tb{Eisenbud-Harris-Teixidor limit linear series} (\textbf{EHT-limit linear series} in short) if:
\begin{enumerate}
  \item $H^0(\mE_j(-(b+1)P_j))=0,H^0(\mE_j(-(b+1)P_{j+1}))=0$, where $\sum d_j-2b(N-1)=d$.
  \item Let $b_{j,1}\ge...\ge b_{j,k}$, $a_{j+1,1}\le...\le a_{j+1,k}$ be the vanishing sequences of $V_j,V_{j+1}$ at $P_{j+1}$ resp. Then, $b_{j,i}+a_{j+1,i}\ge b$, for all $i$.
  \item For every $j$, there exist bases $\{s^{j,j}_1,...,s^{j,j}_k\}$, $\{s^{j,j+1}_1,...,s^{j,j+1}_k\}$ of $V_j$ such that $\ord_{P_j}(s^{j,j}_i)=a_{j,i}, \ord_{P_{j+1}}(s^{j,j+1}_i)=b_{j,i}$; and if $b_{j,i}+a_{j+1,i}=b$, $\phi_{j+1}(s^{j,j+1}_i|_{P_{j+1}})=s^{j+1,j+1}_i|_{P_{j+1}}$.
\end{enumerate}
We say $((\mE_j,V_j)_{j=1}^N,(\phi_j)_{j=2}^N)$ is \tb{refined}, if $b_{j,i}+a_{j+1,i}=b$ for all $i,j$.

We also refer to $d_{\bullet}$ as the \textbf{componentwise degree} of $((\mE_j,V_j)_{j=1}^N,(\phi_j)_{j=2}^N)$.
\end{defn}
\begin{rem}\label{glue}
  We clarify condition (3) in \ref{Def}: $s^{j,j}_P|_P$ (\tb{Notation 10}) is a point in $\PP\mE_j|_P$. Here, we abuse notation and refer to $s^{j,j}_i|_P$ (resp. $s^{j,j+1}_i|_P$) as some lift of the point to $\mE|_P$. Since $\s(s^{j,j}_1,...,s^{j,j}_k)=\s(a_1s^{j,j}_1,...,a_ks^{j,j}_k)=V_j$, $\fa a_1,...,a_k\in K^*$, this does not cause ambiguity.

  Note that instead of specifying one degree $d$ vector bundle on $X_N$, the data of a limit linear series give a collection of similar (in the sense of Remark \ref{similar}) degree $d$ vector bundles $\mE'$ each of which is determined as follows: for every $j$, $\mE'|_{C_j}=\mE_j(e_{j,1}P_j+e_{j,2}P_{j+1})$ for some $e_{j,1},e_{j,2}\in\ZZ$ such that $e_{j,2}+e_{j+1,1}=-b$ and $\mE'|_{C_j}$ is glued to $\mE'|_{C_{j+1}}$ via an isomorphism induced by $\phi_{j+1}$. The last part makes sense, since $(\phi_j)_{j=2}^N$ was given for some $\mE''$ of a particular multi-degree $\vec{d}$ (see Remark \ref{twist}); any similar vector bundle of a different multi-degree $d'_{\bullet}$ can be obtained from $\mE''$ by twisting by a line bundle $\mO_{d''_{\bullet},d'_{\bullet}}$ on $X_N$, which restricts to $\mO(-n_jP_j+n_{j+1}P_{j+1})$ on $C_j$, where $n_j$ are integers such that $\det(\mE''|_{C_j})^{-1}\ot\det(\mE'|_{C_j})=\mO(-2n_jP_j+2n_{j+1}P_{j+1})$. (We shall refer to such an $\mE'$ as \tb{a vector bundle induced by the limit linear series $((\mE_j,V_j),(\phi_j))$}\rm.)
  \end{rem}
Following \cite{Ohrk} 4.2.1, there exists a stack $\GG^{k,\text{EHT}}_{2,d,d_{\bullet}}(X_N)$ of rank two, degree $d$, dimension $k$ EHT-limit linear series, which is a locally closed substack of $\mathcal{P}_{2,d_{\bullet}}^{k}(X_N)$.

For our context, we now define the the moduli of limit linear series with fixed determinant:
\begin{defn}\label{fixdet}
  Denote $\mathcal{M}_{2,\vv{d},\mL}(X_N)$ to be the moduli stack of pairs $(\mE,\psi)$, where $\mE$ is a rank two, multi-degree $\vv{d}$ vector bundle on $X_N$ and $\psi:\det(\mE)\to\mL$ is an isomorphism (after appropriate twisting of $\det(\mE)$ at the nodes, see Definition \ref{fixdef}). Then, the moduli stack of rank two, componentwise degree $d_{\bullet}$, dimension $k$ EHT-limit linear series with fixed determinant $\mL$ is the (2-)fibered product
  $$\GG^{k,\text{EHT}}_{2,\mL,d_{\bullet}}(X_N):=\GG^{k,\text{EHT}}_{2,d,d_{\bullet}}(X_N)\times_{\mathcal{M}_{2,\vv{d}}(X_N)}\mathcal{M}_{2,\vv{d},\mL}(X_N).$$
\end{defn}
\begin{rem}
  In other words, besides the usual data of an EHT-limit linear series, we further specify an isomorphism from the determinant to the prescribed line bundle. Automorphisms of objects in this stack are automorphisms of limit linear series that are compatible with this isomorphism. Consequently, the stack dimension in the fixed determinant case agrees with the dimension of the corresponding coarse moduli space.
\end{rem}
Lastly, we recall the following definitions from \cite{Osp}. They are useful in our construction and relevant to the smoothing theorem (Theorem \ref{smoothing}) we shall apply:
\begin{defn}
Let $C$ be a smooth projective curve over $K$. Let $(\mE,V)$ be a pair, where $\mE$ is a rank-$r$ vector bundle on $C$, and $V$ is a $k$-dimensional space of sections. Given $P,Q\in C$, denote $\van_P(V)=(a_1\le...\le a_k)$, and $\van_Q(V)=(b_1\ge...\ge b_k)$. We say a basis $\{s_i\}$ of $V$ is $\mathbf{(P,Q)}$\textbf{-adapted}, if $\ord_P(s_i)=a_i,\ord_Q(a_i)=b_i$, for all $i$. If $V$ admits an adapted basis, we say $(\mE,V)$ is $\mathbf{(P,Q)}$\textbf{-adaptable}.
\end{defn}
\begin{defn}
  Let $X_N$ be a chain of smooth projective curves $C_1,...,C_N$, with $P_j,Q_j\in C_j$, and $Q_j$ glued to $P_{j+1}$. A refined Eisenbud-Harris-Teixidor limit linear series, $((\mE_j,V_j),(\phi_j))$ is said to be \textbf{chain adaptable}, if the pair $(\mE_j,V_j)$ is $(P_j,Q_j)$-adaptable, for $j=2,...,N-1$.
\end{defn}
\section{Restrictions on Vector Bundles from Prescribed Vanishing Data}
In practice, we shall construct families of rank two limit linear series on a chain $X_N$ by prescribing vanishing sequences at all nodal points $P_j$. These vanishing data play a central role in our analysis. In this section, we discuss various restrictions on the underlying vector bundles resulting from prescribing vanishing data at $P_j$. They turn out to be useful in moduli counting.
\begin{si}\label{sit2}
Throughout this section, we implicitly assume the following:
\begin{enumerate}
  \item $C$ is an elliptic curve and $P,Q$ are two general points on $C$.
  \item Let $V\subset\Gamma(\mE)$ be a $k$-dimensional subspace of a rank two vector bundles $\mE$, whose vanishing sequences at $P,Q$ are $(a_1\le...\le a_k)$ and $(b_1\ge...\ge b_k)$.
  \item $u(\mE)\le{1\over 2}$ (see \tb{Convention}).
\end{enumerate}
\end{si}
\begin{rem}
  Notice that the third assumption here is consistent with the third assumption in \ref{s1}.
\end{rem}
Under these assumptions, we first do a simple calculation.
\begin{lem}\label{0}
  Given $C,\mE,V$ and $P,Q\in C$ and the vanishing sequences $(a_k),(b_k)$ as in \ref{sit2}, for every $t\in\{1,...,k\}$, $\dim(V(-a_tP-b_tQ))\ge|\{i|a_i\ge a_t,b_i\ge b_t\}|$. In particular, $\exists s\in V$ such that $\ord_P(s)\ge a_t,\ord_Q(s)\ge b_t$.
\end{lem}
\begin{proof}
  Clearly, $\dim V(-a_tP)=|\{i|a_i\ge a_t\}|$, $\dim V(-b_tQ)=|\{i|b_i\ge b_t\}|$. Since $(a_i)$ is increasing and $(b_i)$ is decreasing, $$|\{i|a_i\ge a_t\}|+|\{i|b_i\ge b_t\}|=k+|\{i|a_i\ge a_t,b_i\ge b_t\}|.$$
  Thus,
  \begin{equation}\begin{split}
    \dim V(-a_tP-b_tQ)=\dim(V(-a_tP)\cap V(-b_tQ))&\ge (k+|\{i|a_i\ge a_t,b_i\ge b_t\}|)-k\\&=|\{i|a_i\ge a_t,b_i\ge b_t\}|.
  \end{split}\end{equation}
\end{proof}
\begin{lem}\label{2}
  Suppose there exists $i_1,i_2$ such that $a_{i_1}+b_{i_1}=a_{i_2}+b_{i_2}=d$. Then, $\deg(\mE)\ge 2d$. If $\deg(\mE)=2d$, then $\mE\cong \mO(a_{i_1}P+b_{i_1}Q)\op \mO(a_{i_2}P+b_{i_2}Q)$.
\end{lem}
\begin{proof}
By Lemma \ref{0}, there are two sections $s_1,s_2\in V$ whose vanishing orders at $P,Q$ sum up to $d$ or higher. Suppose $\mE$ is indecomposable, since no indecomposable vector bundle of negative degree over an elliptic curve has a section, $\deg(\mE)\ge 2d$. If it is decomposable, if $s_1,s_2$ are both section of one summand of $\mE$, it must have degree $\ge d+1$; otherwise, both summands have degrees $\ge d$. Given $u(\mE)\le{1\over2}$ (see \tb{Convention}), $\deg(\mE)\ge 2d$ always holds.

If $\deg(\mE)=2d$, by the theory of Atiyah bundles (See \cite{Atiyah}), an indecomposable, degree $2d$ vector bundle can have at most one section whose vanishing orders at $P,Q$ sum up to $d$. Hence, $\mE$ must be decomposable. But then the degrees of both summands must be exactly $d$. So, $\mE$ is semi-stable and of the suggested form.
\end{proof}
\begin{lem}\label{3}
  Suppose $\exists \ell$ such that $a_{\ell}+b_{\ell}=d+1$; for some $i$, $a_i=a_{i+1}$ and $a_i+b_i=a_{i+1}+b_{i+1}=d$. Then, $\deg(\mE)\ge 2d+1$. If $\deg(\mE)=2d+1$, then $\mE\cong \mO(a_{\ell}P+b_{\ell}Q)\op \mO(a_iP+b_iQ)$.
\end{lem}
\begin{proof}
  If $\mE$ is indecomposable, since at least one section has vanishing orders at $P,Q$ sum up to $d+1$ or larger, $\deg(\mE)\ge 2d+2$. If it is decomposable, since one summand has degree $\ge d+1$, by the third assumption in \ref{sit2}, $\deg(\mE)\ge 2d+1$.

  If $\deg(\mE)=2d+1$, since an indecomposable vector bundle of this degree cannot have a section whose vanishing orders at $P,Q$ sum up to $d+1$, $\mE$ must be decomposable. By assumption (3) in \ref{sit2}, the degrees of the two summands must be $d+1$ and $d$ resp. By \ref{0}, $\mE$ has two sections whose vanishing orders at $P,Q$ are at least $a_i,b_i$ resp. Hence, $\mE$ must be of the suggested form.
\end{proof}
\section{Vanishing Sequences and Existence of Adapted Bases}\label{DB}
As mentioned earlier, we shall specify families of rank two limit linear series on a chain $X_N$ by prescribing vanishing sequences at all the nodal points. A first thing to check is the legitimacy of the prescribed vanishing sequences.

Another related problem is the existence of adapted basis of $V_j$. The existence of such basis simplifies moduli counting procedure: if $V_j$ always admits a adapted basis, we can consider possible choices for $V_j$ by considering all possible choices of such basis. Hence, we first establish some criterion for their existence. More importantly, after working out moduli counts for limit linear series, we need to apply the smoothing theorem (Theorem \ref{smoothing}) to derive corresponding moduli counts for linear series on smooth curves. The smoothing theorem only applies to loci of chain-adaptable limit linear series.

We start with a result for semi-stable, rank two vector bundles over an elliptic curve:
\begin{lem}\label{one}
  Let $C$ be elliptic, $P,Q$ be two general points on $C$. Let $(a_1\le...\le a_k),(b_1\ge...\ge b_k)$ be two non-negative integer vectors such that $\fa i,d\ge a_i+b_i\ge d-1$ and $\exists ! i_1<i_2$ such that $a_{i_1}+b_{i_1}=a_{i_2}+b_{i_2}=d$. Suppose every integer appears in $(a_i)$ or $(b_i)$ at most twice, and $\not\exists i\neq i_1,i_2$ such that $a_i=a_{i_1}$ and $b_i=b_{i_2}$.

  Then, the vector bundle $\mE=\mO(a_{i_1}P+(d-a_{i_1})Q)\op \mO(a_{i_2}P+(d-a_{i_2})Q)$ ($d\ge a_k$) admits a $k$-dimensional subspace $V$ of sections such that $\van_P(V)=(a_1,...,a_k), \van_Q(V)=(b_1,...,b_k)$.

  Moreover, any such $V$ admits a $(P,Q)$-adapted basis.
\end{lem}

\begin{proof}
We shall prove the existence of one such $V$ and the existence of an adapted basis for every $V$ possible, by showing that the vanishing sequences force $V$ to have some sections $s_1,...,s_k$ with vanishing orders $a_i,b_i$ at $P,Q$ resp., and the vector bundle $\mE$ has such sections as required.

Under the assumptions here, for any $i$, the pair $(a_i,b_i)$ is in one of the following situation: (1) $a_i+b_i=d-1$, and both $a_i,b_i$ are non-repeated vanishing orders; (2) $a_i+b_i=d-1$, and exactly one of $a_i,b_i$ is a repeated vanishing orders, which is repeated once; (3) $a_i+b_i=d-1$, and both $a_i,b_i$ are repeated exactly once in the following way: $a_i=a_{i+1},b_i=b_{i+1}$ (or $a_i=a_{i-1},b_i=b_{i-1}$); (4) $a_i+b_i=d$.

Hereafter, for any $i$, we shall specify one section $s_i$ in $V$ vanishing to order $a_i$ at $P$, and $b_i$ at $Q$.

First, consider situation (4), i.e. $i=i_1$ or $i_2$. By lemma \ref{0}, we have $\dim V(-a_iP-b_iQ)\ge2$ if $a_{i_1}=a_{i_2}$, and $\dim V(-a_iP-b_iQ)\ge1$ otherwise. In any case, $V$ always contains a two subspace $V'$ spanned by two sections whose vanishing orders at $P,Q$ are $a_{i_1},d-a_{i_1}$ and $a_{i_2},d-a_{i_2}$ resp. We take $s_{i_1},s_{i_2}$ to be two such sections. Note also in the case $a_{i_1}=a_{i_2}$, situation (2) cannot occur.

Now suppose $a_i+b_i=d-1$. By direct computation we know $\Gamma(\mE)$ contains a unique two-dimensional subspace of sections vanishing to orders precisely $a_i,b_i$ in situation (1) and (3); it contains a unique two-dimensional subspace of sections vanishing to orders at least $a_i,b_i$ which contains a canonical section in situation (2). Correspondingly, in situation (1)-(3), $V$ must contain the following: in (1), $V$ must contain an $s_i$ vanishing to order precisely $a_i,b_i$ at $P,Q$ resp.; in (2), $V$ must contain some $s_i$ which together with one canonical section span $\Gamma(\mE(-a_iP-b_iQ))$; in (3), $V$ must contain $\Gamma(\mE(-a_iP-b_iQ))$, i.e. some $s_{i-1},s_i$ (or $s_i,s_{i+1}$) which span $\Gamma(\mE(-a_iP-b_iQ))$.

Since in all four situations such $s_i$ can be found, and they together form an adapted basis of $V$, we have proved the lemma.
\end{proof}

Next, we establish a similar result for certain unstable vector bundles on an elliptic curve.
\begin{lem}\label{one2}
  Let $C$ be elliptic curve, $P,Q$ be two general points on $C$. Let $(a_1\le...\le a_k),(b_1\ge...\ge b_k)$ be two non-negative integer vectors of length $k$ satisfying the conditions in lemma \ref{3} (with $a_j,a_{\ell}$ as in \ref{3}). Suppose the following conditions hold:
  \begin{enumerate}
    \item $\fa i$, $d+1\ge a_i+b_i\ge d-1$.
    \item For any $c\in \ZZ$, $|\{i|a_i=c\}|\le 2$, $|\{i|b_i=c\}|\le 2$.
    \item For any $(a,c)\in\ZZ^2$ such that $(a,c)\neq(a_j,b_j)$, $|\{i|(a_i,b_i)=(a,c)\}|\le 1$.
    \item For any $(a,d-a)\in\ZZ^2$ such that $(a,d-a)\neq(a_j,b_j)$, $|\{i|(a_i,b_i)\ge(a,d-a)\}|\le 1$.
    \item There does not exist some $i$ such that $a_i+b_i=a_{i+1}+b_{i+1}=d-1$, $a_{i+1}=a_i+1$ and $a_i,a_{i+1},b_i,b_{i+1}$ are all non-repeated.
  \end{enumerate}

Then, the vector bundle $\mE=\mO(a_{\ell}P+(d+1-a_{\ell})Q)\op \mO(a_jP+(d-a_j)Q)$ on $C$ ($d\ge a_k$) admits a $k$-dimensional subspace $V$ of sections such that $\van_P(V)=(a_1,...,a_k), \van_Q(V)=(b_1,...,b_k)$.

Moreover, any such $V$ admits an adapted basis realizing the two vanishing sequences at $P,Q$ simultaneously.
\end{lem}
\begin{proof}
By condition (1), we break $\{1,..,k\}$ into three subsets:
$$\{1,...,k\}=\{\ell\}\cup\{i|a_i+b_i=d\}\cup\{i|a_i+b_i=d-1\}.$$

Similar to the previous lemma, we prove the existence of such $V$ and the existence of an adapted basis for every $V$ possible, by showing that the vanishing sequences force any $V$ possible to have some $s_1,...,s_k$ with vanishing orders $a_i,b_i$ at $P,Q$ resp., and the vector bundle $E$ actually has such sections.

\noindent\textbf{Case 1: $\mathbf{i=\ell}$.} Since $a_{\ell}+b_{\ell}=d+1$, any $V$ must contain the canonical section of the summand $\mO(a_{\ell}P+(d+1-a_{\ell})Q)$.

\noindent\textbf{Case 2: $\mathbf{a_i+b_i=d}$.} For any $i$ such that $a_i+b_i=d$, we have two situations. First, $i=j$ or $j+1$ ($j$ as given in \ref{3}). In this case, by condition (2) and the assumptions made in \ref{3}, any $k$-dimensional sub-space $V$ with the given vanishing sequences must contain $\Gamma(\mE(-a_iP-b_iQ))$, since by \ref{0} $\dim V(-a_iP-b_iQ)\ge2$ and $\dim\Gamma(\mE(-a_iP-b_iQ))=2$; in particular, it must contain two linearly independent sections $s_j,s_{j+1}$ such that $\ord_P(s)=\ord_P(s')=a_j=a_{j+1},\ord_Q(s)=\ord_Q(s')=b_j=b_{j+1}$. Second, $i\neq j,j+1$. By condition (4), $a_i\neq a_{\ell},b_i\neq b_{\ell}$. In this case, we know $\mE$ has precisely one section (up to scalar multiplication), $s_i$, vanishing to orders $a_i,b_i$ at $P,Q$ resp., and no sections vanishing to order strictly higher at one of the two points. Thus, any feasible $V$ must contain $s_i$, since $\dim V(-a_iP-b_iQ)\ge1$.

\noindent\textbf{Case 3: $\mathbf{a_i+b_i=d-1}$.} If one of $a_i,b_i$ is repeated, by condition (3) and \ref{0}, $\dim V(-a_iP-b_iQ)\ge2$. In particular, there is a section $s$ vanishing to order $a_i,b_i$ at $P,Q$ resp and we are done. So we assume $a_i,b_i$ are both non-repeated.

Suppose $V$ does not contain a section vanishing precisely to orders $a_i,b_i$ at $P,Q$ resp. Then, it must contain a section $s$ such that either (a) $\ord_P(s)\ge a_i+1,\ord_Q(s)=b_i$, or (b) $\ord_P(s)=a_i,\ord_Q(s)\ge b_i+1$ holds. (Note that one cannot have $a_{\ell}>a_i,b_{\ell}>b_i$ by the monotonicity of $(a_i)$ and $(b_i)$.) Suppose $V$ contains $s$ satisfying (a). In particular, $\ord_P(s)$ is one entry in $(a_i)$. Given $\mE$, if $\ord_P(s)=a_i+2$, then $\ord_P(s)=a_{\ell}, \ord_Q(s)=b_{\ell}=b_i$. This violates the assumption that $a_i,b_i$ are both non-repeated. If $\ord_P(s)=a_i+1$, since $a_i+b_i=d-1$ and $(b_i)$ is non-increasing, by condition (3) and (5), one must have $a_{i+1}=a_i+1$ and $b_{i+1}=b_i-1$ and $b_i-1$ is repeated. But then $V$ contains a section $s$ vanishing to orders precisely $a_{i+1},b_{i+1}$ at $P,Q$ resp., which is not a section of the destabilizing summand of $\mE$. This implies that $a_{i+1}$ is repeated, but in this case $a_i<a_{i+1}<a_{i+2}$ and we get contradiction. So $V$ must contain some $s$ vanishing to orders $a_i,b_i$ at $P,Q$. By symmetry, same holds if we start by assuming (b).

Since $\dim\Gamma(\mE(-a_iP-b_iQ))=3$ and is spanned by two sections of $\mO(a_{\ell}P+b_{\ell}Q)$, together with another section vanishing to orders $a_i,b_i$ at $P,Q$. Therefore, such $s_i$ exists.

Thus, any $V$ with the given vanishing sequences at $P,Q$ resp. must contain such sections $s_1,...,s_k$, which form an adapted basis. Moreover, $\mE$ has such sections. Hence, we are done.
\end{proof}
\section{The Fixed Determinant Condition}\label{can}
In a broader context, the Bertram-Feinberg-Mukai conjecture is one evidence of the following phenomenon: the locus of rank two linear series where the underlying vector bundles have some fixed (special) determinant $L$, tends to have an expected dimension that is larger than the expected dimension for a general rank two Brill-Noether locus (i.e. where the determinant is allowed to vary). In proving existence results for fixed determinant cases using degeneration, we need to discuss what it means for a limit linear series $((\mE_j,V_j)_{j=1}^N,(\phi_j)_{j=2}^N)$ to have a fixed determinant.

A well-known important fact is the following: let $\pi:\mathcal{X}\to B$ be a family\footnote{Technically, a flat and proper family such that fibers are 1-dimensional, geometrically connected and geometrically reduced.} of curves over the spec of some DVR whose special fiber is a reducible nodal curve of compact type. By analyzing $\Pic^d(\mathcal{X}/B)$, we know that any degree $d$ line bundle $\mL_{\eta}$ over the generic fiber extends to the whole family, but the extension is not unique. More precisely, if $\mL$ is an extension of $\mL_{\eta}$, then $\mL'$ is also an extension, where $\mL'$ and $\mL$ differ in the following way: let $C_i,C_j$ be two components of $X_0$ meeting at $P$, $\mL'|_{C_i}=\mL|_{C_i}(aP)$, $\mL'|_{C_j}=\mL|_{C_j}(-aP)$; this turns out to be the only flexibility in extending $\mL_{\eta}$. (See Chapter 5.C in \cite{MOCURVE}.) Therefore, when considering vector bundles on a reducible nodal curve of compact type, we adopt the following definition:
\begin{defn}\label{fixdef}
  Let $X$ be a reducible nodal curve of compact type with components $C_1,...,C_n$, and $\mE$ is a rank-$r$ vector bundle on $X$. If $C_i,C_j$ meet at one point, denote $P_{j,i}=P_{i,j}=C_i\cap C_j$. Fix a line bundle $\mL$ on $X$. We say $\mE$ \tb{satisfies the fixed determinant condition with respect to} $\mL$\rm\ if the following conditions hold:
  \begin{enumerate}
   \item For every $j$, $\Lambda^r\mE|_{C_j}\cong\mL|_{C_j}(\sum_{s} a_{s,j}P_{s,j})$, where $s$ runs through all indices such that $C_s\cap C_j\neq\emptyset$ and $a_{s,j}\in\ZZ$.
   \item For all $j_1,j_2$ such that $C_{j_1}\cap C_{j_2}\neq\emptyset$, $a_{j_1,j_2}=-a_{j_2,j_1}$.
  \end{enumerate}
\end{defn}
In cases where $X$ is a chain, this condition has a simple characterization:
\begin{lem}\label{fixed}
Let $\mL$ be a line bundle on a chain $X_N$ with multi-degree $(w_1,...,w_N)$. Denote $C_1,...,C_N$ to be the components of $X_N$ and suppose $C_j\cap C_{j+1}=P_{j+1}$. A rank-$r$ vector bundle $\mE$ on $X_N$ with multi-degree $(d_1,...,d_N)$ satisfies the fixed determinant condition with respect to $\mL$ if and only if for every $j$:
$$\Lambda^r\mE|_{C_j}\cong \mL|_{C_j}(-\sum_{t=0}^{j-1}(w_t-d_t)P_j+\sum_{t=0}^j(d_t-w_t)P_{j+1}).$$
($w_0=d_0=0$ and $P_1,P_{N+1}$ are some fixed general points on $C_1,C_N$ respectively.)
\end{lem}
\begin{proof}
This is a straightforward interpretation of Definition \ref{fixdef} and the proof is trivial.
\end{proof}
Next, we apply this lemma to rank two limit linear series $((\mE_j,V_j),(\phi_j))$ over $X_N$, where $\deg(\mE_j)=d_j$ and $\sum d_j-2b(N-1)=d$. By Remark \ref{glue}, a limit linear series gives a collection of similar vector bundles. Notice that if one of these vector bundles satisfies the fixed determinant condition with respect to $\mL$, then all vector bundles induced by the limit linear series satisfy the same condition. Thus, we make the following definition:\footnote{This works well for limit linear series of arbitrary rank, but here we shall focus on the rank two case.}
\begin{defn}\label{llsfix}
  A rank two limit linear series $((\mE_j,V_j),(\phi_j))$ on a chain $X_N$ is said to satisfy the fixed determinant condition with respect to line bundle $L$ if one (and equivalently, all) of the vector bundles $\mE'$ induced by $((\mE_j,V_j),(\phi_j))$ satisfies this condition.
\end{defn}
One further gets:
\begin{lem}
Let $L$ be a line bundle of multi-degree $(w_1,...,w_N)$ on a chain $X_N$ with components $C_1,...,C_N$ such that $C_j\cap C_{j+1}=P_{j+1}$. A rank two limit linear series $((\mE_j,V_j),(\phi_j))$ on $X_N$, where $\deg(\mE_j)=d_j$ and $\sum d_j-2b(N-1)=d$, satisfies the fixed determinant condition with respect to $\mL$ if and only if for every $j$:
$$\Lambda^2\mE_j\cong
\mL|_{C_j}((2b(j-1)+\sum_{t=0}^{j-1}(w_t-d_t))P_j+(\sum_{t=0}^j(d_t-w_t)-2b)P_{j+1}).\footnote{Again, $P_1,P_{N+1}$ are general points on $C_1,C_N$ respectively; $d_0=w_0=0$.}$$
\end{lem}
\begin{proof}
  Suppose $\mE'|_{C_j}=\mE_j(-a_{j,1}P_j-a_{j,2}P_{j+1})$ for all $j$. As discussed above,\\
  $a_{j,2}+a_{j+1,1}=b$ for every $j$.

  By comparing the degrees of $\mE_1$ and $\mL|_{C_1}$, one gets $\Lambda^2\mE_1\cong \mL|_{C_1}((d_1-w_1)P_2)$.

  By definition \ref{fixdef}, $\Lambda^2\mE'|_{C_1}\cong \mL|_{C_1}((d_1-w_1-2a_{1,2})P_2)$ implies
  $$\Lambda^2\mE'|_{C_2}\cong \mL|_{C_2}((2a_{1,2}-d_1+w_1)P_2+eP_3),\text{ for some integer }e.$$
  Thus,
  $$\begin{aligned}
    \Lambda^2\mE_2&\cong \mL|_{C_2}((2a_{2,1}+2a_{1,2}-d_1+w_1)P_2+(e+2a_{2,2})P_3)\\
  &=\mL|_{C_2}((2b+w_1-d_1)P_2+(e+2a_{2,2})P_3).
  \end{aligned}$$
  By degree comparison, one can conclude that $e+2a_{2,2}=\sum_{t=1}^2(d_t-w_t)+2b$.

  Since the same calculation applies for every $j$, by induction we are done. Note that this condition is independent of the choice of $\mE'$ among the vector bundles induced by the limit linear series.
\end{proof}
For our application, we specialize to the canonical determinant case. Denote $\omega$ to be the dualizing sheaf on a chain $X_N$ of smooth projective curves and suppose the genus of component $C_j$ is $g_j$ with $\sum g_j=g$. Based on the remark and Theorem 5.2.3 in \cite{BCNRD}, $\omega$ is a line bundle determined as follows:
$$\omega|_{C_1}=\omega_1(P_2),\omega|_{C_j}=\omega_j(P_j+P_{j+1}),\omega|_{C_N}=\omega_N(P_N),$$
where $\omega_j$ be the canonical sheaf on $C_j$.

In particular, for the canonical determinant, one can re-write \ref{fixed} as a condition involving $\omega_j$:
\begin{cor}
A rank two limit linear series $((\mE_j,V_j),(\phi_j))$ on the chain $X_N$ satisfies the canonical determinant condition if and only if for every $j$:
$$\det(\mE_j)=\omega_j((\sum_{t=1}^{j-1}(2g_t)-\sum_{t=1}^{j-1}d_t+(j-1)\cdot 2b)P_j+(2+\sum_{t=1}^jd_t-(j-1)\cdot 2b-\sum_{t=1}^j(2g_t))P_{j+1}).$$
\end{cor}
 Throughout our construction, all $C_j$ shall be elliptic and $\deg(\mE_j)\in\{2g-1,2g-2,2g-3\}$, with a simple pattern. In this case, the canonical determinant further simplifies as follows:
\begin{cor}\label{candet}
Let $X_g$ be a chain of $g$ elliptic curves. Let $((\mE_j,V_j)_{j=1}^g,(\phi_j)_{j=2}^g)$ be a rank two, degree $2g-2$, limit linear series on $X_g$ such that $\deg(\mE_j)=d_j$ and $\sum d_j-2(g-1)^2=2g-2$ (i.e. $b=g-1$). If further the odd $d_j$'s are $$d_{j_t}=\begin{cases}2g-1 \text{ when }t\text{ is odd}\\ 2g-3 \text{ when }t\text{ is even}\end{cases} (t=1,...,2T).$$

Then, it satisfies the fixed determinant condition with respect to $\omega$ if and only if:
$$\det(\mE_j)=\begin{cases}
\mO((2j-3)P_j+(d_j-(2j-3))P_{j+1}) \text{\ \ \ \  }j_{2s-1}<j\le j_{2s}\\
\mO((2j-2)P_j+(d_j-(2j-2))P_{j+1}) \text{\ \ \ \  otherwise}.
\end{cases}$$
\end{cor}
\section{General Framework and Terminologies}
We are now ready to describe our general framework.

Hereafter, denote $X_g$ to be a chain of $g$ elliptic curves whose irreducible components are $C_1,...,C_g$ and $P_2,...,P_g$ to be the nodes of $X$. Moreover, we specify $P_1,P_{g+1}$ to be smooth points on $X_1,X_g$ respectively. We always make the assumption that $P_j,P_{j+1}$ are general points in $C_j$. Also denote $\omega_g$ to be the dualizing sheaf on $X_g$.

We shall consider various moduli stacks of rank two limit linear series with canonical determinant, fixed componentwise-degree and prescribed vanishing sequences at $P_2,...,P_g$. It is shown in \cite{Ohrk} (Proposition 4.2.4) that the moduli stack of rank two refined Eisenbud-Harris-Teixidor limit linear series with canonical determinant and componentwise-degree $d_{\bullet}$, $\mathcal{G}^{k,\text{EHT,ref}}_{2,\omega_g,d_{\bullet}}(X_g)$, is a disjoint union of substacks, each of which corresponds to a set of fixed vanishing sequences at $P_2,...,P_g$, presented as a $k\times(2g-2)$ matrix $a^{\Gamma}=[\van_{P2}(V_1),...,\van_{P_j}(V_j),\van_{P_j}(V_{j+1}),...,\van_{P_g}(V_g)]$:
$$\mathcal{G}^{k,\text{EHT,ref}}_{2,\omega_g,d_{\bullet}}(X_g)=\coprod \mathcal{G}^{k,\text{EHT}}_{2,\omega_g,d_{\bullet},a^{\Gamma}}(X_g).\footnote{Here, $\Gamma$ refers to the dual graph of $X_g$, which does not get into play in this paper.}$$
\addtocontents{toc}{\SkipTocEntry}
\subsection*{Standard Vanishing Conditions}
As mentioned in the Introduction, the existence portion of the BFM-conjecture was partially verified by Teixidor-i-Bigas in \cite{M1} for all $g,k$ such that $L(g,k)\ge0$ (\tb{Notation 2}). In this paper, we focus on the cases where $g,k$ satisfy the conditions $\rr\ge 0,L(g,k)<0$. Note that the smallest $g,k$ for which these conditions hold is $g=6$, $k=5$.

Clearly, these vanishing conditions $a^{\Gamma}$ play a central role in our construction. Nevertheless, one may not want to consider arbitrary feasible vanishing conditions, since counting dimensions of corresponding moduli stacks may be complicated. The next definition specifies the kind of vanishing conditions considered in this paper, which we shall refer to as \it{standard vanishing conditions}.\rm
\begin{defn}\label{stan}
  Fix positive integers $g$ and $k$. Fix also an integer vector $(d_1,...,d_g)$ of length $g$ satisfying the following conditions:
  \begin{itemize}
    \item $d_j\in\{2g-3,2g-2,2g-1\}$;
    \item denote $d_{j_t}$ ($t=1,...,2T$) to be the odd $d_j$'s, then $d_{j_{2s-1}}=2g-1$ and $d_{j_{2s}}=2g-3$, for all $s$, with $j_1=2$;
    \item $d_1=d_g=2g-2$.
  \end{itemize}
  Let $(\van_{j1},\van_{j2})_{j=1}^g$ be a set of $2g$ non-negative integer vectors of length $k$ satisfying the following set of conditions:
\begin{enumerate}
  \item for $j=1,...,g-1$ and $i=1,...,k$, $b_{i,j}+a_{i,j+1}=g-1$;
  \item when $d_j=2(g-1)$, $(\van_{j1}),(\van_{j2})$ satisfy conditions in Lemma $\ref{one}$ with $d=g-1$;
  \item when $d_j=2g-1$ (resp. $2g-3$), $(\van_{j1}),(\van_{j2})$ satisfy conditions in Lemma $\ref{one2}$ with $d=g-1$ (resp. $d=g-2$);
  \item $a_{k,2}+b_{k,2}=g-1$;
  \item $a_{k,j_t}\le k_1+j_t-1$ and if $a_{k,j_t}=k_1+j_t-2$, then either $t$ is even or $a_{k,j_t}+b_{k,j_t}=\frac{1}{2}(d_{j_t}-1)$; if $a_{k,j_t}=k_1+j_t-1$, then $t$ is odd and $a_{k-1,j_t}<a_{k,j_t}$;
  \item If $b_{k,j_t}=g-1-k_1-j_t$, then $a_{k,j}+b_{k,j}=g-2$ for all $j:j_t<j<j_{t+1}$.
\end{enumerate}
We shall say $(\van_{j1},\van_{j2})_{j=1}^g$ is a $\mathbf{(g,k)}$-\textbf{standard} vanishing condition at $P_1,...,P_{g+1}$.
\end{defn}
\begin{defn}
  We shall also say a $k\times (2g-2)$ matrix $a^{\Gamma}$ is $(g,k)$-standard, if the columns of $\begin{bmatrix} a(k)&a^{\Gamma}&a(k)^{\rev}\end{bmatrix}$ (see \tb{Notation 4,12}) form a $(g,k)$-standard vanishing condition.
\end{defn}
One motivation for using standard vanishing conditions is the following: under mild assumptions, for all limit linear series $((\mE_j,V_j),(\phi_j))$ on $X_g$ such that $(\van_{P_j}(V_j),\van_{P_{j+1}}(V_j))$ is $(g,k)$-standard, the bundles $\mE_j$ are fixed:
\begin{lem}\label{bunfix}
Suppose $\begin{bmatrix}a(k)&a^{\Gamma}&a(k)^{\rev}\end{bmatrix}=(\van_{j1},\van_{j2})_{j=1}^g$ is $(g,k)$-standard. Let $((\mE_j,V_j),(\phi_j))$ be a rank two, dimension $k$, degree $2g-2$ limit linear series whose componentwise degree is $d_{\bullet}=(d_1,...,d_g)$ and $\van_{P_j}(V_j)=\van_{j1},\van_{P_{j+1}}(V_j)=\van_{j2}$. If $u(\mE_j)\le{1\over2}$ (see \tb{Convention}), then $\mE_j$ is determined to be a decomposable vector bundle. In this case, $b=g-1$.
\end{lem}
\begin{proof}
By definition, $\deg(\mE_j)=d_j$ for all $j$. By Lemma \ref{2}, \ref{3}, all $\mE_j$ are decomposable and fixed. Since $\sum d_j=2g(g-1)$ and $\sum d_j-2b(g-1)=2g-2$, one gets $b=g-1$.
\end{proof}
\begin{rem}\label{uopenr}
By Lemma \ref{bunfix}, the vector bundles $\mE_1,..,\mE_g$ are the same for any $K$-valued object of $\mathcal{G}^{k,\text{EHT}}_{2,\omega_g,d_{\bullet},a^{\Gamma}}(X_g)$. We shall refer to $\mE_1,...,\mE_g$ as the \textbf{underlying vector bundles} of $\mathcal{G}^{k,\text{EHT}}_{2,\omega_g,d_{\bullet},a^{\Gamma}}(X_g)$p.
\end{rem}

Since our main goal is the BFM-conjecture, in practice we shall restrict ourselves to standard vanishing conditions which determine $\mE_1,...,\mE_g$ such that the induced vector bundle on $X_g$ satisfies the canonical determinant condition.

Another motivation for using standard vanishing condition $a^{\Gamma}$ is that all objects in $\mathcal{G}^{k,\text{EHT}}_{2,\omega_g,d_{\bullet},a^{\Gamma}}(X_g)$ are chain-adaptable (see Section \ref{definition}). This follows directly from the definition of $(g,k)$-standardness.
\addtocontents{toc}{\SkipTocEntry}
\subsection*{Adapted Bases and Configurations of Sections at a Point}
Following the previous section, when $a^{\Gamma}$ is $(g,k)$-standard, all objects in $\mathcal{G}^{k,\text{EHT}}_{2,\omega_g,d_{\bullet},a^{\Gamma}}(X_g)$ have fixed underlying vector bundles. Therefore, when counting the dimension of this stack, after taking into consideration automorphisms of the underlying bundles, we can focus on moduli in the choices of $V_j$ and gluing at the nodes. Since all limit linear series here considered are chain-adaptable, counting the moduli in choices of $V_j$ reduces to counting the moduli in choices of $(P_j,P_{j+1})$-adapted basis. However, a major difficulty lies in analyzing feasibilities/possibilities of gluing at the nodal points of $X_g$. To facilitate such analysis, we introduce various notations and definitions.

In Definition \ref{Def}, we use $\phi_j$ to denote an isomorphism $\mE''_{j-1}|_{P_j}\to\mE''_j|_{P_j}$. It induces an isomorphism $\PP\mE_{j-1}|_{P_j}\to\PP\mE_j|_{P_j}$. We abuse notation and denote the latter as $\phi_j$ as well.
\begin{defn}
  Let $V$ be an $n$-dimensional space of sections of some rank two vector bundle $\mE$ on an elliptic curve $C$. Let $P$ be a point on $C$, and $(b_i)$ be the vanishing sequence of $V$ at $P$. Denote
  $$N^V(P)=\{i| b_i \text{ is a non-repeated vanishing order of }V \text{ at }P\}.$$

  We call a partition $\tau^V(P)$ of $N^V(P)$ the \textbf{configuration} at $P$ of $V$, if $\fa s,s'\in V$ such that $\ord_P(s)=b_{j_1},\ord_P(s')=b_{j_2}$, $s|_P=s'|_P$ (see \tb{Notation 10}) if and only if $j_1,j_2$ belong to the same index set in $\tau^V(P)$.

  We denote $|\tau^V(P)|$ to be the size of a configuration at $P$, i.e. the number of disjoint subsets $\tau^V(P)$ divides $N^V(P)$ into.

  Let $V'\subset V$ be s sub-space of sections. Suppose the following are true:
  \begin{enumerate}
    \item $N^{V'}(P)\subset N^{V}(P)$;
    \item any element in $\tau^{V'}(P)$ is a subset of one element in $\tau^V(P)$;
    \item no two elements in $\tau^{V'}(P)$ are subsets of the same element in $\tau^V(P)$.
  \end{enumerate}
  Then, we say $\tau^{V'}(P)$ is a \textbf{sub-configuration} of $\tau^V(P)$.

  Given $\tau^V(P)$, suppose $b_{i_1},...,b_{i_S}$ are the non-repeated vanishing orders in $(b_i)$. Let $s_i$ be any section in $V$ vanishing to order $b_{i_m}$ at $P$ ($m=1,...,M$). We call $(s_{i_m}|_P)_{m=1}^M$ to be \textbf{the sequence of points associated to $V$} in $\PP\mE|_P$.
\end{defn}
\begin{ex}\label{config}
We list a few examples and conventions for configurations:\\
\noindent(1) For any object $((\mE_j,V_j),(\phi_j))$ in $\mathcal{G}^{k,\text{EHT}}_{2,\omega,d_{\bullet},a^{\Gamma}}(X_g)$, it is clear that $N^{V_j}(P_{j+1})=N^{V_{j+1}}(P_{j+1})$ and $\tau^{V_j}(P_{j+1})=\tau^{V_{j+1}}(P_{j+1})$. Therefore, we simply write $N(P_{g+1})$ for the set of indices of non-repeated vanishing orders in the vanishing sequence and $\tau(P_{j+1})$ for any possible configuration. Note: this does NOT suggest that the configuration is fixed among different limit linear series.\\
\ \\
\noindent(2) We say a configuration at $P$ is \textbf{trivial} if $|N(P)|\le 1$. Now suppose $P=C_t\cap C_{t+1}$ is a nodal point on a chain $X_g$. When justifying the existence of chain-adapted limit linear series $((\mE_j,V_j),(\phi_j))$ on $X_g$, if $\tau(P_{t+1})$ is trivial, then there is no obstruction to the existence of such limit linear series coming from the choice of $\phi_{t+1}$.\\
\ \\
\noindent(3) Suppose $\mE_t=\oplus_{s=1}^2 \mO(c_sP_t+(d-c_s)P_{t+1})$ and the prescribed vanishing sequence at $P_t$ is $(a_i)$. Suppose further $\exists i,i'\in N(P_t)$ such that $\tau^{V_t}(P_t)$ contains an element $S$ and $S\supset\{i,i'\}$. If $a_i+a_{i'}=c_1+c_2-1$, then by Lemma \ref{aux}, \ref{sym}, $i,i'$ are also contained in some element in $\tau^{V_t}(P_{t+1})$; if $a_i+a_{i'}\neq c_1+c_2-1$, then $i,i'$ do not belong to a same index set in $\tau^{V_t}(P_{t+1})$ unless $\mO(c_1P_t+(d-c_1)P_{t+1})\cong \mO(c_2P_t+(d-c_2)P_{t+1})$.\\
\ \\
(4) Suppose $\mE_t=\oplus_{s=1}^2 \mO(a_sP_t+(d-a_s)P_{t+1})$, $\mE_{t+1}=\oplus_{s=1}^2 \mO((a_s+1)P_t+(d-a_s-1)P_{t+1})$. Suppose further $a_1,a_2,a_1+1,a_2+1$ are non-repeated vanishing orders in $\van_{P_t}(V_t)$. Then, the sequence of points in $\PP \mE_t|_{P_t}$ (resp. $\PP \mE_{t+1}|_{P_t}$) associated to $V_t$ (resp. $V_{t+1}$) must contain $T^t_1|_{P_t},T^t_2|_{P_t}$ (resp. $T^{t+1}_1|_{P_t},T^{t+1}_2|_{P_t}$), where $T^t_1,T^t_2$ (resp. $T^{t+1}_1,T^{t+1}_2$) are canonical sections of $\mE_t$ (resp. $\mE_{t+1}$) (see Remark \ref{csec}). Moreover, $\phi_{t+1}(T^t_i|_{P_t})=T^{t+1}_i|_{P_t}$ for $i=1,2$.
\end{ex}
\section{General Induction Arguments}
In this section, we shall inductively define $(g,k)$-standard vanishing conditions $a^{\Gamma}$ for varying pairs of $(g,k)$ such that an open substack of \ma\ has dimension $\rr$. This is done for an infinite collection of $(g,k)$ such that $L(g,k)\to-\infty$ as $g,k$ increase. The construction is accomplished in multiple steps.\\
\ \\
\textbf{Step 1}: Fix $k=2k_1+1\ge 5$. Suppose $\mathcal{G}^{k,\text{EHT}}_{2,\omega_g,d_{\bullet},a^{\Gamma}}(X_g)$ is non-empty. Then, $\mathcal{G}^{k+2m,\text{EHT}}_{2,\omega_g(NP_{g+1}),\hat{d}_{\bullet},c^{\Gamma}}(X_g)$ ($N>k_1+m$) is also non-empty. Here, $\hat{d}_i=d_i+2N$ and $c^{\Gamma}$ is a $(k+2m)\times(2g-2)$ matrix defined using $a^{\Gamma}$. More precisely, there is a morphism from a substack of $\mathcal{G}^{k+2m,\text{EHT}}_{2,\omega_g(NP_{g+1}),\hat{d}_{\bullet},c^{\Gamma}}(X_g)$ onto some non-empty open substack of $\mathcal{G}^{k,\text{EHT}}_{2,\omega_g,d_{\bullet},a^{\Gamma}}(X_g)$ with fiber dimension 1.\\
\ \\
\textbf{Step 2}: For $q=\max\{1,\lfloor{k_1\over 3}\rfloor\}$, take $N=4k_1+6-q$ and let $g'=g+N,k'=k+4$. Let $\mathcal{G}^{k',\text{EHT}}_{2,\omega_g(NP_{g+1}),\hat{d}_{\bullet},c^{\Gamma}}(X_g)$ be the moduli stack defined in \tb{Step 1}. We further define a $(g',k')$-standard vanishing condition $c^{\Gamma'}$ such that $c^{\Gamma}$ consists of the first $2g-2$ columns of $c^{\Gamma'}$. This defines a new moduli stack $\mathcal{G}^{k',\text{EHT}}_{2,\omega_{g'},d'_{\bullet},c^{\Gamma'}_{g'}}(X_{g'})$ of limit linear series on a chain of $g'$ elliptic cruves, $X_{g'}$, where the first $g$ entries of $d'_{\bullet}$ are precisely the entries of $\hat{d}_{\bullet}$. We further show there is a forgetful morphism from an open substack of $\mathcal{G}^{k',\text{EHT}}_{2,\omega_{g'},d'_{\bullet},c^{\Gamma'}_{g'}}(X_{g'})$ onto some locally closed substack of $\mathcal{G}^{k,\text{EHT}}_{2,\omega_g(NP_{g+1}),\hat{d}_{\bullet},c^{\Gamma}}(X_g)$. \\
\ \\
\tb{Step 3}: Repeat the process in \tb{Step 2} for $N=2k_1+2$, $g'=g+N$ and $k'=k+2$. This gives an inductive construction in which $L(g,k)$ is constant as $g,k$ vary. Combining this with the previous step, we get the construction for the main result of this paper in cases where $k$ is odd.\\
\ \\
\tb{Step 4}: Every $\mathcal{G}^{k',\text{EHT}}_{2,\omega_{g'},d'_{\bullet},c^{\Gamma'}_{g'}}(X_{g'})$ in \tb{Step 2}, \tb{Step 3}, induces some $\mathcal{G}^{k'-1,\text{EHT}}_{2,\omega_{g''},d''_{\bullet},c^{\Gamma''}}(X_{g''})$, where $N'={k'-1\over 2}$, $g''=g-N'$ and $d''_{\bullet}=(d'_1,...,d'_{g''})-(2N',...,2N')$. This gives the construction for the main result in cases where $k$ is even.
\subsection{Induction Step 1}
The intuition behind Step 1 is very simple: let $\mE$ be some vector bundle on a smooth projective curve $C$ and $V$ be an $n$-dimensional space of sections of $\mE$. For any $n>0$, one can always find a $(k+n)$-dimensional space $V'$ of rational sections containing $V$. This is the same as saying that $\exists V'\subset \Gamma(\mE(Z))$ of dimension $k+n$ containing $V(Z)$, where $D$ is some effective divisor of $C$ of sufficiently large degree.

Now start with any genus $g$, dimension $k$ limit linear series $((\mE_j,V_j),(\phi_j))$, to get a family of genus $g'$, dimension $k'$ limit linear series, we take an effective divisor $Z_j$ of degree $g'-g$ on every component $C_j$ of $X_g$ and consider $\mE_j(Z_j)$ in place of $\mE_j$. We do not take arbitrary $Z_j$ or arbitrary $k'$-dimensional subspace of $\Gamma(\mE_j(Z_j))$ containing $V_j(Z_j)$; rather, we want its vanishing sequence to be of some desired form. We make this idea precise via the following lemma.
\subsubsection{The Construction in \textbf{Step 1}}
\begin{lem}\label{C1}
Fix $k=2k_1+1\ge 5$. Suppose $a^{\Gamma}$ is a $(g,k)$-standard. Let $d_{\bullet}=(d_1,...,d_g)$ be a sequence of non-negative integers, with $d_{j_1},...,d_{j_{2T}}$ being the only odd entries of $d_{\bullet}$. Moreover, for $j=1,...,g$, the pair $(a_{i,2j-2})$ and $(a_{i,2j-1})$ satisfy either Lemma \ref{one} (or Lemma \ref{one2}) with $d=\frac{1}{2}d_j$ (or $d=\frac{1}{2}(d_j-1)$).

Fix some integers $n=2m>0$ and $N> k_1+m$. Let $\hat{d}_{\bullet}=(\hat{d}_1,...,\hat{d}_g)=d_{\bullet}+(2N,...,2N)$ and $\hat{a}^{\Gamma}=a^{\Gamma}+A$, where $A$ is a $k\times(2g-2)$ matrix whose $2j$-th columns are all $(0,...,0)^T$ and whose $(2j-1)$-th columns are all $(N,...,N)^T$, for $j=1,...,g-1$.

Then, there exists a unique $(k+n)\times(2g-2)$ matrix $c^{\Gamma}=(c_{ij})$ satisfying the following conditions:

\begin{enumerate}
  \item $(c_{ij})_{i\le k}=\hat{a}^{\Gamma}$.
  \item For $j=1,...,g-1$, $c_{i,2j-1}+c_{i,2j}=g+N-1$.
  \item For $j=1,...,g$, if $(a_{i,2j-2})$ and $(a_{i,2j-1})$ satisfy Lemma \ref{one} (resp. Lemma \ref{one2}) with $d=\frac{1}{2}d_j$ (resp. with $d=\frac{1}{2}(d_j-1)$), so do $(c_{i,2j-2})$ and $(c_{i,2j-1})$ with $d=\frac{1}{2}\hat{d}_j$ (resp. with $d={1\over2}(\hat{d}_j-1)$).\footnote{We define $(a_{i,0})=a(k),a_{i,2g-1}=a(k)^{\rev},(c_{i,0})=a(k+n),(c_{i,2g-1})=(N+k_1,[N+k_1-1]_2,[N+k_1-2]_2,...,[N]_2,N-1-k_1,[N-2-k_1]_2,[N-3-k_1]_2,...,[N-m-k_1]_2,N-m-k_1-1)$.}
  \item For $i=k+1,...,k+n$, if $c_{i,2j_t-2}+c_{i,2j_t-1}=\frac{1}{2}(\hat{d}_{j_t}-1)$, then $c_{i,2j_{t+1}-2}+c_{i,2j_{t+1}-1}=\frac{1}{2}(\hat{d}_{j_{t+1}}-1)-1$; if $c_{i,2j_t-2}+c_{i,2j_t-1}=\frac{1}{2}(\hat{d}_{j_t}-1)-1$, then $c_{i,2j_{t+1}-2}+c_{i,2j_{t+1}-1}=\frac{1}{2}(\hat{d}_{j_{t+1}}-1)$.
  \item $c_{k+n,2}+c_{k+n,3}=\frac{1}{2}(\hat{d}_2-1)$.
\end{enumerate}
\end{lem}
\begin{proof}
  See Appendix Proposition \ref{AM1}.
\end{proof}
\subsubsection{Sufficiently Generic Objects}
Given a $(g,k)$-standard $a^{\Gamma}$ and the corresponding $c^{\Gamma}$ as in \ref{C1}, we shall relate some locally closed substack $\GG^0$ of $\mathcal{G}^{k+2m,\text{EHT}}_{2,\omega_g(NP_{g+1}),\hat{d}_{\bullet},c^{\Gamma}}(X_g)$ to $\mathcal{G}^{k,\text{EHT}}_{2,\omega_g,d_{\bullet},a^{\Gamma}}(X_g)$ via some stack morphism $\psi$. However, $\psi$ will not be surjective and we describe objects in \ma\ over which $\psi$ has non-empty fibers. To do so, we need some technical definitions. We start with a simple observation.
\begin{ex}\label{fiso}
  Let $\mE=\mO(aP+(d-a)Q)\op\mO(cP+(d-c)Q)$ be a vector bundle over an elliptic curve, where $P,Q$ are general points and $0\le a,c\le d$. For $e:0\le e\le d$ such that $e\neq a-1,a,c-1,c$, $\Gamma(\mE(-eP-(d-1-e)Q))$ is two-dimensional and consists of sections vanishing precisely to order $e,d-1-e$ at $P,Q$ resp. Suppose $\s(s_1,s_2)=\Gamma(\mE(-eP-(d-1-e)Q))$. Then, $\s(s_1|_P,s_2|_P)=\mE|_P$, $\s(s_1|_Q,s_2|_Q)=\mE|_Q$. The conditions $s_1|_P\mapsto s|_Q,s_2|_P\mapsto s_2|_Q$ linearly extends to an isomorphism $\mE|_P\to\mE|_Q$, which induces an isomorphism $\tau(e):\PP\mE|_P\to \PP\mE|_Q$.

  More generally, let $((\mE_j,V_j),(\phi_j))$ be an object in \ma. Suppose $\mE_j=\mO(a_j,d-a_j)\op\mO(c_j,d-c_j)$ (\tb{Notation 7}) for $j:j^1<j<j^2$. Fix a sequence of integers $e_{j^1+1},...,e_{j^2-1}$ such that $e_j\neq a_j-1,a_j,c_j-1,c_j$. By previous observation, one gets isomorphisms $\tau(e_j):\PP\mE_j|_{P_j}\to \PP\mE_j|_{P_{j+1}}$ by fixing a basis for $\Gamma(\mE_j(-e_jP-(d-1-e_j)P_{j+1}))$. In particular, $\tau:=\tau(e_{j^2-1})\circ...\circ\tau(e_{j^1+1})$ is an isomorphism $\PP\mE_{j^1+1}|_{P_{j^1+1}}\to\PP\mE_{j^2-1}|_{P_{j^2-1}}$.
\end{ex}
Example \ref{fiso} motivates the following definition:
\begin{defn}\label{suff}
 Given some $(g,k)$-standard $a^{\Gamma}$, let $((\mE_j,V_j),(\phi_j))$ be an object of \ma. Fix $q\in\PP \mE_{j^1}|_{P_{j^1+1}},q'\in \PP \mE_{j^1+M}|_{P_{j^1+M}}$ and suppose for all $j:j^1<j<j^1+M$, $\mE_j=\mO(a_j,d-a_j)\op\mO(c_j,d-c_j)$ is semi-stable. Let $a$ be an integer such that $a+p\neq a_{j^1+p}-1,a_{j^1+p},c_{j^1+p}-1,c_{j^1+p}$ for all $p=1,2,...,M-1$. Define
 $$\tau_a:=\phi_{j^1+M}\circ\tau(a+M-1)\circ...\circ\tau(a+2)\circ\phi_{j^1+2}\circ\tau(a+1)\circ\phi_{j^1+1}.$$
 Here, $\tau(e_j):\PP\mE_j|_{P_j}\to \PP\mE_j|_{P_{j+1}}$ is the isomorphism defined in Example \ref{fiso}.

 For any $a\in\ZZ$, we say $q$ is \textbf{glued to $q'$ in $a$-th order via $\phi_{j^1+1},...,\phi_{j^1+M}$}, if for some $S\ge\max\{0,M-a\}$ there exists $s\in\Gamma(\mE_{j^1}(SP_{j^1+1}))$, such that:
 \begin{enumerate}
   \item $\ord_{P_{j^1+1}}(s_{j^1})=a+S\ge0$ and $s|_{P_{j^1+1}}=q$;
   \item $\tau_{b-a-1}(q)=q'$.\footnote{Here, $\tau_{b-a-1}$ should be thought of as the naturally induced isomorphism from $\PP\mE_{j^1}(SP_{j^1+1})|_{P_{j^1+1}}$ to $\PP\mE_{j^1+M}(SP_{j^1+M})|_{P_{j^1+M}}$.}
 \end{enumerate}
 In this case, suppose $s'\in\Gamma(\mE_{j^1+M}(SP_{j^1+M}))$ such that $\ord_{P_{j^1+M}}(s')=b-a+M-1$ and $s'|_{P_{j^1+M}}=q'$, we also say $s$ \tb{is glued to} $s'$ \textbf{via $\phi_{j^1+1},...,\phi_{j^1+M}$}.
\end{defn}
Recall that we denote $\mE_1,...,\mE_g$ to be the underlying vector bundles of \ma\ (see Remark \ref{uopenr}). Among them, $\mE_{j_1},...,\mE_{j_{2T}}$ are the only unstable bundles.
\begin{defn}
 Given any $j_t$, denote $q^t_1$ (resp. $q^t_2$) to be the point in $\PP\mE_{j_t}|_{P_{j_t}}$ (resp. $\PP \mE_{j_t}|_{P_{j_t+1}}$) which is the image of sections of the destabilizing summand of $\mE_{j_t}$. A geometric object $((\mE_j,V_j),(\phi_j))$ of \ma\ is said to be \textbf{sufficiently generic} if $\fa t$, $q^t_2$ is not glued to $q^{t+1}_1$ in $a$-th order, for all $a$ such that $a+N$ is a non-repeated integer among $c_{k+1,2j_t-1},...,c_{k+n,2j_t-1}$, where $c^{\Gamma}=(c_{i,j})$, $N$ and $n$ are as in Lemma \ref{C1}, with the parameter $S$ in Definition \ref{suff} taken to be $N$.\footnote{Technically speaking, one should call this property \it{sufficiently generic with respect to}\rm\ $N$, but we will refer to it as is stated, since $N$ is clear from the context.}
\end{defn}
This is the key notion needed for analyzing induction \textbf{Step 1}. An important property we shall apply is that being sufficiently generic is an open condition. We start with an auxiliary lemma.
\begin{lem}\label{gmor}
  Let $\mE_1,...,\mE_g$ be the underlying vector bundles of \ma. Suppose for all $j:j_t<j<j_{t+1}$, $\mE_j=\mL_{j,1}\op\mL_{j,2}$ is semi-stable. Let $a$ be an integer satisfying the following condition:
  \begin{equation}\label{cond1}
   \begin{split}
     a\le a_{k,2j_t-1},\text{ with equality only when } a_{k,2j_t-1}<a_{k-1,2j_t-1}\\ \text{and }a_{k,2j-2}+a_{k,2j-1}=g-2, \fa j:j_t<j<j_{t+1}.
    \end{split}
  \end{equation}
  Then, ``$q^t_2$ is glued to $q^{t+1}_1$ in $a$-th order via $\phi_{j_t+1},...,\phi_{j_{t+1}}$'' is a closed condition.
\end{lem}
\begin{proof}
Given \eqref{cond1}, for all $j:j_t<j<j_{t+1}$, $b-a+(j-j_t)>\max\{a_j,c_j\}$. In particular, $b-a+(j-j_t)\neq a_j-1,a_j,c_j-1,c_j$.

By definition, if $q^t_2$ is glued to $q^{t+1}_1$ in $a$-th order, $\tau_{b-a-1}(q^{t+1}_1)=q^t_2$, where $$\tau_{b-a-1}:\PP\mE_{j_t}(NP_{j_t+1})|_{P_{j_t+1}}\to\PP\mE_{j_{t+1}}|_{P_{j_{t+1}}}\footnote{See Definition \ref{suff}.}$$ is an isomorphism. Since $\tau_{b-a+1}$ is a composition of $\phi_j$'s and some $\tau(e_j)$'s, where $\tau(e_j)$ is a fixed isomorphism $\PP\mE_j(NP_{j+1})|_{P_j}\to\PP\mE_j(NP_{j+1})|_{P_{j+1}}$, this clearly poses a closed condition on $\prod\iso(\PP\mE_j|_{P_{j+1}},\PP\mE_{j+1}|_{P_{j+1}})$. It further induces a closed condition on \ma\ since there is a morphism \ma$\to\prod \iso(\mE_j|_{P_{j+1}},\mE_{j+1}|_{P_{j+1}})\to\prod \iso(\PP \mE_j|_{P_{j+1}},\PP \mE_{j+1}|_{P_{j+1}})$, where the first arrow is a forgetful morphism.
\end{proof}
\begin{cor}
  Suppose $a^{\Gamma}$ is $(g,k)$-standard. Then, being sufficiently generic is an open condition on \ma.
\end{cor}
\begin{proof}
  First of all, by conditions 2, 4 and 5 in Lemma \ref{C1}, one gets $c_{k+1,2j-1}=c_{k+1,2j-3}+1=g+N-1-k_1-j$. Hence, the condition that $c_{k+n,2j_t-1}\le a+N\le c_{k+1,2j_t-1}$ implies that $a\le g-1-k_1-j_t$. By condition 5 in the definition of $(g,k)$-standardness, $a_{k,2j_t-1}=g-1-k_1-j_t$ only if $a_{k,2j-2}+a_{k,2j-1}=g-2, \fa j:j_t<j<j_{t+1}$; if $c_{k+1,2j_t-1}=c_{k,2j_t-1}$, then $a_{k-1,2j_t-1}>a_{k,2j_t-1}$ must hold. Hence, condition \eqref{cond1} in Lemma \ref{gmor} holds for all vanishing orders in question. Consequently, being sufficiently generic is then the complement of finitely many closed conditions. Hence, the corollary follows.
\end{proof}
Analogously, we also describe a substack of $\mathcal{G}^{k+n,\text{EHT}}_{2,\omega_g(NP_{g+1}),\hat{d}_{\bullet},c^{\Gamma}}(X_g)$, which maps onto sufficiently generic objects in \ma:
\begin{lem}\label{suffover}
Define $\GG^0$ to be the fully faithful sub-category of $\mathcal{G}^{k+n,\text{EHT}}_{2,\omega_g(NP_{g+1}),\hat{d}_{\bullet},c^{\Gamma}}(X_g)$ consisting of objects $((\mE_j,V_j),(\phi_j))$ satisfying the following conditions:
\begin{enumerate}
 \item For every $j$, $V_j$ contains some $k$-dimensional subspace $V'_j$ such that $$\van_{P_j}(V'_j)=(c_{i,2j-2})_{i\le k}, \van_{P_{j+1}}(V'_j)=(c_{i,2j-1})_{i\le k};$$
 \item given $V'_j$ as in (1), denote $\mE'_j:=\mE_j(-NP_{j+1})$ and $i:\Gamma(\mE'_j)\to\Gamma(\mE_j)$ to be the obvious injection, $((\mE'_j,i^{-1}(V'_j)),(\phi_j))$ is an object in \ma.
\end{enumerate}
  Then, $\GG^0$ is a locally closed substack of $\mathcal{G}^{k+n,\text{EHT}}_{2,\omega_g(NP_{g+1}),\hat{d}_{\bullet},c^{\Gamma}}(X_g)$.
\end{lem}
\begin{proof}
Define $b'$ by the equation $\sum\deg(\mE_j)-2b'(g-1)=2g-2+2N$ and denote $b=b'-N$.

Recall that $\phi_j$ are isomorphisms $\mE''_j|_{P_{j+1}}\stackrel{\sim}{\to}\mE''_{j+1}|_{P_{j+1}}$, where $\mE''_j=\mE_j(-(b'-b_j)P_j-b_{j+1}P_{j+1})$ and $b_1,...,b_{g+1}$ are some pre-chosen integers such that $b_1=b',b_{g+1}=0$. Note that it makes sense to take gluing data $\phi_j$ in $((\mE'_j,i^{-1}(V'_j)),(\phi_j))$, because
$$
\mE''_j=\mE_j(-(b'-b_j)(P_j)-b_{j+1}(P_{j+1}))=\mE'_j(-(b-(b_j-N))(P_j)-(b_{j+1}-N)(P_{j+1}))
$$
and there is a canonical identification
$$\mE'_g(-(b-(b_g-N))(P_g)+N(P_{g+1}))|_{P_g}\stackrel{\sim}{\to}\mE'_g(-(b-(b_g-N))(P_g))|_{P_g}.$$

A simple calculation shows that objects in $\GG^0$ have the vanishing sequence $(c_i)$ at $P_{g+1}$:
$$(N+k_1,[N+k_1-1]_2,...,[N]_2,N-1-k_1,[N-2-k_1]_2,...,[N-m-k_1]_2,N-m-k_1-1)$$
Fixing a vanishing sequence at the smooth point $P_{g+1}$ is equivalent to posing a Schubert condition together with finitely many open conditions. They together determine a locally-closed substack $\GG$ of $\mathcal{G}^{k+n,\text{EHT}}_{2,\omega_g(NP_{g+1}),\hat{d}_{\bullet},c^{\Gamma}}(X_g)$. We claim that $\GG^0$ is an open substack of $\GG$.

Given an object $((\mE_j,V_j),(\phi_j))$ in $\mathcal{G}^{k+n,\text{EHT}}_{2,\omega_g(NP_{g+1}),\hat{d}_{\bullet},c^{\Gamma}}(X_g)$, for any $j$, $V_j$ always contains some $k$-dimensional subspace $V'_j$ such that $\van_{P_j}(V'_j)=(c_{i,2j-2})_{i\le k}, \van_{P_{j+1}}(V'_j)=(c_{i,2j-1})_{i\le k}$. Hence, $((\mE_j,V_j),(\phi_j))$ is not an object of $\GG^0$ if and only if there do not exist
such $V'_j$ ($j=1,...,g$) that are compatible with $\phi_2,...,\phi_g$.

First of all, if $c_{k,2j-1}>c_{k+1,2j-1}$, then any $V'_j,V'_{j+1}$ satisfying the given vanishing conditions will be compatible with $\phi_{j+1}$. More precisely, the configuration $\tau^{i(V'_j)}(P_{j+1})=\tau^{i(V'_{j+1})}(P_{j+1})$ is a sub-configuration of $\tau^{V_j}(P_{j+1})=\tau^{V_{j+1}}(P_{j+1})$. Therefore, $((\mE_j,V_j),(\phi_j))$ is not an object in $\GG^0$ if and only if $\exists j^1<j^2$ such that the following conditions hold:
\begin{enumerate}
  \item $c_{k,2j^1-2}<c_{k+1,2j^1-2}$, $c_{k,2j^2-1}>c_{k+1,2j^2-1}$ (and $V'_{j^2}=V_{j^2}(-c_{k,2j^2-1}P_{j^2+1})$);
  \item for all $j:j^1<j<j^2$, $c_{k,2j-2}=c_{k+1,2j-2},c_{k,2j-1}=c_{k+1,2j-1}$;
  \item there do not exist $s_1\in V_{j^1},s_2\in V_{j^2}$ such that $\ord_{P_{j^t}}(s^t)=c_{k,2j^t-2}$, $\ord_{P_{j^t+1}}(s^t)=c_{k,2j^t-1}$ for $t=1,2$,
  and $s^1|_{P_{j^1+1}}$ is glued to $s^2|_{P_{j^2}}$ via $\phi_{j^1+1},...,\phi_{j^2}$.
\end{enumerate}
\tb{Claim:} for every fixed pair $(j^1,j^2)$ of such indices, 1-3 form a closed condition on $\GG$.

By \tb{Step 4.1} in Appendix A.1 and Corollary \ref{AC1}, the conditions $c_{k,2j^1-2}<c_{k+1,2j^1-2}$, $c_{k,2j^1-1}=c_{k+1,2j^1-1}$ imply that $j^1=j_{2t}$ for some $t\le T$, so $\mE_{j^1}$ is unstable. We denote $q^1$ to be the image of sections of the destabilizing summand of $\mE_{j^1}$ in $\PP\mE_{j^1}|_{P_{j^1+1}}$. It is clear that by varying the choice of $V'_{j^1}$, $s^1|_{P_{j^1+1}}$ could be any point in $\PP\mE_{j^1}|_{P_{j^1+1}}$ but not $q^1$. Meanwhile, $\mE_{j^2}$ is either semi-stable or unstable; when it is semi-stable, $s^2|_{P_{j^2}}$ is the image of one of the canonical sections of $\mE_{j^2}$ in $\PP\mE_{j^2}|_{P_{j^2}}$; when it is unstable, $s^2|_{P_{j^2}}$ is the image of sections of the destabilizing summand. In either case, $s^2|_{P_{j^2}}$ is a fixed point $q^2$ in $\PP\mE_{j^2}|_{P_{j^2}}$. Hence, $(j^1,j^2)$ is a pair of indices satisfying the conditions (1)-(3) if and only if $q^1$ glues to $q^2$ via $\phi_{j^1+1},...,\phi_{j^2}$. Following the same argument as in the proof of Lemma \ref{gmor}, we get the claim.

Consequently, $\GG^0$ is the complement of a union of finitely many closed loci in $\GG$. Hence, it is a open substack of $\GG$ and a locally-closed substack of $\mathcal{G}^{k+n,\text{EHT}}_{2,\omega_g(NP_{g+1}),\hat{d}_{\bullet},c^{\Gamma}}(X_g)$.
\end{proof}
We are ready to describe the morphism $\psi$. In order to do so, we need to describe $T$-valued points of \ma\, where $T$ is some arbitrary $K$-scheme. Following \cite{Olls}, we give a brief summary of relevant definitions in \ref{secsub}.\\
\indent We now state and proof the main result for Induction Step 1:
\begin{thm}\label{Induction}
Fix $k=2k_1+1\ge 5$. Given $\GG^0$ as in \ref{suffover} and denote $\GG_1$ to be the open substack of \ma\ of sufficiently generic objects. Then, there is a stack morphism $\psi:\GG^0\to\GG_1$.
\end{thm}
\begin{proof}
We first describe the morphism $\psi$ on objects. Let $((\mE_j,V_j),(\phi_j))$ be an object of $\GG^0$. Recall $b'$ is defined via the identity $\sum\deg(\mE_j)-2b'(g-1)=2g-2+2N$.

Let $T$ be some $K$-scheme and suppose $((\mE_j,V_j),(\phi_j))$ is an object of $\GG^0$ over $T$. Here, $\mE_j$ is a rank two vector bundle on $T\times C_j$, $V_j$ is a rank-$k$ sub-bundle (in the sense of \ref{subb}) of $\pi^j_*\mE_j$ and $\phi_j:\mE''_j|_{T\times P_{j+1}}\to\mE''_{j+1}|_{T\times P_{j+1}}$. Here,
\begin{enumerate}
  \item $\pi^j_*$ is the projection $T\times C_j\to T$;
  \item $\mE''_j:=\mE_j(-(b'-e_j)(T\times P_j)-e_{j+1}(T\times P_{j+1}))$, where $e_1,e_2,..,e_{g+1}$ are some integers with $e_1=b',e_{g+1}=0$ (see Remark \ref{twist});
  \item $(c_{i,2j-2})$ and $(c_{i,2j-1})^{\rev}$ are the vanishing sequences of $V_j$ along $P_j$, $P_{j+1}$ (in the sense of \ref{van}) resp.
\end{enumerate}
Define $\psi((\mE_j,V_j),(\phi_j))$ to be an object of the form $((\mE'_j,V'_j),(\phi_j))$, where $\mE'_j=\mE_j(-N(T\times P_{j+1}))$. Note that by the proof of Lemma \ref{suffover}, it makes sense to take $(\phi_j)$ as the gluing data.

It remains to describe $V'_j$. We first decide notations for some bundle maps:
\begin{enumerate}
  \item Denote $i:\pi^j_*\mE'_j\to\pi^j_*\mE_j$ to be the obvious inclusion map;
  \item denote $V(-mP):=\ker(\gamma^j_m(P))$, where $\gamma^j_m(P)$ is the map $V_j\to\pi^j_*(\mE_j|_{m(T\times P)})$, for $P=P_j,P_{j+1}$;
  \item denote $\gamma^j_{m,\mL}(P)$ to be the map $V_j(-mP)\to\pi^j_*(\mE_j(-m(T\times P))|_{T\times P})/\mL$, where $\mL$ is some sub-line bundle of $\pi^j_*(\mE_j(-m(T\times P))|_{T\times P})$, for $P=P_j,P_{j+1}$.
\end{enumerate}

We start from $V'_g$. Define $V'_g=i^{-1}(V_g(-NP_{g+1}))$. Since $V_g$ has fixed vanishing sequence along $P_{g+1}$, by Appendix Lemma \ref{case1}, $V_g(-NP_{g+1})$ is a rank $k$ sub-bundle of $\pi^g_*(\mE_g)$ (in the sense of \ref{subb}) and hence $V'_g$ is a rank $k$ sub-bundle of $\pi^g_*(\mE'_g)$. One can directly check that the vanishing sequence of $V'_g$ along $P_g$ and $P_{g+1}$ (in the sense of \ref{van}) are $(a_{i,2g-2})$ and $(k_1,[k_1-1]_2,...,[2]_2,[0]_2)^{\rev}$ resp.\\
\indent One then defines all $V'_j$ in the following way:\\
\noindent\tb{Case 1} If $c_{k,2j-1}>c_{k+1,2j-1}$, then $V'_j:=i^{-1}(V_j(-c_{k,2j-1}P_{j+1}))$; by the same argument as for $V'_g$, $V'_j$ is a rank-$k$ sub-bundle of $\pi^j_*\mE'_j$ satisfying the desired vanishing conditions.\\
\tb{Case 2} If $c_{k,2j-1}=c_{k+1,2j-1}$, then there exists a smallest $j'>j$ such that $c_{k,2j'-1}>c_{k+1,2j'-1}$ and $V'_{j'}$ is defined as above. By Lemma \ref{case1}$, V_{j'}(-c_{k,2j'-2}P_{j'})$ is a sub-line bundle of $\pi^{j'}_*\mE_{j'}$. Furthermore, a sub-line bundle of $\pi^{j'}_*(\mE_{j'}(-c_{k,2j'-2}(T\times P_{j'}))|_{T\times P_{j'}})$ is determined via the map:
$$V_{j'}(-c_{k,2j'-2}P_{j'})\to\pi^{j'}_*(\mE_{j'}(-c_{k,2j'-2}(T\times P_{j'})))\to\pi^{j'}_*(\mE_{j'}(-c_{k,2j'-2}(T\times P_{j'}))|_{T\times P_{j'}}).$$
We denote this sub-line bundle to be $\mL^{j'}$ and define $$V'_{j'-1}:=i^{-1}(\ker(\gamma^{j'-1}_{c_{k,2j'-3},\phi^{-1}_{j'}(\mL^{j'})}(P_{j'}))).$$
By Appendix Corollary \ref{case2}, $V'_{j'-1}$ is a rank $k$ sub-bundle of $\pi^{j'-1}_*\mE'_{j'-1}$ satisfying the desired vanishing conditions. Notice it is important that we start with an object in $\GG^0$ instead of arbitrary objects in $\mathcal{G}^{k+n,\text{EHT}}_{2,\omega_g(NP_{g+1}),\hat{d}_{\bullet},c^{\Gamma}}(X_g)$; in general, $V'_{j'-1}$ may have rank $k-1$.

If $j<j'-1$, $V'_{j'-1}(-c_{k,2j'-4}P_{j'-1})$ further determines a sub-line bundle $\mL^{j'-1}$ of\\ $\pi^{j'-1}_*(\mE_{j'-1}(-c_{k,2j'-4}(T\times P_{j'-1}))|_{T\times P_{j'-1}})$ and one can define $V'_{j'-2}$ in a similar way to $V'_{j'-1}$. Hence, we can inductively define $V'_j,...,V'_{j'-1}$.

We have thus described all $V'_j$. Notice that $V'_j$ are compatible with the gluing data $(\phi_j)$ precisely by their definition.

Next, we describe $\psi$ for morphisms. Let $x,y$ denote objects in $\GG^0$ over $K$-schemes $S$ and $T$ resp. There is an arrow $\xi:y\to x$ if and only if for some morphism $f:S\to T$, there exists an isomorphism $y\to f^*x$. Note that isomorphisms of limit linear series are induced by isomorphisms of the underlying bundles. Suppose $y\to f^*x$ is induced by $(\theta_j:\mE_j\stackrel{\sim}{\to}\mF_j)$, $\theta_j$ induces an isomorphism $\mE_j(-N(T\times P_{j+1}))\to\mF_j(-N(T\times P_{j+1}))$, which further induces an isomorphism $\psi(y)\stackrel{\sim}{\to}f^*\psi(x)$. This determines a (Cartesian) arrow $\psi(y)\to\psi(x)$ which we define as $\psi(\xi)$. It only remains to show that this definition is functorial.

Given two arrows $\xi_1:y\to x$ and $\xi_2:z\to y$ over $f:S\to T$ and $g:U\to S$ resp., there exist isomorphisms $\theta^1:y\to f^*x$ and $\theta^2:z\to g^*y$ induced by relevant vector bundle isomorphisms $(\theta^1_j),(\theta^2_j)$. Then $\xi_2\circ\xi_1:z\to x$ corresponds to an isomorphism induced by $(\theta^2_j\circ g^*\theta^1_j)$. Since $\psi(\xi_1)$, $\psi(\xi_2)$ and $\psi(\xi_2\circ\xi_1)$ are also induced by $(\theta^1_j)$, $(\theta^2_j)$ and $(\theta^2_j\circ g^*\theta^1_j)$ resp. (thought of as isomorphisms between underlying vector bundles of $\psi(y)$ and $\psi(x)$, $\psi(z)$ and $\psi(y)$ , $\psi(z)$ and $\psi(x)$ resp.), functoriality follows. Therefore, $\psi$ is a stack morphism from $\GG^0$ to \ma.

By looking at the induced topological map $|\psi|$ (see \tb{Notation 15}), one can conclude that $\psi$ factorizes through $\GG_1$.
\end{proof}
We first prove a non-emptiness result for $\GG^0$.
\begin{prop}\label{nonempty3}
For every point $y\in|\GG_1|$, there exists $x\in|\GG^0|$ such that the induced map $|\psi|:|\GG^0|\to|\GG_1|$ sends $x$ to $y$.
\end{prop}
\begin{proof}
We construct a pre-image $((\mE_j,V_j),(\phi_j))$ for an $F$-valued object $((\mE'_j,V'_j),(\phi_j))$ of $\GG_1$, where $F$ is some algebraically closed extension of $K$. Let $\mE_j=\mE'_j(NP_{j+1})$ and take the same gluing data $(\phi_j)$. It then remains to define $V_j$. Recall that $V_j$ is supposed to be of rank $k+n=k+2m$.

Define $V_1:=\ker(\Gamma(\mE_1(-c_{k+n,1}P_2))\to\mE_1(-c_{k+n,1}P_2)|_{P_2})/\mL)$, where $\mL$ is a line bundle determined by
$$V'_1(-k_1P_1)\to\mE'_1(-a_{k,1}P_2)|_{P_2}=\mE_1(-c_{k,1}P_2)|_{P_2}\stackrel{\sim}{\to}\mE_1(-c_{k+n,1}P_2)|_{P_2}.\footnote{The last isomorphism is non-canonical, but $L$ is independent of the choice of such an isomorphism.}$$

For a general $j$, notice that among the vanishing orders $c_{k+1,2j-2},...,c_{k+n,2j-2}$, only $c_{k+1,2j-2}$, $c_{k+n,2j-2}$ could be non-repeated vanishing orders (see Appendix Lemma \ref{A1}, \ref{A2}). There are three situations:
\begin{enumerate}
  \item[(a)] $c_{k+1,2j-2}$, $c_{k+n,2j-2}$ are both repeated vanishing orders;
  \item[(b)] only $c_{k+n,2j-2}$ is non-repeated;
  \item[(c)] $c_{k+1,2j-2}$, $c_{k+n,2j-2}$ are both non-repeated.
\end{enumerate}

When $c_{k+1,2j-2}$ is non-repeated, define $\mL^{j-1}_1$ to be the image of the map: $$V_{j-1}(-c_{k+1,2j-3}P_j)\to\mE_{j-1}(-c_{k+1,2j-3}P_j)|_{P_j}.$$
In this case, define
$$W_1:=\ker(\Gamma(\mE_j(-c_{k+1,2j-2}P_j-c_{k+n,2j-1}P_{j+1}))\to\mE_j(-c_{k+1,2j-2}P_j)|_{P_j}/\phi_j(\mL^{j-1}_1)).$$

When $c_{k+n,2j-2}$ is non-repeated, define $\mL^{j-1}_2$ to be the image of the map: $$V_{j-1}\to\mE_{j-1}(-c_{k+n,2j-3}P_j)|_{P_j}.$$
In this case, define
$$W_2:=\ker(\Gamma(\mE_j(-c_{k+1,2j-2}P_j-c_{k+n,2j-1}P_{j+1}))\to\mE_j(-c_{k+n,2j-2}P_j)|_{P_j}/\phi_j(\mL^{j-1}_2)).$$
We now define $V_j$ inductively:\\
\noindent\tb{Case 1:} $\mE_j$ is semi-stable.\\
\ \\
(1) In situation (a), $V_j:=V'_j\op\Gamma(\mE_j(-c_{k+1,2j-2}P_j-c_{k+n,2j-1}P_{j+1}))$;\\
(2) in situation (b), $V_j:=V'_j+W_2$;\\
(3) in situation (c), $V_j:=V'_j+(W_1\bigcap W_2).$\\
\ \\
\noindent\tb{Case 2:} $\mE_j$ is unstable. Denote again $\mE_{j_1},...,\mE_{j_{2T}}$ to be the unstable bundles among $\mE_1,...,\mE_g$. Define $V''_j:=V'_j+\Gamma(\mE_j(-c_{k+1,2j-2}P_j-c_{k+n,2j-1}P_{j+1}))$.\\
\ \\
(1) Situation (a) occurs if and only if $j=j_{2t}$ and we define
  $$V_j:=\ker(V''_j\to\mE_j(-c_{k+n,2j-1}P_{j+1})|_{P_{j+1}}/\mL^j),$$
  where $\mL^j$ is a line in $\mE_j(-c_{k+n,2j-1}P_{j+1})|_{P_{j+1}}$ whose image $q_j$ in $\PP\mE_j|_{P_{j+1}}$\footnote{$\PP\mE_j|_{P_{j+1}}$ is canonically isomorphic to $\PP\mE_j(-c_{k+n,2j-1}P_{j+1})|_{P_{j+1}}$.} is glued to $q_{j_{2t+1}}$ in $c_{k+n,2j-1}$-th order, where $q_{j_{2t+1}}$ is the image of any section of the destabilizing summand of $\mE_{j_{2t+1}}$ in $\PP\mE_{j_{2t+1}}|_{P_{j_{2t+1}}}$;\\
(2) in situation (b), $V_j:=V''_j$;\\
(3) in situation (c), $V_j:=V'_j+\ker(V''_j(-c_{k+1,2j-2}P_j)\to\mE_j(-c_{k+1,2j-2}P_j)|_{P_j}/\phi_j(\mL^{j-1}_1))$.

Note that in case 2(1), the existence of such $\mL^j$ is guaranteed by the sufficient genericity of $((\mE'_j,V'_j),(\phi_j))$. \end{proof}
Based on the construction in Theorem \ref{Induction}, we point out a fact which will be relevant to the moduli count later:
\begin{cor}\label{cor1dim}
Given $\psi:\GG^0\to\GG_1$ as in Theorem \ref{Induction} and $F$ an algebraically closed extension of $K$, let $((\mE'_j,V'_j),(\phi_j))$ be an $F$-valued object in $\Ob(\GG_1)$. Then, an $((\mE_j,V_j),(\phi_j))$ in $\Ob(\GG^0)$ such that $\psi((\mE_j,V_j),(\phi_j))=((\mE'_j,V'_j),(\phi_j))$ is uniquely determined by the choice of $V_{j_{2T}}$. Moreover, the choice of $V_{j_{2T}}$ one-to-one corresponds to points in $\PP(\mE_{j_{2T}}|_{P_{j_{2T}+1}})\backslash\{q^*\}$, where $q^*$ is the image of sections of the destabilizing summand of $\mE_{j_{2T}}$.
\end{cor}
\begin{proof}
Following the same construction as in \ref{nonempty3}, we claim that the proposed $V_1$ is the only $k$-dimensional subspace of $\Gamma(\mE_1)$ satisfying the desired vanishing conditions which is compatible with $\phi_2$ and any feasible choice of $V_2$: given the proposed vanishing conditions, $V_1$ is uniquely determined by $q_1$, the image of $$V_1\to\mE_1(-c_{k+n,1}P_2)|_{P_2}\to\PP\mE_1(-c_{k+n,1}P_2)|_{P_2};$$
since $c_{k+n,2}+c_{k+n,3}={1\over2}(\deg(\mE_2)-1)$, $\phi_2(q_1)$ must be the image of sections of the destabilizing summand of $\mE_2$ in $\PP\mE_2|_{P_2}$.

Similarly, in cases 1(2), 1(3), 2(3) in \ref{Induction}, $V_j$ is determined by $V_{j-1}$, $V'_j$ and $\phi_j$.

In cases 1(1), 2(2), $V_j$ is uniquely determined by $V'_j$ and the vanishing conditions.

All $V_{j_{2t}}$ in 2(1) such that $t<T$ are uniquely determined by $V'_{j_{2t}}$, the vanishing conditions and $\phi_{j_{2t}+1}$,...,$\phi_{j_{2t+1}}$. Therefore, once $V_{j_{2T}}$ is chosen, all $V_j$ are determined and hence a unique object $((\mE_j,V_j),(\phi_j))$ is obtained.

For the second part of the statement, notice that $V_{j_{2T}}$ is determined by its $(P_{j_{2T}},P_{j_{2T}+1})$-adapted basis. Given $V'_{j_{2T}}$ and the vanishing conditions,\\ $V_{j_{2T}}(-(c_{k+n,2j_{2T}-1}+1)P_{j_{2T}+1})$ is uniquely determined and $V_{j_{2T}}$ depends only on the choice of a section $s$ such that $\ord_{P_{j_{2T}}}(s)=c_{k+n,2j_{2T}-2}, \ord_{P_{j_{2T}+1}}(s)=c_{k+n,2j_{2T}-1}.$ Such sections one-to-one corresponds to points in $\PP(\mE_{j_{2T}}|_{P_{j_{2T}+1}})\backslash\{q^*\}$.
\end{proof}
A key notion throughout the rest of the paper is the fiber dimension of morphisms between algebraic stacks. We quote the definition from \cite{Reldim}.
\begin{defn}\label{fiber}
Given $f:\mathcal{X}\to \mathcal{Y}$, a locally finite type morphism of algebraic
stacks. For $x\in|\mathcal{X}|$, let $y$ be a geometric point representing $f(x)$ and $\ti{x}$ a point of $\mathcal{X}\times_{\mathcal{Y}}y$ lying over $x$. We define the fiber dimension $\delta_xf$ to be the dimension at $\ti{x}$ of $\mathcal{X}\times_{\mathcal{Y}}y$.
\end{defn}
\begin{prop}\label{1dim}
Let $\psi:\mathcal{G}^0\to\GG_1$ be as defined in \ref{Induction}. Then, $\delta_x\psi=1$ for every $x$ of $|\GG^0|$.
\end{prop}
\begin{proof}
Suppose $|\psi|(x)=y$. Consider the fiber
 $$\xymatrix{\GG_y\ar[d]\ar[r]&\GG^0\ar[d]_{\psi}\\\spec(F)\ar[r]^{y}&\GG_1}$$

Let $((\mE'_j,V'_j),(\phi_j))$ be the object in $\GG_1$ represented by $y$. Denote $\mE_j=\mE'_j(NP_{j+1})$.

Objects in $\GG_y$ are triples of the form $(X,((\mF_j,W_j),(\chi_j)),(f_j))$: $X\stackrel{f}{\to}\spec(F)$ is an $F$-scheme, $((\mF_j,W_j),(\chi_j))$ an object in $\GG^0$ over $X$, and $(f_j)$ gives an isomorphism between the image of $X$, $((f^*\mE'_j,f^*V'_j),(f^*\phi_j))$, and $\psi((\mF_j,W_j),(\chi_j))$ in $\GG_1$. More concretely, $\psi((\mF_j,W_j),(\chi_j))$ is isomorphic to $((f^*\mE'_j,f^*V'_j),(f^*\phi_j))$ via $(f_1,...,f_g)$ if and only if the following holds:
$$\mF_j(-NP_{j+1})\underset{f_j}{\stackrel{\sim}{\to}}f^*\mE'_j, W'_j=f^*V'_j \text{ and }f^*\phi_j=f_{j+1}\circ\chi_j\circ f_j^{-1}\text{ for all }j.$$

We now describe a 1-dimensional scheme $G_y$ which represents $\GG_y$.

Let $\Gr(s,\Gamma(\mE_j);(a_{i,j}),(b_{i,j}))$ be the locally-closed sub-scheme of $\Gr(s,\Gamma(\mE_j))$\\
parametrizing subspaces of sections with vanishing sequences $(a_{i,j}),(b_{i,j})$ at $P_j$, $P_{j+1}$ respectively. Denote $$\mathbf{G_k}=\prod_{j=1}^g\Gr(k,\Gamma(\mE'_j);(a_{i,2j-2}),(a_{i,2j-1})),\mathbf{G_{k+n}}=\prod_{j=1}^g\Gr(k+n,\Gamma(\mE_j);(c_{i,2j-2}),(c_{i,2j-1})).$$

The morphism $\spec(F)\stackrel{y}{\to}\GG_1$ induces a morphism $\spec(F)\to\mathcal{M}_{2,\vv{d},\omega_g(NP_{g+1})}$ and we have the following (2-)fibered diagram:
$$\xymatrix{\ti{G}_y\ar[r]\ar[d]&\mathcal{G}^{k+n,\text{EHT}}_{2,\omega_g(NP_{g+1}),\hat{d}_{\bullet},c^{\Gamma}}(X_g)\ar[d]^{\pi}\ar[r]&\mathcal{G}^{k+n,\text{EHT}}_{2,2g-2+2N,\hat{d}_{\bullet},c^{\Gamma}}(X_g)\ar[d]\\
\spec(F)\ar[r]&\mathcal{M}_{2,\vv{d},\omega_g(NP_{g+1})}\ar[r]&\mathcal{M}_{2,\vv{d}}}$$
By Remark \ref{p} and the fact that $\mathcal{G}^{k+n,\text{EHT}}_{2,2g-2+2N,\hat{d}_{\bullet},c^{\Gamma}}(X_g)$ is a locally closed substack of $\mathcal{P}^{k}_{2,d_{\bullet}}$, it follows that $\ti{G}_y$ is a locally closed sub-scheme of $\mathbf{G_{k+n}}$. Since $\GG^0$ is locally-closed in $\mathcal{G}^{k+n,\text{EHT}}_{2,\omega_g(NP_{g+1}),\hat{d}_{\bullet},c^{\Gamma}}(X_g)$, we further conclude that the fiber $G^0_y$ of $\GG^0$ over $\mathcal{M}_{2,\vv{d},\omega_g(NP_{g+1})}$ is also locally-closed in $\mathbf{G_{k+n}}$.

Similarly, the fiber $G^1_y$ of $\GG_1$ over $\mathcal{M}_{2,\vv{d},\omega_g(NP_{g+1})}$ is locally-closed in $\mathbf{G_k}$.

The morphism $\psi$ induces a map between the fibers $\psi_y:G^0_y\to G^1_y$. Define $G_y=\psi_y^{-1}(V'_1,...,V'_g)$ in $G^0_y$. This is equivalent to saying $((\mE_j,V_j),(\phi_j))$ is an object in $\GG^0$ over the object determined by $y$. We claim $G_y$ is 1-dimensional.

Consider the morphism
$$G_y\to \Gr(k+n,\Gamma(\mE_{j_{2T}});(c_{i,2j_{2T}-2}),(c_{i,2j_{2T}-1}))\to\PP\mE_{j_{2T}}(-c_{k+n,2j_{2T}}P_{j_{2T}+1})|_{P_{j_{2T}+1}}$$
where the first arrow is a projection and the second one is defined by sending $V_{j_{2T}}$ to the line determined by $V_{j_{2T}}/i(V'_{j_{2T}})$, where $i$ is the inclusion $\Gamma(\mE'_{j_{2T}})\to\Gamma(\mE_{j_{2T}})$. By Corollary \ref{cor1dim}, this morphism maps injectively onto a dense open subset of the target. Hence, $G_y$ is 1-dimensional.

That $G_y$ represents $\GG_y$ follows from the following commutative 2-diagram:
$$\xymatrix{\GG^0\times_{\cm}\spec(F)\times_{(\GG_1\times_{\cm}\spec(F))}\spec(F)\ar[r]\ar[d]&\GG^0\times_{\cm}\spec(F)\ar[r]\ar[d]&\GG^0\ar[d]\\\spec(F)\ar[r]&\GG_1\times_{\cm}\spec(F)\ar[r]&\GG_1}$$
The right square is Cartesian, due to the canonical isomorphism $$\spec(F)\times_{\cm}\GG_1\times_{\GG_1}\GG^0=\GG^0\times_{\cm}\spec(F)$$
(see Tag 02XD, \cite{stacks-project}). Hence, the whole diagram is Cartesian. In particular,\\ $\GG^0\times_{\cm}\spec(F)\times_{(\GG_1\times_{\cm}\spec(F))}\spec(F)$ is isomorphic to $\GG^0\times_{\GG_1}\spec(F)$. So,
$$G_y=G^0_y\times_{G^1_y}\spec(F)=\GG^0\times_{\cm}\spec(F)\times_{(\GG_1\times_{\cm}\spec(F))}\spec(F)\cong\GG^0\times_{\GG_1}\spec(F)=\GG_y.$$
\end{proof}
\subsection{Geometric Fibers of Certain Forgetful Morphisms}\label{forget}
Hereafter, we focus on morphisms between moduli stacks of Eisenbud-Harris-Teixidor limit linear series on chain $X_g$ for different $g$'s. In particular, given $g<g'$, there is a natural forgetful morphism $\mathcal{G}^{k',\text{EHT}}_{2,\omega_{g'},d^{g'}_{\bullet},c^{\Gamma}_{g'}}(X_{g'})\to\mathcal{G}_g$ (\tb{Notation 17}), defined by forgetting the part of data over the last $g'-g$ components.

\indent In general, when $g_1>g_2>g_3$, we have the following commutative diagram:
$$\begin{tikzcd}
\GG_{g_1} \arrow{dr} \arrow{r}{p_{g_1,g_2}}&\GG_{g_2}\arrow{d}\\
&\GG_{g_3}
\end{tikzcd}$$

We shall study the dimension of the fiber of $p_{g_1,g_2}$ over a geometric point represented by $x:\spec(F)\to\GG_2$, i.e. the fibered product
$$\xymatrix{\GG^{g_1,g_2}_x\ar[r]\ar[d]&\GG_1\ar[d]^{p_{g_1,g_2}}\\ \spec(F)\ar[r]^x&\GG_2}$$

Recall that the moduli stacks $\GG_1,\GG_2$ here considered satisfy the property that all $K$-valued objects have the same underlying vector bundles.

Suppose $x$ represents the object $((\mE_j,V_j),(\phi_j))$ in $\GG_2$. Objects in $\GG^{g_1,g_2}_x$ over an $F$-scheme $X\stackrel{p}{\to}\spec(F)$ are of the form $$(X,((\mF_j,W_j),(\chi_j)),(f_1,...,f_{g_2})),$$
where $((\mF_j,W_j),(\chi_j))$ is an object in $\GG_1$ over $X$ such that $p_{g_1,g_2}((\mF_j,W_j),(\chi_j))$ is isomorphic to $((p^*\mE_j,p^*V_j),(p^*\phi_j))$ via $f=(f_1,...,f_{g_2})$ where $f_j$ is an isomorphism between $\mF_j(-(g_1-g_2)P_{j+1})$ and $p^*\mE_j$.

We shall now describe a smooth cover of $\GG^{g_1,g_2}_x$.

Denote $\tau_{P_{g_2+1}}:=\tau^{V_{g_2}}(P_{g_2+1})$. Define
$$W=\dis\prod_{j=g_2+1}^{g_1} \Gr(k,\Gamma(\mE_j);(c_{i,2j-2}),(c_{i,2j-1}))\times\dis\prod_{j=g_2+1}^{g_1} \iso(\mE_j|_{P_{j+1}},\mE_{j+1}|_{P_{j+1}}).$$

Let $N_j$ be the number of non-repeated vanishing orders in $(c_{i,2j-2})$. Since we restrict ourselves to refined limit linear series, $N_j$ equals the number of non-repeated vanishing orders in $(c_{i,2j-3})$. Define a morphism
$$g:W\to \prod_{j=g_2+1}^{g_1}((\PP \mE_j|_{P_j})^{N_j}\times(\PP \mE_j|_{P_j})^{N_j})$$ sending $((V_j),(\phi_j))$ to $((q^{g_2},p^{g_2+1}),...,(q^{g_1-1},p^{g_1}))$, where
$$p^j:=(s^j_1|_{P_j},...,s^j_{N_j}|_{P_j}),q^j:=(\phi_{j+1}(s^j_1|_{P_{j+1}}),...,\phi_{j+1}(s^j_{N_j}|_{P_{j+1}}))$$ and $s^j_i\in V_j$ is section realizing the $i$-th non-repeated vanishing order at the corresponding point. Denote $\Delta_j$ to be the diagonal of $(\PP \mE_j|_{P_j})^{N_j}\times(\PP \mE_j|_{P_j})^{N_j}$ and define $\Delta=\prod_{j=g_2+1}^{g_1}\Delta_j$.

Define $G_x$ to be the closed sub-scheme of $g^{-1}(\Delta)$ parametrizing objects such that\\ $\tau^{V_{g_2+1}}(P_{g_2+1})=\tau_{P_{g_2+1}}$.
\begin{prop}\label{fiber2}
$\exists u:G_x\to\GG^{g_1,g_2}_x$, which is a $\prod_{j=g_2+1}^{g_1}\aut^0(\mE_j)$-torsor.
\end{prop}
\begin{proof}
Let $F$ be an algebraically closed extension of $K$.

By Appendix \ref{cl2}, the underlying vector bundles of an $F$-valued object in $\GG_1$ (resp. $\GG_2$) corresponds to a locally closed point inside $\prod_{j=1}^{g_1} \mathcal{M}_{2,\mL_j}(C_j)$ (resp. $\prod_{j=1}^{g_2} \mathcal{M}_{2,\mL_j}(C_j)$), where $\mathcal{M}_{2,\mL_j}(C_j)$ denotes the moduli stack of rank two vector bundles on $C_j$ with some fixed determinant $\mL_j\cong\det(\mE_j)$.

Denote $\mE_1,...,\mE_{g_2}$ to be the underlying vector bundles $\GG_1$ (See Remark \ref{uopenr}). Let $\mathcal{Z}$ be the (reduced) locally closed substack of $\prod_{j=g_2+1}^{g_1} \mathcal{M}_{2,\mL_j}(C_j)$, whose underlying topological space is a singleton corresponding to the (fixed) vector bundles $\mE_{g_2+1},...,\mE_{g_1}$.

Since we are only interested in dimensions of the moduli stacks, by the basic theory of stack dimension, we may pass to the induced reduced substack whenever necessary (see Tag 0509, \cite{stacks-project}). Thus, one has the following 2-fibered diagram:

$$\xymatrix{\GG^{g_1,g_2}_x\times_{\mathcal{Z}}\spec(K)\ar[rr]\ar[d]&&\spec(K)\ar[d]\\\GG^{g_1,g_2}_x\ar[r]&\GG_1\ar[r]&\mathcal{Z}}$$
where the bottom left arrow is the natural projection and the bottom right arrow is an obvious forgetful morphism.

We show that $G_x$ is isomorphic to $\GG^{g_1,g_2}_x\times_{\mathcal{Z}}\spec(K)$. An object in the fibre product is determined by three data: an object $((\mF_j,W_j),(\chi_j))$ of $\GG_1$; a limit linear series isomorphism $h$ between $((\mF_j,W_j)_{j\le g_2},(\chi_j)_{j\le g_2})$ and (appropriate pull-back of) the data determined by $x$; isomorphisms $f_j$ between $\mF_j$ and (appropriate pull-back of) $\mE_j$ for $j=g_2+1,...,g_1$. Define a 1-morphism $\theta:G_x\to\GG^{g_1,g_2}_x\times_{\mathcal{Z}}\spec(K)$ as follows: given $((V_j)_{j=g_2+1}^{g_1},(\phi_j)_{j=g_2+1}^{g_1})$, one naturally gets an object $((\mE_j,V_j)_{j\ge 1},(\phi_j)_{j\ge2})$ of $\GG_1$ by using the data specified by $x$. $\theta$ then sends $((V_j)_{j=g_2+1}^{g_1},(\phi_j)_{j=g_2+1}^{g_1})$ to the object in the fibre product determined by this object of $\GG_1$ together with identity bundle morphisms. For morphisms, $\theta$ sends any $G_x$-morphism $g$ to an arrow between corresponding objects of the fibre product, one of which is identified with the pull-back of the other via $g$. In particular, there is a 1-1 correspondence between morphisms in $\h_{G_x}(S,T)$ and morphisms from $\theta(S)$ to $\theta(T)$ and $\theta$ is fully-faithful.

It remains to show $\theta$ is essentially surjective. Consider an object $\xi$ of $\GG^{g_1,g_2}_x\times_{\mathcal{Z}}\spec(K)$. $\xi$ specifies an object $((\mF_j,W_j),(\chi_j))$ of $\GG_1$; a limit linear series isomorphism $h$ induced by bundle isomorphisms $h_1,...,h_{g_2}$; and bundles isomorphisms $f_{g_2+1},...,f_{g_1}$. The data\\ $((f_j(W_j))_{j>g_2},(f_{j+1}\circ\chi_j\circ f_j^{-1})_{j>g_2})$ gives an object $\sigma$ of (the associated stack of) $G_x$. Then, $\xi$ is isomorphic to the image of $\theta(\sigma)$; the isomorphism is induced by $h_1,...,h_{g_2},f_{g_2+1},...,f_{g_1}$.

We define $u$ to be the morphism $G_x\stackrel{\sim}{\to}\GG^{g_1,g_2}_x\times_{\mathcal{Z}}\spec(K)\to\GG^{g_1,g_2}_x$. The fact that $u$ is a $\prod_{j=g_2+1}^{g_1}\aut^0(\mE_j)$-torsor follows from the fact that $\spec(K)\to\mathcal{Z}$ is a $\prod_{j=g_2+1}^{g_1}\aut^0(\mE_j)$-torsor.
\end{proof}
\begin{cor}\label{fdim}
$\dim\GG^{g_1,g_2}_x=\dim G_x-\dis\sum_{j=g_2+1}^{g_1}\dim\aut^0(\mE_j)$.
\end{cor}
\subsection{Induction Step 2}
Given some non-empty stack $\mathcal{G}^{k,\text{EHT}}_{2,\omega_g,d_{\bullet},a^{\Gamma}}(X_g)$, where $k=2k_1+1$ is odd, by Proposition \ref{nonempty3}, we obtain a non-empty substack $\GG^0$ of $\mathcal{G}^{k+4,\text{EHT}}_{2,\omega_g(NP_{g+1}),\hat{d}_{\bullet},c^{\Gamma}}(X_g)$, whose objects all have vanishing sequence
\begin{equation}\label{gvan}
  (c_i):=(N+k_1,[N+k_1-1]_2,...,[N]_2,N-1-k_1,[N-2-k_1]_2,N-3-k_1)
\end{equation}
along $P_{g+1}$. Denote $\ti{c}^{\Gamma}:=[c^{\Gamma}|c_i]$, which is a $(k+4)\times(2g-1)$-matrix.

Now we shall define another moduli stack $\mathcal{G}^{k+4,\text{EHT}}_{2,\omega_{g'},d'_{\bullet},c^{\Gamma'}_{g'}}(X_{g'})$ of rank two limit linear series on a chain of $g'$ elliptic curves, $X_{g'}$, where $g'>g$. Technically, we need to specify a $(k+4)\times (2g'-2)$ matrix $c^{\Gamma'}_{g'}$, which is $(g',k+4)$-standard. Specifically, we setup the following notations/definitions:
\begin{enumerate}
  \item Denote $c^{\Gamma'}_{g'}:=(c_{ij})$, where $(c_{ij})_{j\le2g-1}=\ti{c}^{\Gamma}$.
  \item Fix $k=2k_1+1$, take $q=\max\{1,\lfloor{k_1\over3}\rfloor\}$ and define $N=4k_1+6-q$.
  \item Define $g':=g+N$.
  \item Denote $k':=k+4$.
  \item Define $b'=g'-1$ which satisfies the equation $\sum_{j=1}^{g'}d'_j-2b'(g'-1)=2g'-2$.
\end{enumerate}

\textbf{It turns out that one does not need to write down $c^{\Gamma'}_{g'}$ explicitly.\ }\rm To reduce complexity in describing our constructions, we first introduce a simplification for recording vanishing conditions. Since the vanishing sequences $(c_{i,2j-2}),(c_{i,2j-1})$ either satisfy conditions in \ref{one}, or in \ref{one2} and the vanishing at $P_1$ is always $([0]_2,...,[k_1+m-1]_2,k_1+m)$, it suffices to record all the underlying vector bundles (together with some extra data). We summarize this idea in the following two lemmas:
\begin{lem}\label{concise1}
  Fix an elliptic curve $C$, two general points $P,Q\in C$ and a positive integer $d$. Suppose $(a_i),(b_i)$ are integer sequences satisfying Lemma \ref{one} with respect to the given $d$. Then $(a_i),(b_i)$ are determined by the pair $(\mE,(a_i))$, where $\mE$ is a rank two vector bundle on $C$ of the form $\mO(a_{i_1},d-a_{i_1})\op \mO(a_{i_2},d-a_{i_2})$ (\textbf{Notation 8}).
\end{lem}
\begin{rem}
Hereafter, we shall refer to such $i_1,i_2$ as \textit{special indices} of $(a_i)$, $(b_i)$, when they are in the situation of \ref{one}.
\end{rem}
\begin{proof}
  To give $\mE$ is equivalent to specifying the special indices $i_1,i_2$. Given $i_1,i_2$ and $(a_i)$, $(b_i)$ is determined by the following rule: $b_i=\begin{cases}
    d-a_i  &\text{ for }i=i_1,i_2\\
    d-a_i-1&\text{ otherwise }.
  \end{cases}$
\end{proof}
For $\mE_j$ unstable, we have a similar result:
\begin{lem}\label{concise2}
Let $(a_i),(b_i)$ be integer sequences of length $k$ satisfying Lemma \ref{one2} with respect to some given $d$. Then $(a_i),(b_i)$ are determined by a triple $(\mE,(a_i),\tau_*)$, where $\mE$ is a rank two vector bundle of the form $\mO(a_{\ell},d+1-a_{\ell})\op \mO(a_{i^*},d-a_{i^*})$ and $\tau_*$ is some subset of $\{1,...,k\}$ such that $\fa i\in\tau_*$, $a_i$ is a non-repeated entry in $(a_i)$.
\end{lem}
\begin{rem}\label{spid}
  We also refer to $\ell$, $i^*$ as the special indices of $(a_i)$, $(b_i)$, when they are in the situation of \ref{one2}.
\end{rem}
\begin{proof}
Take $\tau_*$ to be $\{i\in\{1,...,k\}|a_i \text{ is a non-repeated term in }(a_i), a_i+b_i=d\}$.

Given $\mE$, $(a_i)$ and $\tau_*$, $(b_i)$ is determined by the following rule:
$$b_i=\begin{cases}
    d+1-a_i  &\text{ if }i=\ell\\
    d-a_i    &\text{ if }i\in\tau_*\text{ or }i:a_i=a_{i+1}\\
    d-a_i    &\text{ if }i=i^*+1\\
    d-1-a_i  &\text{ otherwise.}
\end{cases}$$
\end{proof}
By Lemma \ref{concise1}, \ref{concise2}, when describing our inductive construction it suffices to specify the vector bundles $\mE_j$, together with a set of indices $\tau_{j_t}$ for each unstable bundle $\mE_{j_t}$. Since $\mE_1,...,\mE_g$ are determined from the preceding step of induction, it suffices to do so for $j>g$.
\subsubsection{The First Construction}\label{l1}\hspace*{\fill} \\
The following is a list of vector bundles and configurations for our first construction: \\
\ \\
\framebox[1\width]{\ $\mE_{g+1},\mE_{g+2},\mE_{g+3}$ are semi-stable:\ }
\begin{enumerate}
  \item $\mE_{g+1}=\mO(g-k_1,b'-(g-k_1))\op\mO(g+k_1,b'-(g+k_1))$;
  \item $\mE_{g+2}=\mO(g-k_1,b'-(g-k_1))\op\mO(g+k_1+2,b'-(g+k_1+2))$;
  \item $\mE_{g+3}=\mO(g-k_1+1,b'-(g-k_1+1))\op\mO(g+k_1+3,b'-(g+k_1+3))$.
\end{enumerate}
\framebox[1\width]{\ $\mE_{g+4}$ is unstable:\ }
$$\mE_{g+4}=\mO(g-k_1+3,5k_1-q+3)\op \mO(g+k_1+3,3k_1-q+2), \tau_{g+4}=\{k+1,k+4\}.$$
\framebox[1\width]{\ $\mE_{g+5},\mE_{g+6},\mE_{g+7}$ are semi-stable:\ }
\begin{enumerate}
  \item $\mE_{g+5}=\mO(g-k_1+4,b'-(g-k_1+4))\op\mO(g+k_1+3,b'-(g+k_1+3))$;
  \item $\mE_{g+6}=\mO(g-k_1+6,b'-(g-k_1+6))\op\mO(g+k_1+3,b'-(g+k_1+3))$;
  \item $\mE_{g+7}=\mO(g-k_1+6,b'-(g-k_1+6))\op\mO(g+k_1+5,b'-(g+k_1+3))$.
\end{enumerate}

\noindent\framebox[1\width]{\ $\mE_{g+8}$ is unstable:\ }
$$\mE_{g+8}=\mO(g-k_1+5,b'-(g-k_1+6))\op\mO(g+k_1+8,b'-(g+k_1+8)), \tau_{g+8}=\{1,k\}.$$
\noindent\framebox[1\width]{\ $\mE_{g+9}$ is semi-stable:\ }
$$\mE_{g+9}=\mO(g-k_1+8,b'-(g-k_1+7))\op\mO(g+k_1+8,b'-(g+k_1+8)).$$

\noindent\framebox[1\width]{\ For $j=g+9+11s+t$ ($0\le s\le q-2$, $1\le t\le11$), $\mE_j$ are semi-stable except for $t=1,4$:\ }\\
\ \\
(1) $\mE_{g+10+11s}=\mO(g-k_1+11+14s,b'-(g-k_1+11+14s))\op\mO(g+k_1+7+8s,b'-(g+k_1+6+8s))$,\\
\hspace*{4 mm} $\tau_{g+10+11s}=\{1\}$;\\
(2) $\mE_{g+11+11s}=\mO(g-k_1+11+14s,b'-(g-k_1+11+14s))\op\mO(g+k_1+8+8s,b'-(g+k_1+8+8s))$;\\
(3) $\mE_{g+12+11s}=\mO(g-k_1+12+14s,b'-(g-k_1+12+14s))\op\mO(g+k_1+9+8s,b'-(g+k_1+9+8s))$;\\
(4) $\mE_{g+13+11s}=\mO(g-k_1+14+14s,b'-(g-k_1+14+14s))\op\mO(g+k_1+9+8s,b'-(g+k_1+10+8s))$,\\
\hspace*{4 mm} $\tau_{g+13+11s}=\{5+6s,k,k+1,k+4\}$;\\
\noindent(5) $\mE_{g+14+11s}=\mO(g-k_1+16+14s,b'-(g-k_1+16+14s))\op\mO(g+k_1+10+8s,b'-(g+k_1+10+8s))$;\\
(6) $\mE_{g+15+11s}=\mO(g-k_1+15+14s,b'-(g-k_1+15+14s))\op\mO(g+k_1+13+8s,b'-(g+k_1+13+8s))$;\\
(7) $\mE_{g+16+11s}=\mO(g-k_1+19+14s,b'-(g-k_1+19+14s))\op\mO(g+k_1+11+8s,b'-(g+k_1+11+8s))$;\\
(8) $\mE_{g+17+11s}=\mO(g-k_1+20+14s,b'-(g-k_1+20+14s))\op\mO(g+k_1+12+8s,b'-(g+k_1+12+8s))$;\\
(9) $\mE_{g+18+11s}=\mO(g-k_1+19+14s,b'-(g-k_1+19+14s))\op\mO(g+k_1+15+8s,b'-(g+k_1+15+8s))$;\\
(10) $\mE_{g+19+11s}=\mO(g-k_1+21+14s,b'-(g-k_1+21+14s))\op\mO(g+k_1+15+8s,b'-(g+k_1+15+8s))$;\\
(11) $\mE_{g+20+11s}=\mO(g-k_1+22+14s,b'-(g-k_1+22+14s))\op\mO(g+k_1+16+8s,b'-(g+k_1+16+8s))$;\\

\noindent\framebox[1\width]{\ For $j>g+11q-2$, $\mE_j$ are semi-stable:\ }\\
\ \\
\indent To shorten symbols, write $h=g+11q-2$ and consider all $\ell=0,...,k_1-3q$:\\
\ \\
(1) $\mE_{h+4\ell+1}=\mO(g-k_1+14q+5\ell-3,b'-(g-k_1+14q+5\ell-3))\op$

\hspace*{7cm}$\mO(g+k_1+8q+3\ell-1,b'-(g+k_1+8q+3\ell+1))$;

\noindent(2) $\mE_{h+4\ell+2}=\mO(g-k_1+14q+5\ell-2,b'-(g-k_1+14q+5\ell-2))\op \mO(g+k_1+8q+3\ell,b'-(g+k_1+8q+3\ell))$;\\
\ \\
(3) $\mE_{h+4\ell+3}=\mO(g-k_1+14q+5\ell-2,b'-(g-k_1+14q+5\ell-2))\op$\\
\hspace*{7cm}$\mO(g+k_1+8q+3\ell+2,b'-(g+k_1+8q+3\ell+2))$;\\
\ \\
(4) $\mE_{h+4\ell+4}=\mO(g-k_1+14q+5\ell-1,b'-(g-k_1+14q+5\ell-1))\op$\\
\hspace*{7cm}$\mO(g+k_1+8q+3\ell+3,b'-(g+k_1+8q+3\ell+3))$.
\noindent\framebox[1\width]{\ $\mE_{g'-3}$,...,$\mE_{g'}$ are semi-stable:\ }
\begin{enumerate}
  \item $\mE_{g'-3}=\mO(g'-4,3)^{\op 2}$;
  \item $\mE_{g'-2}=\mO(g'-4,3)\op\mO(g'-2,1)$;
  \item $\mE_{g'-1}=\mO(g'-3,2)\op\mO(g'-1,0)$;
  \item $\mE_{g'}=\mO(g'-1,0)^{\op 2}$.
\end{enumerate}
\subsubsection{Vanishing Sequences in the First Construction}
For future reference, we list certain vanishing sequences in our first construction:
\begin{lem}\label{parta}We have\\
$(c_{i,2g+5})^T=([5k_1+4-q]_2,5k_1+2-q,\underbrace{[5k_1+1-q]_2,...,[4k_1+3-q]_2}_{2k_1-2 \text{ entries}},$\\
\hspace*{8cm}$3k_1+3-q,[3k_1+2-q]_2,3k_1-q)$;\\
$(c_{i,2g+7})^T=(5k_1+4-q,[5k_1+3-q]_2,5k_1+1-q,\underbrace{[5k_1-q]_2,...,[4k_1+3-q]_2}_{2k_1-4 \text{ entries}},$\\
\hspace*{6cm}$4k_1+2-q,3k_1+3-q,[3k_1+2-q]_2,3k_1-q)$;\\
$(c_{i,2g+13})^T=(5k_1+1-q,[5k_1-q]_2,[5k_1-1-q]_2,5k_1-3-q$,\\
\hspace*{1.3cm}$\underbrace{[5k_1-4-q]_2,...,[4k_1-q]_2,4k_1-1-q}_{2k_1-5 \text{ entries}},3k_1+1-q,[3k_1-q]_2,3k_1-3-q)$;\\
$(c_{i,2g+15})^T=(5k_1-q,[5k_1-1-q]_2,5k_1-2-q,5k_1-3-q$,\\
\hspace*{1.3cm}$\underbrace{[5k_1-5-q]_2,...,[4k_1-2-q]_2}_{2k_1-4 \text{ entries}},[3k_1-1-q]_2,3k_1-2-q,3k_1-3-q)$;\\
$(c_{i,2g+17})^T=(5k_1-1-q,[5k_1-2-q]_2,[5k_1-3-q]_2$,\\
\hspace*{1.3cm}$\underbrace{[5k_1-6-q]_2,[5k_1-7-q]_2,...,
  [4k_1-3-q]_2}_{2k_1-4 \text{ entries}},[3k_1-2-q]_2,[3k_1-3-q]_2)$.\end{lem}
\begin{lem}\label{parte}
For $s=0,...,q-1$, we have:\\
$(c_{i,2g+17+22s})^T=(5k_1-1-q-11s,\underbrace{[5k_1-2-q-11s]_2,...,[5k_1-3-q-14s]_2}_{4+6s\text{ entries}}$,\\
$\underbrace{[5k_1-6-q-14s]_2,...,
  [4k_1-3-q-11s]_2}_{2k_1-4-6s \text{ entries}},[3k_1-2-q-8s]_2,[3k_1-3-q-8s]_2).$\\
$(c_{i,2g+19+22s})^T=(5k_1-1-q-11s,5k_1-2-q-11s,$\\
$\underbrace{[5k_1-3-q-11s]_2,...,[5k_1-3-q-14s]_2}_{2+6s\text{ entries}}$, $5k_1-4-q-14s$,\\
$[5k_1-6-q-14s]_2,5k_1-7-q-14s,\underbrace{[5k_1-8-q-14s]_2,...,
  [4k_1-3-q-11s]_2}_{2k_1-8-6s \text{ entries}},$\\
$4k_1-4-q-11s,3k_1-1-q-8s,[3k_1-3-q-8s]_2,3k_1-4-q-8s).$\\
$(c_{i,2g+25+22s})^T=(5k_1-5-q-11s,\underbrace{[5k_1-6-q-11s]_2,...,[5k_1-7-q-14s]_2,}_{4+6s\text{ entries}}$\\ $5k_1-8-q-14s,[5k_1-9-q-14s]_2,5k_1-11-q-14s,\underbrace{[5k_1-12-q-14s]_2,....,[4k_1-7-q-11s]_2,}_{2k_1-8-6s \text{ entries}}$ $3k_1-4-q-8s, [3k_1-5-q-8s]_2,3k_1-7-q-8s).$
\end{lem}
\begin{lem}\label{partf}
We have for $\ell=0,...,k_1-3q+1$,\\
$(c_{i,2h+8\ell-1})^T=(5k_1-12q-4\ell+10,\underbrace{[5k_1-12q-4\ell+9],...,[5k_1-15q-5\ell+11]_2}_{6q+2\ell-2\text{ entries}}$,\\
$\underbrace{[5k_1-15q-5\ell+8]_2,...,[4k_1-12q-4\ell+8]_2}_{2k_1-6q-2\ell+2\text{ entries}},[3k_1-9q+6-3\ell]_2,[3k_1-9q+5-3\ell]_2)$.
\end{lem}
\subsubsection{Fiber Dimensions}
Our next goal is to give an upper bound to the dimension of (an open substack of)\\ $\mathcal{G}^{k',\text{EHT}}_{2,\omega_{g'},d^{g'}_{\bullet},c^{\Gamma'}_{g'}}(X_{g'})$ defined previously. We show that the dimension is bounded from above by $\rho_{g',k'}$.

First of all, we establish some simple counting results on fiber dimensions of forgetful morphisms $p_{r+1,r}$ in cases where $\mE_{r+1}$ is semi-stable and decomposable.

Fix $\GG=\mathcal{G}^{k',\text{EHT}}_{2,\omega_{g'},d^{g'}_{\bullet},c^{\Gamma'}_{g'}}(X_{g'})$ and let $\GG_r=\GG^{k',\text{EHT}}_{2,\oo_r(A_rP_{r+1}),d^r_{\bullet},c^{\Gamma'}_r}(X_r)$ (\tb{Notation 17}).
\begin{lem}\label{m1}
Consider the forgetful morphism $p_{r+1,r}:\GG_{r+1}\to\GG_r$. Suppose $c^{\Gamma'}$ is $(g',k')$-standard. Suppose also $\mE_{r+1}=\mL\op \mL'$ is semi-stable and $\mL\not\cong\mL'$.

Denote $i_1,i_2$ to be the special indices of $(c_{i,2r})$, $(c_{i,2r+1})$ and $\tau(P_{r+1})$ to be the configuration at $P_{r+1}$. Also introduce the following two index subsets:
\begin{enumerate}
  \item $\mA_1=\{S_t\in\tau(P_{r+1})|S_t\ni i_1\text{ or }i_2\}$;
  \item $\mA_2=\{S_t\in\tau(P_{r+1})|S_t\ni i_s-1\neq i_1,i_2 \text{ and } c_{i_s-1,2r}=c_{i_s,2r}-1,\text{ for }s=1,2\}$.
\end{enumerate}

Then, the fiber dimension $m_{r+1}$ of $p_{r+1,r}$ at any point is less than $2-|\mA_1\bigcup\mA_2|$.

\end{lem}
\begin{proof}
  Let $x$ be any geometric point of $\GG_r$ such that the fiber of $p_{r+1,r}$ at $x$, $\GG_x$, is non-empty.

 Since $\mE_{r+1}=\mL\op\mL'$ is semistable and $\mL\not\cong\mL'$, $\dim \aut^0(\mE_{r+1})=1$.

  By Corollary \ref{fdim}, $m_{r+1}= \dim(G_x)-\dim (\aut^0(\mE_{r+1}))=\dim(G_x)-1$, where $G_x$ is as given in Proposition \ref{fiber2}. Hence, it only remains to compute $\dim(G_x)$.

  Consider the configuration $\tau(P_{r+1})$ associated to $x$. Let $(q_t)$ be the sequence of points in $\PP\mE_r|_{P_{r+1}}$ associated to $V_r$. $V_{r+1},\phi_{r+1}$ are compatible with the configuration if and only if $\tau^{V_{r+1}}(P_{r+1})=\tau(P_{r+1})$ and $\phi_{r+1}(q_t)=p_t$, where $(p_t)$ is the sequence of points in $\PP \mE_{r+1}|_{P_{r+1}}$ associated to $V_{r+1}$.

  Thus, $\dim(G_x)$ depends on the configuration $\tau(P_{r+1})$. More precisely, we have the following observations:

  First of all, if some subset $S_t$ in the configuration contains $i_1$ (or $i_2$), the associated point $p_t$ in $\PP\mE_{r+1}|_{P_{r+1}}$ is the image of one of the canonical sections of $\mE_{r+1}$. If $S_i$ contains $i_1-1$ (or $i_2-1$) and $c_{i_1-1,2r}=c_{i_1,2r}-1$, same is true. In either case, the condition $\phi_{r+1}(p_t)=q_t$ cuts out a codimension 1 sub-scheme inside $\iso(\mE_r|_{P_{r+1}},\mE_{r+1}|_{P_{r+1}})$. Meanwhile, for all other indices $i$ not in the same $S_t$, given any point $q_i$ in $\PP\mE_{r+1}|_{P_{r+1}}$, one can always find (up to scalar) a unique section vanishing to order $c_{i,2r},c_{i,2r+1}$ at $P_{r+1},P_{r+2}$ resp.; hence, $\fa q_i$, $\exists (V_{r+1},\phi_{r+1})$ such that $q_i$ is in the sequence of points in $\PP\mE_{r+1}|_{P_{r+1}}$ associated to $V_{r+1}$.

  Secondly, $V_{r+1}$ is uniquely determined by the associated sequence of points $(p_t)$.

  Thirdly, denote $f:G_x\to\iso(\mE_r|_{P_{r+1}},\mE_{r+1}|_{P_{r+1}})$ to be the obvious projection. Any non-empty fiber of $f$ is zero-dimensional, since $\phi_{r+1}$ uniquely determines $(p_t)$, and hence determines $V_{r+1}$.

  Due to the fixed determinant condition, $\dim\iso(\mE_r|_{P_{r+1}},\mE_{r+1}|_{P_{r+1}})=3$. Thus,
  $$\dim(G_x)=\dim \im(f)\le3-|\mA_1\bigcup\mA_2|.$$
  Since the fiber dimension $m_{r+1}=\dim(G_x)-1$, the lemma follows.
\end{proof}
\begin{rem}\label{often}
  A frequently encountered situation of $\tau(P_{r+1})$ is the one in Example \ref{config} (4). In that case, $m_{r+1}\le0$. To detect this situation, it suffices to look at the vector bundles $\mE_r,\mE_{r+1}$.
\end{rem}
Very similarly, we also state the following lemma:
\begin{lem}\label{m2}
  Same setup as in Lemma \ref{m1}, except that now assume $\mE_{r+1}=\mL^{\op 2}$. Then, $m_{r+1}\le0$.
\end{lem}
\begin{proof}
Notice that in this case $\dim\aut^0(\mE_{r+1})=3$. Again, the fibers of the projection $f:G_x\to\iso(\mE_r|_{P_{r+1}},\mE_{r+1}|_{P_{r+1}})$ is zero-dimensional. Since $m_{r+1}=\dim(G_x)-\dim \aut^0(\mE_{r+1})=\dim(G_x)-3$, $\dim\iso(\mE_r|_{P_{r+1}},\mE_{r+1}|_{P_{r+1}})=3$, the result follows.
\end{proof}
With the above two lemmas, we immediately obtain the following result:
\begin{cor}\label{mc1}
  Given the construction in \ref{l1}, we have the following upper bounds on the fiber dimensions $m_{r+1}$:
  \begin{enumerate}
    \item $m_{g+2}\le2$, $m_{g+1},m_{g+5}\le1$, $m_{g+3},m_{g+7},m_{g+9}\le0$;
    \item $m_{g+6}\le 2$ if $k_1>2$, and $m_{g+6}\le 1$ if $k_1=2$;
    \item $m_{g+11+11s}\le 2$, $m_{g+12+11s}\le 0$, $m_{g+14+11s}\le 1$, $m_{g+15+11s}\le 1$, $m_{g+16+11s}\le 2$, $m_{g+17+11s}\le 0$, $m_{g+18+11s}\le 0$, $m_{g+19+11s}\le 2$, $m_{g+20+11s}\le 0$, for $s=0,...,q-2$;
    \item $m_{h+4\ell+1},m_{h+4\ell+3}\le2$, $m_{h+4\ell+2},m_{h+4\ell+4}\le0$, for $\ell=0,...,k_1-3q$;
    \item $m_{g'-3}\le0$, $m_{g'-2}\le2$, $m_{g'-1}\le0$, $m_{g'}\le0$.
  \end{enumerate}
\end{cor}
\begin{proof}
By \ref{m1}, statements of the form $m_j\le 2$ are trivial.

For $j=g+3,g+12+11s,g+17+11s,g+20+11s,h+4\ell+2,h+4\ell+4$ and $g'-1$, the conclusion follows from the observation in \ref{often}.

For $j=g'-3,g'$, the results follow from \ref{m2} directly.

For $j=g+1,g+7,g+9,g+18+11s$, $g+14+11s,g+15+11s$, the results follow from recognizing the special indices and comparing the relevant vanishing orders (see the proof of \ref{standard}).

One can compute that $c_{5,2g+10}$ is a repeated vanishing order if $k_1>2$, but is non-repeated vanishing if $k_1=2$. This leads to the two different upper bounds for $m_{g+6}$ in the two cases.
\end{proof}
If $\mE_{r+1}$ is unstable, the situation is slightly more complicated. In general one needs to make a choice of $(q_t)$ and $(q'_t)$ to determine $V_{r+1}$, where $(q_t)$ and $(q'_t)$ are the associated sequences of points at $P_{r+1},P_{r+2}$ resp. \begin{lem}\label{m3}
Same setup as in Lemma \ref{m1}, except now assume $\mE_{r+1}$ is unstable. Then, we have $m_{r+1}\le1-\ep+|M|$, where
\begin{enumerate}
  \item $M=\{i|c_{i,2r+1} \text{ is non-repeated, and }c_{i,2r}+c_{i,2r+1}=\frac{1}{2}(\deg(\mE_{r+1})-1)-1\}$;
  \item $\ep=\begin{cases}1 \text{ if }\exists S_t\ni i \text{ such that }c_{i,2r}+c_{i,2r+1}\ge\frac{1}{2}(\deg(\mE_{r+1})-1)\\
0 \text{ otherwise}\end{cases}$.
\end{enumerate}
\end{lem}
\begin{proof}
In this case, $\dim \aut^0(\mE_{r+1})=2$. Again, following the notation in Corollary \ref{fdim}, we have $m_{r+1}=\dim(G_x)-\dim \aut^0(\mE_{r+1})=\dim(G_x)-2$.

Suppose $i\in M$. In particular, $c_{i,2r+1}\neq c_{i_1,2r+1},c_{i_1,2r+1}-1$. Then, $$\dim\Gamma(\mE_{r+1}(-c_{i,2r}P_{r+1}-c_{i,2r+1}P_{r+2}))=3$$ and $\Gamma(\mE_{r+1}(-c_{i,2r}P_{r+1}-c_{i,2r+1}P_{r+2}))=\s(s_1,s_2,s_3)$ such that
$$\begin{cases}
\ord_{P_{r+1}}(s_1)=c_{i,2r}+1,\ord_{P_{r+2}}(s_1)=c_{i,2r+1}\\
\ord_{P_{r+1}}(s_2)=c_{i,2r},\ord_{P_{r+2}}(s_2)=c_{i,2r+1}+1\\
\ord_{P_{r+1}}(s_3)=c_{i,2r},\ord_{P_{r+2}}(s_3)=c_{i,2r+1}
\end{cases}$$
Under our assumptions, $V_{r+1}$ always has a $(P_{r+1},P_{r+2})$-adapted basis, and such a basis is determined by the associated sequences of points $(q_t),(q'_t)$ at $P_{r+1},P_{r+2}$ resp. On one and, if $S_t\ni i'$ such that $c_{i',2r}+c_{i',2r+1}=\frac{1}{2}(\deg(\mE_{r+1})-1)$, then $q_t=q^*$\footnote{$q^*$ is the image of sections of the destabilizing summand of $\mE_{r+1}$.}. This poses one non-trivial closed condition on $\iso(\mE_r|_{P_{r+1}},\mE_{r+1}|_{P_{r+1}})$. Thus, image of the projection $f:G_x\to \iso(\mE_r|_{P_{r+1}},\mE_{r+1}|_{P_{r+1}})$ has dimension $\le 3-\ep$.

On the other hand, given $\phi_{r+1}$, $(q_t)$ is determined. The fiber dimension of $f$ is then bounded from above by the choice of $(q'_t)$, which is bounded from above by $|M|$.

In summary, $m_{r+1}=\dim(G_x)-2\le(3-\ep+|M|)-2=1-\ep+|M|$.
\end{proof}
\begin{cor}\label{mc3}
  Given the construction in \ref{l1}, we have the following upper bounds on the fiber dimensions $m_{r+1}$:
  \begin{enumerate}
    \item $m_{g+4}\le1$, $m_{g+8}\le2$;
    \item $m_{g+10+11s}\le 3$, $m_{g+13+11s}\le 1$, for $s=0,...,q-2$.
  \end{enumerate}
\end{cor}
\begin{proof}$\ep$ can be read off from the configuration data given in the construction; $M$ can be deduced from the given configuration $\tau_j$, together with the vanishing sequences provided in \ref{parta}, \ref{parte}:

  For $j=g+4$, $\ep=1$ and $M=\{k\}$. So, $m_{g+4}\le1$.

  For $j=g+8$, $\ep=1$ and $M=\{5,k+3\}$. So, $m_{g+8}\le2$.

  For $j=g+10+11s$, $\ep=1$ and $M=\{6s+5,k,k+4\}$. So, $m_{g+10+11s}\le3$.

  For $j=g+13+11s$, $\ep=1$ and $M=\{1\}$. So, $m_{g+13+11s}\le1$.
\end{proof}
\begin{rem}
In general, it suffices to find upper bounds of fiber dimensions, rather than calculating the fiber dimensions. The reason is two-fold. First of all, if one can show $\dim\mathcal{G}^{k',\text{EHT}}_{2,\omega_{g'},d^{g'}_{\bullet},c^{\Gamma'}_{g'}}(X_{g'})\le\rho_{g',k'}$ and it contains limit linear series whose underlying vector bundle is semi-stable, by the fact that $\rho_{g',k'}$ is also a lower bound for the dimension of the semi-stable locus of $\mathcal{G}^{k',\text{EHT}}_{2,\omega_{g'},d^{g'}_{\bullet},c^{\Gamma'}_{g'}}(X_{g'})$ (\cite{Ber}), we're done. Secondly, due to the vanishing sequences we prescribe at nodal points, it turns out that for some $g_2<g_1<g'$ the forgetful morphism $p_{g_1,g_2}$ maps onto a closed substack of $\GG_{g_2}$ determined by some particular configuration $\tau(P_{g_2+1})$. In those situations, we sometimes need to analyze fibers of forgetful morphisms $p_{j',j}$ where $j'-j>1$.
\end{rem}
\subsubsection{Conclusion from the First Construction}
\ \\
In this part, we shall summarize out conclusion from Induction Step 2, Part 1. In particular, we prove the following theorem:
\begin{thm}\label{main1}
Let $k=2k_1+1$ be odd, and $a^{\Gamma}$ be $(g,k)$-standard. Suppose $\mathcal{G}^{k,\text{EHT}}_{2,\omega_g,d_{\bullet},a^{\Gamma}}(X_g)$ contains an open substack of expected dimension $\rr$ of semi-stable EHT limit linear series, which contains sufficiently generic objects. Let $c^{\Gamma'}_{g'}$ be as defined in section \ref{l1}. Then, $\mathcal{G}^{k',\text{EHT}}_{2,\omega_{g'},d'_{\bullet},c^{\Gamma'}_{g'}}(X_{g'})$ also contains an open substack of dimension $\rho_{g',k'}$ of semi-stable limit linear series, which contains sufficiently generic objects.
\end{thm}
\begin{rem}
  Intuitively, this theorem give the induction from $(g,k)$ to $(g+2k+4-q,k+4)$. At every step of the induction, the parameter $|L|$ increase by $2q$.
\end{rem}
Hereafter, whenever $\mE_r$ is unstable, we shall denote $q^r_1,q^r_2$ to be the images of sections of the destabilizing summand of $\mE_r$ in $\PP\mE_r|_{P_r},\PP\mE_r|_{P_{r+1}}$ resp. Whenever $\mE_r$ is semi-stable, we denote $q^r_{1,1},q^r_{1,2}$ (resp. $q^r_{2,1},q^r_{2,2}$) to be the images of the canonical sections $\mE_r$ in $\PP\mE_r|_{P_r},\PP\mE_r|_{P_{r+1}}$ resp.

We carry out the dimension estimate in multiple steps:
\begin{lem}\label{g6}
$\dim\GG_{g+9}\le\dim\GG_g+6$.
\end{lem}
\begin{proof}
First of all, by Corollary \ref{mc1}, \ref{mc3}, we know the dimensions of fibers, $m_j$ for $j=g+1,...,g+9$, are bounded from above by $1,2,0,1,1$, 2 (1, if $k_1=2$), $0,2,0$ resp. Hence,
\begin{center}$\dim\GG_{g+9}\le\dim\GG_g+9$ when $k_1>2$, and $\dim\GG_{g+9}\le\dim\GG_g+8$ when $k_1=2$.\end{center}

To get the claimed dimension upper bound, one needs to analyze some of the configurations more closely.

First of all, the special indices of $(c_{i,g+16})$ and $(c_{i,2g+17})$ are $5$ and $k+4$ (see the proof of \ref{standard}); given $(c_{i,2g+15})$ (see \ref{parta}), $\tau(P_{g+9})$ must contain some set $S$, which contains $5,k+3$. In particular, let $s,s'$ be two section in $V_{g+8}$ such that $\ord_{P_{g+9}}(s)=c_{5,2g+15}$, and $\ord_{P_{g+9}}(s')=c_{k+3,2g+15}$; then, $s|_{P_{g+9}}=s'|_{P_{g+9}}$ must hold. This is a codimension one condition on $\GG_{g+8}$. Hence, $\dim\GG_{g+9}\le \dim\GG_{g+8}-1+m_{g+9}\le \dim\GG_{g+8}-1$.

Secondly, when $k_1>2$, $$c_{1,2g+14}+c_{1,2g+15},c_{k,2g+14}+c_{k,2g+15},c_{k+4,2g+14}+c_{k+4,2g+15}\ge\frac{1}{2}(\deg(\mE_{g+8})-1).$$ Hence, $\tau^{V_{g+8}}(P_{g+8})=\{\{1,k,k+4\},\{6,k+1\}\}$ in general. The point associated to $\{1,k,k+4\}$ is the image of sections of the destabilizing summand in $\PP\mE_{g+8}|_{P_{g+8}}$.

We claim that $\GG_{g+8}$ maps to a codimension one closed substack of $\GG_{g+7}$ and hence $\dim\GG_{g+8}\le(\dim\GG_{g+7}-1)+2=\dim\GG_{g+7}+1$. Given $(c_{i,2g+13})$ in \ref{parta}, one can see that $c_{1,2g+12},c_{k,2g+12},c_{k+4,2g+12}$ are all non-repeated vanishing orders. For a general object in $\GG_{g+7}$, $1,k,k+4$ do not belong to the same set in $\tau^{V_{g+7}}(P_{g+8})$. But for any object in $\GG_{g+8}$, $\tau(P_{g+8})$ contains the set $\{1,k,k+4\}$. This poses a codimension one condition on $\GG_{g+7}$. Thus, $\dim\GG_{g+9}\le \dim\GG_{g+7}$.

Similarly, given $(c_{i,2g+5})$, $(c_{i,2g+7})$ as in \ref{parta}, one sees that $$c_{3,2g+6}+c_{3,2g+7},c_{k+1,2g+6}+c_{k+1,2g+7},c_{k+4,2g+6}+c_{k+4,2g+7}\ge\frac{1}{2}(\deg(\mE_{g+4})-1).$$
Since $c_{3,2g+6},c_{k+1,2g+6},c_{k+4,2g+6}$ are the only non-repeated vanishing orders in $(c_{i,2g+6})$, $\tau(P_{g+4})$ contains a single subset $S=\{3,k+1,k+4\}$, whose corresponding point in $\PP\mE_{g+4}|_{P_{g+4}}$ is $q^{g+4}_1$. This clearly poses a non-trivial closed condition on $\GG_{g+3}$, and hence $\dim\GG_{g+4}\le(\dim\GG_{g+3}-1)+1=\dim\GG_{g+3}$.

Thus, when $k_1>2$,

$$\begin{aligned}
  &\dim\GG_{g+9}\le\dim\GG_g+m_{g+1}+m_{g+2}+m_{g+3}+m_{g+5}+m_{g+6}+m_{g+7}\\
  &=\dim\GG_g+1+2+0+1+2+0\\
  &=\dim\GG_g+6.
\end{aligned}$$

When $k_1=2$, one gets fewer non-repeated vanishing orders in $(c_{i,2g+14})$, and thus a different configuration at $P_{g+8}$. However, by a very similar argument, one gets the same dimension estimate.
\end{proof}
Similarly, we have:
\begin{lem}\label{g7}
For $s=0,...,q-2$, $\dim\GG_{g+11s+20}\le\dim\GG_{g+9+11s}+9$.
\end{lem}
\begin{proof}
  By Corollary \ref{mc1}, \ref{mc3},
  $$\dim\GG_{g+11s+20}\le\dim\GG_{g+9+11s}+\sum_{r=1}^{11}m_{g+9+11s+r}\le\dim\GG_{g+9+11s}+12.$$
  To get the claimed upper bound, we need to further analyze the configurations.

  Given $\tau_{g+13+11s}$ as in the construction, $\tau(P_{g+13+11s})$ contains a set $S$ which contains the indices $5+6s$ and $k+4$. Hence, for any two sections $t,t'$ in $V_{g+12+11s}$ such that
  $$\ord_{P_{g+13+11s}}(t)=c_{5+6s,2g+24+22s},\ord_{P_{g+13+11s}}(t')=c_{k+4,2g+24+22s},$$
  $t|_{P_{g+13+11s}}=t'|_{P_{g+13+11s}}$ must hold. Given $(c_{i,2g+25+22s})$ in \ref{parta}, the bundle $\mE_{g+13+11s}$, the data $\tau_{g+13+11s}$ and the relevant special indices (see \ref{standard}), one can compute:
  $$c_{5+6s,2g+22+22s}+c_{k+4,2g+22+22s}=(g-k_1+12+14s)+(g+k_1+9+8s)-1,$$ $$c_{5+6s,2g+20+22s}+c_{k+4,2g+20+22s}=(g-k_1+11+14s)+(g+k_1+8+8s)-1.$$
  By Corollary \ref{auxc}, let $s,s'$ be any sections in $V_{g+10+11s}$ such that
  $$\ord_{P_{g+11+11s}}(s)=c_{6s+5,2g+21+22s}\text{ and }\ord_{P_{g+11+11s}}(s')=c_{k+4,2g+21+22s},$$ $s|_{P_{g+11+11s}}=s'|_{P_{g+11+11s}}$ must hold. Denote $\mathcal{P}$ to be the codimension one closed locus in $\GG_{g+10+11s}$ such that $s|_{P_{g+11+11s}}=s'|_{P_{g+11+11s}}$; the forgetful morphism $p_{g+11+11s,g+10+11s}$ maps $\GG_{g+11+11s}$ to $\mathcal{P}$. Consequently,
  $$\dim \GG_{g+11+11s}\le(\dim\GG_{g+10+11s}-1)+m_{g+11+11s}\le\dim\GG_{g+10+11s}+1.$$
  Moreover, given $(c_{i,2g+17+22s}),(c_{i,2g+19+22s})$ (\ref{parte}), for a general object in $\GG_{g+12+11s}$,
  $$\tau(P_{g+11+11s})=\{\{5+6s,k+4\},\{k\},\{1,2,8+6s,k+1\}\}$$
  By \ref{aux}, \ref{auxc} and referring to relevant special indices, a general object in $\GG_{g+12+11s}$ must have the following configuration at $P_{g+13+11s}$: $$\{\{1\},\{2\},\{5+6s,k+4\},\{8+6s,k+1\},\{k\}\}.$$
  But for objects in $\GG_{g+13+11s}$, the configuration at $P_{g+13+11s}$ should be $$\{\{1\},\{2\},\{5+6s,8+6s,k,k+1,k+4\}\}.$$
  Consequently,
  $$\dim\GG_{g+13+11s}\le(\dim\GG_{g+12+11s}-2)+m_{g+13+11s}.$$
  Therefore,
$$\begin{aligned}
    &\dim \GG_{g+20+11s}\le(\dim\GG_{g+12+11s}-2)+m_{g+13+11s}+m_{g+14+11s}+...+m_{g+20+11s}\\
    &\le-1+\dim\GG_{g+9+11s}+m_{g+10+11s}+m_{g+12+11s}+...+m_{g+20+11s}\\
    &=\dim\GG_{g+9+11s}+9.
  \end{aligned}$$
\end{proof}
\begin{cor}
  $\dim\GG'\le\rho_{g',k'}$.
\end{cor}
\begin{proof}
  Clearly, $\GG_g$ is a substack of $\GG^0$ defined in Lemma \ref{suffover}. Denote $\GG$ to be the open substack of \ma\ of sufficiently generic objects.

  Combine Proposition \ref{1dim}, Lemma \ref{g6}, \ref{g7}, Corollary \ref{m1}, \ref{m2}, we get
  $$\begin{aligned}
  &\dim\GG'\le\dim\GG_{g+9}+\sum_{j>g+9}m_j\le\dim\GG+1+6+9(q-1)+4(k_1-3q+1)+2\\
  &=\rr+4k_1-3q+4=\rho_{g',k'}.
\end{aligned}$$
\end{proof}
Given the condition that $\GG$ contains sufficiently generic objects, one gets $\GG_g$ is non-empty. A careful analysis of the configurations $\tau(P_{g+j})$ then shows $\mathcal{G}^{k',\text{EHT}}_{2,\omega_{g'},d'_{\bullet},c^{\Gamma'}_{g'}}(X_{g'})$ is non-empty. (See Appendix \ref{nonempty}). One can also verify the canonical determinant condition easily. (See Appendix \ref{can1}.) Moreover, the existence of non-empty semi-stable locus follows from posing some open conditions on gluing data. (See Appendix \ref{sst}.) Similarly, the existence of sufficiently generic objects. (See Appendix \ref{suffred}) Therefore, Theorem \ref{main1} holds.
\begin{cor}\label{main11}
  Same setup as in \ref{main1}. For any $q':1\le q'\le q$, define $g'=g+4k_1+6-q'$, $N=g'-g$ and let $d'_{\bullet}=d_{\bullet}+(2N,...,2N)$. Then, $\mathcal{G}^{k',\text{EHT}}_{2,\omega_{g'},d'_{\bullet},c^{\Gamma'}_{g'}}(X_{g'})$ contains an open substack of dimension $\rho_{g',k'}$ of semistable limit linear series, which contains sufficiently. generic objects.
\end{cor}
\begin{proof}
  We essentially utilize the same construction as for Theorem \ref{main1}, replacing $q$ by $q'$ everywhere. The dimension estimate and all other arguments apply.
\end{proof}
\subsubsection{The Second Construction}\label{l2}
\indent In this part, we describe the construction for an induction argument, where $k$ is odd and $L(g,k)$ stays constant while $k,g$ increase.

 Suppose $a^{\Gamma}$ is $(g,k)$-standard and \ma\ is non-empty. Take $g'=g+2k_1+2$ and $k'=k+2$. Let $\GG^0$ be the stack defined in Lemma \ref{suffover}, with respect to $N=g'-g$ and $n=2$. we define $\mathcal{G}^{k',\text{EHT}}_{2,\omega_{g'},d'_{\bullet},c^{\Gamma'}_{g'}}(X_{g'})$ by determining a $(k+2)\times 2(g'-1)$-matrix $c^{\Gamma'}_{g'}$ using the following data (in this case, every $\mE_j$ is semi-stable):
\begin{enumerate}
  \item $\mE_j=\mO(g-k_1+2t,3k_1+1-2t)\op \mO(g+k_1,k_1+1)$, for $j=g+1+t$, $0\le t\le k_1-1$;
  \item $\mE_j=\mO(g-1+2t,2k_1+2-2t)\op\mO(g+2k_1+1,0)$, for $j=g+k_1+1+t$, $0\le t\le k_1$;
  \item $\mE_{g'}=\mO(g'-1,0)^{\op 2}$.
\end{enumerate}
Fix $\GG=\mathcal{G}^{k',\text{EHT}}_{2,\omega_{g'},d^{g'}_{\bullet},c^{\Gamma'}_{g'}}(X_{g'})$ and let $\GG_r=\GG^{k',\text{EHT}}_{2,\oo_r(A_rP_{r+1}),d^r_{\bullet},c^{\Gamma'}_r}(X_r)$ (\tb{Notation 17}).

Here is the dimension estimate:
\begin{lem}\label{twosp}
$\dim\mathcal{G}^{k',\text{EHT}}_{2,\omega_{g'},d'_{\bullet},c^{\Gamma'}_{g'}}(X_{g'})\le\rho_{g',k'}$.
\end{lem}
\begin{proof}
Still denote $m_j$ for the fiber dimension of the morphism $p_{j,j-1}:\GG_j\to\GG_{j-1}$.

By the calculation in the proof of \ref{suffover},
$$(c_{i,2g-1})^T=(3k_1+2,[3k_1+1]_2,...,[2k_1+2]_2,k_1+1,k_1)$$
For $j=g+1+t$, $0\le t\le k_1-1$, special indices are $i_1=2t+1,i_2=k+1$; for $j=g+k_1+1+t$, $0\le t\le k_1$, special indices $i_1=2t+1,i_2=k+2$.

For $g<j\le g'-2$, one can inductively check that exactly one of $c_{i_1,2j-2}$, $c_{i_2,2j-2}$ is non-repeated, where $i_1,i_2$ are the special indices. Hence, $m_j\le 1$.

For $j=g'-1$, $c_{k-1,2g'-4},...,c_{k+2,2g'-4}$ are all non-repeated and $m_{g'-1}\le0$.

$m_{g'}\le0$ follows directly from \ref{m2}.

Put together, we have the following computation:
$$\begin{aligned}
&\dim\GG''\le \dim\GG_g+\dis\sum_{j=g+1}^{g'}m_j\le\rr+1+1\cd(g'-g-2)+0+0\\
&=3g-3-\binom{k+1}{2}+2k_1+1=\rho_{g',k'}
\end{aligned}$$
\end{proof}
In this case, since all $\mE_j$ ($j>g$) are semi-stable, it is trivial to verify the semi-stability condition on a non-empty open locus of $\GG''$.

Meanwhile, one can verify the canonical determinant condition straightforwardly (see Appendix Lemma \ref{can2}). Therefore, based on the non-emptiness result in Appendix Proposition \ref{nonempty2}, we arrive at the following:
\begin{thm}\label{main2}
Let $k=2k_1+1$ be odd and $a^{\Gamma}$ be $(g,k)$-standard. Suppose $\GG=\mathcal{G}^{k,\text{EHT}}_{2,\omega_g,d_{\bullet},a^{\Gamma}}(X_g)$ contains an open substack of expected dimension $\rr$ of semi-stable EHT limit linear series, which contains sufficiently generic objects. Let $\mathcal{G}^{k',\text{EHT}}_{2,\omega_{g'},d'_{\bullet},c^{\Gamma'}_{g'}}(X_{g'})$ be as defined in \ref{l2}. Then, $\mathcal{G}^{k',\text{EHT}}_{2,\omega_{g'},d'_{\bullet},c^{\Gamma'}_{g'}}(X_{g'})$ contains an open substack of dimension $\rho_{g',k'}$ of semi-stable limit linear series.
\end{thm}
\subsubsection{The Third Construction}\label{l3}
We now describe a construction for cases where $k$ is even. This is based on the constructions for odd $k$'s.
\begin{lem}\label{C3}
Given $k=2k_1+1\ge 5$, $a^{\Gamma}$ and $d_{\bullet}=(d_1,...,d_g)$ as in Lemma \ref{C1}. Take $m=2$, $N=3k_1+3$. Define $\hat{d}_{\bullet}:=(d'_1,...,d'_g)=d_{\bullet}+(2N,...,2N)$. Let $c^{\Gamma}$ be the $(k+4)\times(2g-2)$ matrix obtained from Lemma \ref{C1} for such data. Denote $c^{\Gamma}_{k+3}=(c_{ij})$ to be the $(k+3)\times(2g-2)$ matrix consisting of the first $k+3$ rows of $c^{\Gamma}$. Then, $c^{\Gamma}_{k+3}$ satisfies the following conditions:
\begin{enumerate}
  \item For $j=1,...,g-1$, $c_{i,2j-1}+c_{i,2j}=g+3k_1+2$.
  \item For $j=1,...,g$, if $(a_{i,2j-2})$ and $(a_{i,2j-1})$ satisfy Lemma \ref{one} (resp. Lemma \ref{one2}) with $d=\frac{1}{2}d_j$ (resp. $d=\frac{1}{2}(d_j-1)$), so do $(c_{i,2j-2})$ and $(c_{i,2j-1})$, with $d=\frac{1}{2}d'_j$ (resp. $d=\frac{1}{2}(d'_j-1)$).
  \item For $i=k+1,...,k+3$, if $c_{i,2j_t-2}+c_{i,2j_t-1}=\frac{1}{2}(d'_{j_t}-1)$, then $c_{i,2j_{t+1}-2}+c_{i,2j_{t+1}-1}=\frac{1}{2}(d'_{j_{t+1}}-1)-1$; if $c_{i,2j_t-2}+c_{i,2j_t-1}=\frac{1}{2}(d'_{j_t}-1)-1$, then $c_{i,2j_{t+1}-2}+c_{i,2j_{t+1}-1}=\frac{1}{2}(d'_{j_{t+1}}-1)$.
\end{enumerate}
\end{lem}
\begin{proof}
Since $c^{\Gamma}_{k+3}$ consists of the first $k+3$ rows of $c^{\Gamma}$, conditions 1,3 are direct consequences of Lemma \ref{C1}. For condition 2, there are two cases. The assumptions in \ref{one}, \ref{one2} are either considering the existence of special indices, or patterns of repetition among $(c_{i,2j-2},c_{2j-1})$, or the sum $c_{i,2j-2}+c_{i,2j-1}$. If $(c_{i,2j-2})$ and $(c_{i,2j-1})$ satisfy Lemma \ref{one} (or \ref{one2}), so do $(c_{i,2j-2})_{i\le k+3}$ and $(c_{i,2j-1})_{i\le k+3}$.
\end{proof}
\begin{rem}
  For the third construction, space of sections with desired vanishing sequences still admit adapted basis. This provides us with convenience when proving non-emptiness of the stack.
\end{rem}
\begin{prop}\label{even1}
Denote $\ti{\GG}_1$ to be the open locus of sufficiently generic objects in \ma\ satisfying the following strengthened genericity property:\\
\noindent\tb{(A)} For $t=1,...,2T$, let $q^t_1$ (resp. $q^t_2$) be the image in $\PP\mE_{j_t}|_{P_{j_t}}$ (resp. $\PP \mE_{j_t}|_{P_{j_t+1}}$) of sections of the destabilizing summand of $\mE_{j_t}$; $q^t_2$ is not glued to $q^{t+1}_1$ in $(c_{k+4,2j_t-1}-N)$-th order.

Suppose $\ti{\GG}_1$ is non-empty. Then, the locally closed substack of\\ $\mathcal{G}^{k+3,\text{EHT}}_{2,\omega_g((3k_1+3)P_{g+1}),\hat{d}_{\bullet},c^{\Gamma}_{k+3}}(X_g)$ consisting of limit linear series with vanishing sequence
$$(4k_1+3,[4k_1+2]_2,...,[3k_1+3]_2,2k_1+2,[2k_1+1]_2)$$
at $P_{g+1}$ is also non-empty.
\end{prop}
\begin{proof}
Under the current assumption, \ma\ contains sufficiently generic objects. By Proposition \ref{Induction}, the locally closed substack $\GG^0$ of $\mathcal{G}^{k+4,\text{EHT}}_{2,\omega((3k_1+3)P_{g+1}),\hat{d}_{\bullet},c^{\Gamma}}(X_g)$ of limit linear series with vanishing sequence
$$(4k_1+3,[4k_1+2]_2,...,[3k_1+3]_2,2k_1+2,[2k_1+1]_2,2k_1)$$
at $P_{g+1}$ is non-empty.

Recall the stack morphism $\psi:\GG^0\to\GG_1$ in Theorem \ref{Induction}. Let $((\mE_j,V_j),(\phi_j))$ be a $K$-valued object in $\GG^0$ over an object satisfying Property \tb{(A)} under $\psi$. We shall determine a $(k+3)$-dimensional subspace $\ti{V}_j\subset V_j$ for every $j$, with the desired vanishing sequences at each node.\\
\tb{Case 1:} $c_{k+3,2j-1}>c_{k+4,2j-1}$. In this case, $\ti{V}_j:=V_j(-c_{k+3,2j-1}P_{j+1}).$\\
\tb{Case 2:} $c_{k+3,2j-1}=c_{k+4,2j-1}$. In this case, one defines:
$$\ti{V}_j:=\ker (V_j\to\mE_j(-c_{k+4,2j-1}P_{j+1})|_{P_{j+1}}/\mL^{j}),$$
where $\mL^j$ is a line in the fiber whose image $q_j$ in $\PP\mE_j|_{P_j}$ is glued to $q_{j_t}$ in $c_{k+3,2j-1}$-th order via $\phi_{j+1}$,...,$\phi_{j_t}$, where $q_{j_t}$ is the image of sections of the destabilizing summand of $\mE_{j_t}$ in $\PP\mE_{j_t}|_{P_{j_t}}$ and $j_t>j$ is the smallest such index that $\mE_{j_t}$ is unstable. Notice that $((\mE_j,V_j),(\phi_j))$ is over an object satisfying Property \tb{(A)}\rm, so such $\mL^j$ always exists. \\
\indent It is easy to see the $\ti{V}_j$ thus chosen are compatible with the gluing data and have the desired vanishing sequences at $P_j,P_{j+1}$ resp. Hence, $((\mE_j,\ti{V}_j),(\phi_j))$ is an object of $\mathcal{G}^{k+3,\text{EHT}}_{2,\omega((3k_1+3)P_{g+1}),\hat{d}_{\bullet},c^{\Gamma}_{k+3}}(X_g)$, which is non-empty.
\end{proof}
\begin{rem}
Similar to being sufficiently generic, satisfying Property \tb{(A)} is an open condition: it is a complement of finitely many closed conditions.
\end{rem}
Moreover, we have the following:
\begin{lem}\label{zero}
  $((\mE_j,\ti{V}_j),(\phi_j))$ obtained in \ref{even1} is uniquely determined by $((\mE'_j,V'_j),(\phi_j))=\psi((\mE_j,V_j),(\phi_j))$, where $\psi$ is as defined in Theorem \ref{Induction}.
\end{lem}
\begin{proof}
When $j=1$, $\ti{V}_1=\Gamma(\mE_1(-c_{k+3,1}P_2))$ is obviously uniquely determined.

Given $\ti{V}_{j-1}$ and $\phi_j$, if $\mE_j$ is semi-stable, under our assumptions on vanishing sequences, $\ti{V}_j$ is uniquely determined by its associated sequence of points in $\PP\mE_j|_{P_j}$, which is determined by the associated sequence of points of $\ti{V}_{j-1}$ together with $\phi_j$. If $\mE_j$ is unstable, $\ti{V}_j$ is determined by its associated sequence of points in $\PP\mE_j|_{P_j}$, $\PP\mE_j|_{P_{j+1}}$. We have three cases to analyze:
\begin{enumerate}
  \item $c_{k,2j-1}=c_{k+1,2j-1}$;
  \item $c_{k,2j-1}>c_{k+1,2j-1}>c_{k+2,2j-1}$;
  \item $c_{k,2j-1}>c_{k+1,2j-1}=c_{k+2,2j-1}$.
\end{enumerate}
In the first two cases, one can check that the two associated sequences of $\ti{V}_j$ are solely determined by (the associated sequences of) $V'_{j-1}$; in the third case, $j=j_{2t-1}$ and the associated sequences are determined by $\ti{V}_{j-1}$, $\phi_j$,..,$\phi_{j_{2t}}$ and $\ti{V}_{j_{2t}}$, where $c_{k,2j_{2t}-1},c_{k+1,2j_{2t}-1}$ fall into one of the first two cases and hence $\ti{V}_{j_{2t}}$ is determined solely by $V'_{j_{2t}}$.
\end{proof}
We now work out the relevant fiber dimension.
\begin{thm}\label{even3}
Let $\ti{\GG}_1$ be the substack of \ma\ of sufficiently generic objects satisfying Property \tb{(A)} (\ref{even1}) and $\ti{\GG}^0$ be the substack of $\mathcal{G}^{k+3,\text{EHT}}_{2,\omega((3k_1+3)P_{g+1}),\hat{d}_{\bullet},c^{\Gamma}_{k+3}}(X_g)$ of objects satisfying the two properties in Lemma \ref{suffover} and the following one:\\
\noindent\tb{(A')} For $t=1,...,2T$, let $q^t_1$ (resp. $q^t_2$) be the image in $\PP\mE_{j_t}|_{P_{j_t}}$ (resp. $\PP \mE_{j_t}|_{P_{j_t+1}}$) of sections of the destabilizing summand of $\mE_{j_t}$; $q^t_2$ is not glued to $q^{t+1}_1$ in $c_{k+4,2j_t-1}$-th order.\\
Then, there exists a stack morphism $\psi_2:\ti{\GG}^0\to\ti{\GG}_1$ with fiber dimension zero.
\end{thm}
\begin{proof}
The construction of $\psi_2$ is similar to the construction of $\psi$ in Theorem \ref{Induction}.

\indent Let $((\mE_j,V_j),(\phi_j))$ be an object over some $K$-scheme $T$. $\psi_2$ sends $((\mE_j,V_j),(\phi_j))$ to $((\mE'_j,V'_j),(\phi_j))$, where $\mE'_j:=\mE_j(-(3k_1+3)P_{j+1})$ and $V'_j$ are defined as follows:\\
\tb{Case 1}: $c_{k,2j-1}>c_{k+1,2j-1}$. In this case, $V'_j:=V_j(-c_{k,2j-1}(T\times P_{j+1}))$, which is a sub-bundle of $\pi^j_*\mE_j(-(3k_1+3)P_{j+1})$ in the sense of \ref{subb}.\\
\tb{Case 2}: $c_{k,2j-1}=c_{k+1,2j-1}$. In this case, there exists a smallest $j'>j$ such that $c_{k,2j'-1}>c_{k+1,2j'-1}$ and $V'_{j'}$ is defined in \tb{Case 1}. One can then inductively define $V'_{j'-1},...,V'_j$ in exactly the same way as in Theorem \ref{Induction} and we omit the details.\\
\indent The operation of $\psi_2$ on morphisms is defined in the same way as for $\psi$ in \ref{Induction}.

To conclude the fiber dimension over a geometric point $x:\spec(F)\to\ti{\GG}_1$, we claim that in this case the fiber $\ti{\GG}^0_x$ is just a point. Following Lemma \ref{zero} and by exactly the same kind of argument as in Proposition \ref{1dim}, one can establish an equivalence of categories between $\spec(F)$ and $\ti{\GG}^0_x$.

This concludes the proof.
\end{proof}
Given $\mathcal{G}^{k+3,\text{EHT}}_{2,\omega((3k_1+3)P_{g+1}),\hat{d}_{\bullet},c^{\Gamma}_{k+3}}(X_g)$ as in Theorem \ref{even3}, denote $g'=g+3k_1+3$. We now define $\GG^{k+3,\text{EHT}}_{2,\oo_{g'},d'_{\bullet},c^{\Gamma'}_{g'}}(X_{g'})$ by determining some $(k+3)\times (2g'-2)$ matrix $c^{\Gamma'}_{g'}$ with $c^{\Gamma}_{k+3}$ being a sub-matrix consisting of its first $2g-2$ columns. We shall present the data in the same way as before:\\
\ \\
\noindent\framebox[1\width]{\ $\mE_{g+1}$ is unstable:\ }\ $\mE_{g+1}=\mO(g-k_1,b'-(g-k_1))\op\mO(g+k_1,b'-(g+k_1-1))$, $\tau_{g+1}=\{1,k\}$. \\
\ \\
\noindent\framebox[1\width]{\ $\mE_{g+2},\mE_{g+3},\mE_{g+4}$ are semi-stable:\ }
\begin{enumerate}
  \item $\mE_{g+2}=\mO(g-k_1-1,b'-(g-k_1-1))\op \mO(g+k_1+2,b'-(g+k_1+2))$;
  \item $\mE_{g+3}=\mO(g-k_1+1,b'-(g-k_1+1))\op \mO(g+k_1+2,b'-(g+k_1+2))$;
  \item $\mE_{g+4}=\mO(g-k_1+2,b'-(g-k_1+2))\op \mO(g+k_1+3,b'-(g+k_1+3))$.
\end{enumerate}

\noindent\framebox[1\width]{\ $\mE_{g+5}$ is unstable:\ }\\
\ \\
$\mE_{g+5}=\mO(g-k_1+4,b'-(g-k_1+4)])\op\mO(g+k_1+3,b'-(g+k_1+4))$, $\tau_{g+5}=\{5,k+1\}$.\\
\ \\
\noindent\framebox[1\width]{\ $\mE_{g+6}$,...,$\mE_{g'}$ are semi-stable:\ }\\
\ \\
(1) $\mE_{g+5+j}=\mO(g-k_1+5+2j,b'-(g-k_1+5+2j))\op\mO(g+k_1+3,b'-(g+k_1+3))$,

    for $j=1,...,k_1-1$;\\
(2) $\mE_{g+k_1+5+2s}=\mO(g+5+3s,b'-(g+5+3s))\op\mO(g+2k_1+3+s,b'-(g+2k_1+3+s))$,

for $s=0,...,k_1-2$;\\
(3) $\mE_{g+k_1+6+2s}=\mO(g+6+3s,b'-(g+6+3s))\op\mO(g+2k_1+4+s,b'-(g+2k_1+4+s))$,

for $s=0,...,k_1-2$;\\
(4) $\mE_{g'}=\mO(g'-1,0)^{\op 2}$.\\
\indent Similar to the first construction, we list certain vanishing sequences:
\begin{lem}\label{evenonly}We have\\
$(c_{i,2g+1})^T=(4k_1+3,[4k_1+2]_2,4k_1+1,[4k_1-3]_2,...,[3k_1]_2,3k_1-1,2k_1,[2k_1-1]_2)$;\\
$(c_{i,2g+7})^T=(4k_1+1,[4k_1]_2,4k_1-2,[4k_1-3]_2,...,[3k_1]_2,3k_1-1,2k_1,[2k_1-1]_2)$;\\
$(c_{i,2g+9})^T=([4k_1-1]_2,[4k_1-2]_2,4k_1-4,[4k_1-5]_2,...,[3k_1-1]_2,3k_1-2,3k_1-3,$\\
$2k_1-1,[2k_1-2]_2)$;\\
$(c_{i,2g+2k_1+3})^T=([3k_1+2]_2,[3k_1+1]_2,[3k_1-1]_2,...,[2k_1+3]_2,2k_1+2,2k_1+1,2k_1,$\\
$2k_1-1,[k_1+1]_2)$.
\end{lem}
Eventually, we carry out the dimension estimate:
\begin{lem}
If $\ti{\GG}_1$ (\ref{even3}) is non-empty, $\dim\GG^{k+3,\text{EHT}}_{2,\oo_{g'},d'_{\bullet},c^{\Gamma'}_{g'}}(X_{g'})\le\rho_{g',k'}$.
\end{lem}
\begin{proof}
We still denote $\GG_r=\GG^{k+3,\text{EHT}}_{2,\oo_r(A_rP_{r+1}),d'^r_{\bullet},c^{\Gamma'}_r}(X_r)$ (\tb{Notation 17}) and $m_{r+1}$ for the fiber dimension of $p_{r+1,r}:\GG_{r+1}\to\GG_r$.

Note that $\ti{\GG}^0$ defined in Theorem \ref{even3} is a non-empty open substack of $\GG_g$.

Given the vanishing sequences in \ref{even1}, \ref{evenonly}, one can read off relevant information and apply \ref{m3} to conclude: $$m_{g+1}\le1-1+2=2, m_{g+5}\le1-1+2=2.$$
By referring to the bundles, one sees that $\tau(P_{g+4}),\tau(g+k_1+6+2s)$ fall into the situation in \ref{often} and $m_{g+4}\le 0$, $m_{g+k_1+6+2s}\le0$ for $s=0,...,k_1-2$.

Given $(c_{i,2g+9})$ in \ref{evenonly} and $\mE_{g+5+j}$ ($j=1,...,k_1-3$), it is easy to see that $k+1$ is always a special index and $c_{k+1,2g+2j+9}=g+k_1+3$ is always a non-repeated vanishing order. Hence, $m_{g+5+j}\le1$ for $j=1,...,k_1-3$.

Given $(c_{i,2g+1}),(c_{i,2g+2k_1+3})$ (\ref{evenonly}) and $\mE_{g+2},\mE_{g+k_1+3}$, conclude $m_{g+2}\le0,m_{g+k_1+3}\le0$.

By \ref{m2} and referring to the bundles, get $m_{g+k_1+4}\le0, m_{g'}\le0$.

Given $(c_{i,2g+2k_1+3})$ (\ref{evenonly}) and $\mE_{g+k_1+2}$, it also follows that a general object in $\GG_{g+k_1+2}$ has configuration $\{\{k-2,k+1\},\{k-1\},\{k\}\}$
at $P_{g+k_1+3}$. Given $\mE_{g+k_1+3}$, a general object in $\GG_{g+k_1+3}$ has configuration $\{\{k-2,k+1\},\{k-1,k\}\}$ at $P_{g+k_1+3}$. Consequently,
$$\dim\GG_{g+k_1+3}\le(\dim\GG_{g+k_1+2}-1)+m_{g+k_1+3}.$$

Combine all the results, we have
$$\begin{aligned}
 &\dim\GG^{k+3,\text{EHT}}_{2,\oo_{g'},d'_{\bullet},c^{\Gamma'}_{g'}}(X_{g'})\le \rr+\sum_{j=g+1}^{g'}m_j\\
 &=[\rho_{g,k}+2+0+2+0+2+(k_1-3)\cdot 1-1]+0+0+(k_1-1)\cdot(2+0)+0\\
 &=\rr+3k_1=\rho_{g',k'}
\end{aligned}$$

\end{proof}
 Again, one can verify the canonical determinant condition straightforwardly (see Appendix \ref{can2}). Based on the non-emptiness result in Appendix \ref{nonempty4} and the verification of semi-stability condition in \ref{sst2}, we draw the following conclusion:
\begin{thm}\label{main3}
Let $k=2k_1+1$ be odd and $a^{\Gamma}$ be a $(g,k)$-standard vanishing condition. Suppose $\mathcal{G}^{k,\text{EHT}}_{2,\omega_g,d_{\bullet},a^{\Gamma}}(X_g)$ contains an open substack of dimension $\rr$ of semi-stable EHT limit linear series with sufficiently generic objects satisfying Property \tb{(A)} (\ref{even1}). Then $\GG^{k+3,\text{EHT}}_{2,\oo_{g'},d'_{\bullet},c^{\Gamma'}_{g'}}(X_{g'})$ also contains an open substack of semi-stable limit linear series of expected dimension $\rho_{g',k'}$.
\end{thm}
\section{Final Conclusion}
In this section, we shall summarize our main results. We start with the computation for the base case of the induction.
\subsection{The Base Case}
The base case for our induction is where $(g,k)=(6,5)$. In this case we define
$$a^{\Gamma}=\begin{bmatrix}
  5&0&5&0&5&0&4&1&3&2\\
  5&0&5&0&4&1&4&1&2&3\\
  3&2&4&1&3&2&2&3&2&3\\
  3&2&2&3&2&3&1&4&0&5\\
  2&3&2&3&1&4&1&4&0&5
\end{bmatrix}_{5\times 10}.$$
\begin{prop}\label{base}
Denote $d_{\bullet}=(10,11,10,10,9,10)$. Then, $\dim\mathcal{G}^{5,\text{EHT}}_{2,d_{\bullet},a^{\Gamma}}(X_6)\le 0$ and contains a non-empty, semi-stable locus.
\end{prop}
\begin{proof}
 One can easily check $a^{\Gamma}$ is $(6,5)$-standard. In particular, the underlying bundles are
 $$\begin{aligned}
 &\mE_1=\mO(5P_2)^{\op2}\\
 &\mE_3=\mO(5P_4)\op\mO(3P_3+2P_4)\\
 &\mE_5=\mO(3P_5+2P_6)\op\mO(4P_5)
 \end{aligned}\hspace{0.6in}\begin{aligned}
   &\mE_2=\mO(5P_3)\op\mO(2P_2+4P_3)\\
   &\mE_4=\mO(P_4+4P_5)\op\mO(4P_4+P_5)\\
   &\mE_6=\mO(5P_6)^2
 \end{aligned}$$
 and the canonical determinant condition is satisfied.

To verify the moduli count, we utilize our results on fiber dimensions in Lemma \ref{m1}, \ref{m2}, \ref{m3}. As before, denote $\GG_j$ to be the stack of limit linear series on the first $j$ component of $X_6$, to which $\mathcal{G}^{5,\text{EHT}}_{2,d_{\bullet},a^{\Gamma}}(X_6)$ admits a natural forgetful morphism.

We need to compute the dimension of $\GG_1$ separately. Given the vanishing sequences $a_0:=(0,0,1,1,2),a_1:=(5,5,3,3,2)$, there exists a smooth cover $u:\Gr(5,\Gamma(\mE_1);a_0,a_1)\to\GG_1$, which is an $\aut^0(\mE_1)$-torsor. Since $\Gr(5,\Gamma(\mE_1);a_0,a_1)$ is 1-dimensional (a 5-dimensional space of sections is uniquely determined by the choice of a section $s$ such that $\ord_{P_1}(s)=\ord_{P_2}(s)=2$), $\dim \aut^0(\mE_1)=3$, we conclude that $\GG_1$ has stack dimension $-2$.

By Lemma \ref{m1}, \ref{m2}, \ref{m3}, we get the fiber dimension $m_j$ ( for$j=1,...,5$) is bounded from above by $0,2,0,0,0$ resp. Hence, $\dim\GG_6\le 0$.

We now show $\GG$ is non-empty. To do so, one can first take any $V_1,...,V_6$ with the desired vanishing sequences at the nodes (the existence of such $V_j$ is a consequence of the fact that $a^{\Gamma}$ is $(6,5)$-standard) and then find feasible $\phi_2,...,\phi_6$ correspondingly. Among the configurations, only $\tau(P_4)$ is non-trivial. Hence, one only needs to show the existence of a feasible $\phi_4$. Given any choice of $V_3$ and $V_4$, let $s,s'$ be any sections in $V_3,V_4$ resp. such that $\ord_{P_3}(s)=2,\ord_{P_3}(s')=3$ and denote $q_{31},q_{32}$ (resp, $q_{41},q_{42}$) to be images of the canonical sections of $\mE_3$ (resp. $\mE_4$) in $\PP\mE_3|_{P_4}$ (resp. in $\PP \mE_4|_{P_4}$). If $s|_{P_4}\neq q_{31},q_{32}$, $s'|_{P_4}\neq q_{41},q_{42}$, one can always find some $\phi_4$ such that $\phi_4(q_{31})=q_{41},\phi_4(q_{32})=q_{42}$ and $\phi_4(s|_{P_4})=s'|_{P_4}$ and hence we are done.

To see that it contains a non-empty semi-stable locus, notice that the vector bundle $\mE'$ on $X_6$ given by $((\mE_j)_{j=1}^6,(\phi_j)_{j=2}^5)$ satisfies the assumption in Proposition \ref{ssimple}. Hence, for the resulting bundle to be semi-stable it suffices to show there does not exist some invertible sub-sheaf $\mL$ of $\mE'$ such that $\mL|_{C_2},\mL|_{C_5}$ are the destabilizing summands of $\mE_2,\mE_5$ resp.

For this to hold, denote $q_{21}$ to be the point in $\PP\mE_2|_{P_3}$ which is the image of sections of the destabilizing summand and $q_{31},q_{32}$ to be the points in $\PP\mE_3|_{P_3}$ which are images of the canonical sections of $\mE_3$. When choosing $\phi_3$, we require that $\phi_3(q_{21})\neq q_{31},q_{32}$. It is a non-empty condition, because the configuration $\tau(P_3)$ is trivial, $a_{3,4}$ is the only non-repeated vanishing order and $3$ is not a special index of $(a_{i,4})$ and $(a_{i,5})$.
\end{proof}
\subsection{Summary of the Main Results}
Recall that we give three inductive constructions in chapter 5. In all cases we consider, $L(g,k)\le -1$. Hence, $g\le k_1^2+k_1$ when $k=2k_1+1$, and $g\le k_1^2$ when $k=2k_1$. Let $(g(n),k(n))$ be the parameters in the $n$-th step of each induction. We have:
\begin{enumerate}
  \item For the first construction, $(g(n+1),k(n+1))=(g(n)+2k(n)+4-q,k(n)+4)$, where $(g(1),k(1))$ can be any pair for which one has a construction and $q:1\le q\le\max\{1,\left\lfloor{k(n)-1\over 6}\right\rfloor\}$.
  \item For the second construction, $(g(n+1),k(n+1))=(g(n)+k(n)+1,k(n)+2)$ with $(g(1),k(1))$ being any pair in (1).
  \item In the third construction, we give a construction for parameters $$(g',k')=(g+{3\over2}(k+1),k+3),$$ for every pair of $(g,k)$ in (1) and (2).
\end{enumerate}
\begin{lem}\label{final}
Given the base case $(g,k)=(6,5)$, the constructions in (1) and (2) run through all pairs in $\mathcal{S}_1\bigcup\mathcal{S}_2\bigcup\mathcal{S}_3$, where
$$\begin{aligned}
&\mathcal{S}_1:=\{(g,k)|L(g,k)=-1,\rr\ge0\},\\
&\mathcal{S}_2:=\{(g,k)|k\ge 9,k\equiv1(\text{mod }4),k_1^2+k_1\ge g\ge k_1^2+k_1-1-\left\lfloor{(k_1-2)^2\over12}\right\rfloor\},\\
&\mathcal{S}_3:=\{(g,k)|k\ge7,k\equiv3(\text{mod }4),k_1^2+k_1\ge g\ge k_1^2+k_1-\left\lfloor{(k_1-2)^2+3\over12}\right\rfloor\}.\end{aligned}$$
\end{lem}
\begin{proof}
To get the pairs in $\mathcal{S}_1$, apply the 2nd construction with $(g(1),k(1))=(6,5)$.

To get the pairs in $\mathcal{S}_2$, set $(g(1),k(1))=(6,5)$ in the first construction. Then, $k(n)\equiv1$(mod 4). Denote $k_1(n):={k(n)-1\over2}$, $q(n)=\max\{1,\left\lfloor{k_1(n)\over3}\right\rfloor\}$.

Elementary computation shows that $L(g(n+1),k(n+1))=L(g(n),k(n))-2q(n).$

Also,
  $$\begin{aligned}
   g(n+1)&={1\over2}L(g(n+1),k(n+1))+(k^2_1(n+1)+k_1(n+1)+{1\over2})\\
   &=-1-\left\lfloor{(k_1(n+1)-2)^2\over12}\right\rfloor+k^2_1(n+1)+k_1(n+1)
  \end{aligned}$$
Notice that for any $n$, $g(n)$ is the smallest $g$ such that $L(g,k)=L(g(n),k(n))$ for some $k\ge9$ and $g\ge k_1^2+k_1-1-\left\lfloor{(k_1-2)^2\over12}\right\rfloor$. Without loss of generality, one can check that for $(g,k)=(g(n)-k(n)-1,k(n)-2)$, the desired inequality does not hold.

If $k\equiv1$(mod 4), then $k=k(m)$ for some $(g(m),k(m))$ in situation (1). Suppose $g=k_1^2+k_1-1-\left\lfloor{(k_1-2)^2\over12}\right\rfloor$, then $(g,k)=(g(n),k(n))$ for some $n$ in situation (1); suppose $g>k_1^2+k_1-1-\left\lfloor{(k_1-2)^2\over12}\right\rfloor$, we have two possibilities:
\begin{enumerate}
  \item For some $n$, $k=k(n)$ and $q(n)-1\ge g-g(n)\ge1$; in this case, the construction for the pair $(g,k)$ is directly obtained from the first construction (see Corollary \ref{main11});
  \item $k=k(n)$ and $g-g(n)>q(n)-1$; the construction for such a pair is obtained using the second construction via some induction starting from a pair $(g^*,k^*)$, where $g^*=g(m)+\ep$, $k^*=k(m)$ for some $m<n$, $q(m)-1\ge\ep\ge1$.
\end{enumerate}
Hence, the constructions run through all pairs in $\mathcal{S}_2$.

It remains to consider pairs in $\mathcal{S}_3$. Reset $(g(1),k(1))=(12,7)$ and use the first construction. Again one can compute
$$\begin{aligned}
   g(n+1)&={1\over2}L(g(n+1),k(n+1))+(k^2_1(n+1)+k_1(n+1)+{1\over2})\\
   &=-\left\lfloor{(k_1(n+1)-2)^2+3\over12}\right\rfloor+k^2_1(n+1)+k_1(n+1)
  \end{aligned}$$
Then, by the same argument as for the previous case, we conclude the constructions run through all pairs in $\mathcal{S}_3$ as well.
\end{proof}
\begin{cor}\label{final2}
  The construction in (3) runs through all pairs in $\mathcal{T}_1\bigcup\mathcal{T}_2$, where
  $$\begin{aligned}
   &\mathcal{T}_1=\{(g,k)|k=2k_1\ge8,k\equiv0(\text{mod }4),k_1^2+k_1\ge g\ge-2-\left\lfloor{(k_1-4)^2\over12}\right\rfloor+k_1^2\},\\
   &\mathcal{T}_2=\{(g,k)|k=2k_1\ge10,k\equiv2(\text{mod }4),k_1^2+k_1\ge g\ge-1-\left\lfloor{(k_1-4)^2\over12}\right\rfloor+k_1^2\}.
  \end{aligned}$$
\end{cor}
\begin{proof}
 It suffices to see that $(g,k)\in\mathcal{S}_2$ is equivalent to the condition $(g+{3\over2}k,k+3)\in\mathcal{T}_1$; the condition $(g,k)\in\mathcal{S}_3$ is equivalent to the condition $(g+{3\over2}k,k+3)\in\mathcal{T}_2$.
\end{proof}
\begin{cor}\label{final3}
For the following pairs of $(g,k)$:
$$\begin{cases}
k_1^2+k_1\ge g\ge k_1^2+k_1-\left\lfloor{(k_1-2)^2+3\over12}\right\rfloor, k=2k_1+1\ge5;\\
k_1^2+k_1\ge g\ge k_1^2-\left\lfloor{(k_1-4)^2\over12}\right\rfloor-1, k=2k_1\ge8.
\end{cases}$$
the moduli stack $\mathcal{G}^{k,\text{EHT}}_{2,\omega_g,d_{\bullet},a^{\Gamma}}(X_g)$ contains a non-empty open substack of semi-stable EHT limit linear series, of dimension $\rr$.
\end{cor}
\begin{proof}
For all $(g,k)\in\mathcal{S}$ in Lemma \ref{final}, $k_1^2+k_1\ge g\ge k_1^2+k_1-\left\lfloor{(k_1-2)^2+3\over12}\right\rfloor, k=2k_1+1\ge5$.

For all $(g,k)\in\mathcal{T}$ in Corollary \ref{final2}, $k_1^2+k_1\ge g\ge k_1^2-\left\lfloor{(k_1-4)^2\over12}\right\rfloor-1, k=2k_1\ge8$.

The corollary follows from Proposition \ref{base} and Theorem \ref{main1}, \ref{main2} \ref{main3}.
\end{proof}
Eventually, we paraphrase the smoothing theorem from \cite{Osp} for the special case of canonical determinant to conclude that the corresponding space of linear series on a smooth curve genus $g$ has one component of the expected dimension.
\begin{thm}\label{smoothing}[Theorem 1.2 in \cite{Osp}]
Suppose the stack of chain-adaptable, rank 2, dimension $k$, limit linear series with canonical determinant, on a chain $X_g$ of genus $g$, has expected dimension $\rr$ at one point. Then, for a general smooth curve $C$ of genus $g$, the space of rank 2, dimension $k$ linear series with canonical determinant is non-empty, with one component of the expected dimension.
\end{thm}
In our construction, all limit linear series are chain-adaptable. Moreover, in Appendix Lemma \ref{sst}, we have verified that on a non-empty open locus the data satisfy the semi-stability condition. Recalling Lemma \ref{reduce}, one can then conclude that for all pairs $(g,k)$ in Corollary \ref{final3}, the corresponding moduli space of linear series has one component of expected dimension, which contains a non-empty open locus on which the vector bundles underlying the linear series are all \textbf{stable}. This combined with the well-known existence result in \cite{M1} due to Teixidor i Bigas (See Theorem 1.1 in \cite{M1}) gives Theorem \ref{mresult}.
\appendix
\section{Technical Results in Induction Step 1}
\setcounter{thm}{0}
\renewcommand{\thethm}{\Alph{section}.\arabic{thm}}
\begin{prop}\label{AM1}
  In Lemma \ref{C1}, the matrix $c^{\Gamma}$ is determined by the following conditions:
  \begin{enumerate}
  \item $(c_{i,j})_{i\le k}=\hat{a}^{\Gamma}$.
  \item For $j=1,...,g-1$, $c_{i,2j-1}+c_{i,2j}=g+N-1$.
  \item For $j=1,...,g$, if $(a_{i,2j-2})$ and $(a_{i,2j-1})$ satisfy Lemma \ref{one} with $d=\frac{1}{2}d_j$, so do $(c_{i,2j})$ and $(c_{j,2j+1})$, with $d=\frac{1}{2}d'_j$.\footnote{Recall that $(a_{i,0})=a(k),(a_{i,2g-1})=a(k)^{\rev},(c_{i,0})=a(k+n)$,

      $(c_{i,2g-1})=(N+k_1,[N+k_1-1]_2,...,[N]_2,N-1-k_1,[N-2-k_1]_2,...,[N-m-k_1]_2,N-m-k_1-1)$.}
  \item For $j=1,...,g$, if $(a_{i,2j-2})$ and $(a_{i,2j-1})$ satisfy Lemma \ref{one2} with $d=\frac{1}{2}(d_j-1)$, so do $(c_{i,2j-2})$ and $(c_{i,2j-1})$ with $d=\frac{1}{2}(d'_j-1)$.
  \item For $i=k+1,...,k+n$, if $c_{i,2j_t-2}+c_{i,2j_t-1}=\frac{1}{2}(d'_{j_t}-1)$, then $c_{i,2j_{t+1}-2}+c_{i,2j_{t+1}-1}=\frac{1}{2}(d'_{j_{t+1}}-1)-1$; if $c_{i,2j_t-2}+c_{i,2j_t-1}=\frac{1}{2}(d'_{j_t}-1)-1$, then $c_{i,2j_{t+1}-2}+c_{i,2j_{t+1}-1}=\frac{1}{2}(d'_{j_{t+1}}-1)$.
  \item $c_{k+n,2}+c_{k+n,3}=\frac{1}{2}(d'_2-1)$.\footnote{For the definition of $n,N,d_j,d'_j$, see Lemma \ref{C1}.}
  \end{enumerate}
\end{prop}
Under these conditions, one can obtain at most one $c^{\Gamma}$ satisfying all the conditions. We shall describe this $c^{\Gamma}$ and then check that it indeed satisfies all the conditions.

Recall that in our context, the condition ``$(a_{i,2j-2})$ and $(a_{i,2j-1})$ satisfy Lemma \ref{one}'' is equivalent to $\mE_j$ being semi-stable and decomposable, and ``$(a_{i,2j-2})$ and $(a_{i,2j-1})$ satisfy Lemma \ref{one2}'' is equivalent to $\mE_j$ being unstable. We shall refer to these equivalent conditions interchangeably.
\subsection{Determining $c^{\Gamma}$}
\begin{st}{1}From condition 1, one gets $(c_{i,j})_{i\le k}=\hat{a}^{\Gamma}$.\end{st}
\begin{st}{2}Let $g'=g+N$. From condition (3), one gets $c_{i,1}=(g'-2)-\left\lfloor\frac{i-1}{2}\right\rfloor$ for $i=k+1,...,k+n$.\end{st}
\begin{st}{3}From condition 2, for $j=1,...,g$, $(c_{i,2j})=(g'-1,....,g'-1)-(c_{i,2j-1})$.\end{st}
\begin{st}{4}We claim that $(c_{i,1}),...,(c_{i,2j-2})$ uniquely determine $(c_{i,2j-1})$. Given \tb{Step 1}, it suffices to determine $(c_{i,2j-1})_{i>k}$. There are two cases to consider: (1) $\mE_j$ is semi-stable; (2) $\mE_j$ is unstable.\end{st}
\textbf{Step 4.1} If $\mE_j$ is semi-stable, from condition 3, $c_{i,2j-1}=(g'-2)-c_{i,2j-2}$ for all $i>k$.\\
\textbf{Step 4.2} For $(c_{i,2j_t-1})$, recall $j_1=2$. Given \textbf{Step 3}, $(c_{i,2})$ is determined. By conditions 4 and 6, for all $i>k$, $$c_{i,3}=\left\{\begin{aligned}
  &\frac{1}{2}(d'_2-1)-1-c_{i,2} &\text{ when }i \text{ is even}\\
  &\frac{1}{2}(d'_2-1)-c_{i,2} &\text{ when }i \text{ is odd.}
\end{aligned}\right.$$
By condition 5, $(c_{i,2j_t-1})_{j>k}$ are inductively determined for all $t$.

Thus, there is at most one $c^{\Gamma}$ satisfying conditions 1-6.
\subsection{$c^{\Gamma}$ satisfies all the conditions in \ref{AM1}.}
\indent Following \textbf{Step 1}-\textbf{Step 4}, conditions 1, 2, 5, 6 automatically hold. It remains to check (3) and (4). We break this process into several lemmas.

First, we point out some simple patterns among the entries of $c^{\Gamma}$ thus determined. They can all be checked by elementary calculation and we omit the proofs.
\begin{lem}\label{A1} Let $e=c_{k+1,2j_t-2}$. Then,\\
 $(c_{i,2j_t-2})_{i>k}=\begin{cases}
 (e,e+1,e+1,...,e+(m-1),e+(m-1),e+m) &\text{ if } t \text{ is odd}\\
 (e,e,...,e+(m-1),e+(m-1)) &\text{ if } t \text{ is even.}
 \end{cases}$
\end{lem}
\begin{cor}\label{AC1}
Let $e=c_{k+1,2j_t-1}$. Then,\\
 $(c_{i,2j_t-1})_{i>k}=\begin{cases}
 (e,e,e-1,e-1,...,e-(m-1),e-(m-1)) &\text{ if } t \text{ is odd}\\
 (e,e-1,e-1,e-2,e-2,...,e-(m-1),e-(m-1),e-m)&\text{ if } t \text{ is even.}
 \end{cases}$
\end{cor}
\begin{lem}\label{A2} Suppose $j\neq j_1,...,j_{2T}$. Let $e=c_{k+1,2j-2},f=c_{k+1,2j-1}$. Then,\\
 $(c_{i,2j-2})_{i>k}=\begin{cases}
 (e,e+1,e+1,...,e+(m-1),e+(m-1),e+m)&\text{ if } |\{j_t|j_t<j\}| \text{ is even}\\
 (e,e,...,e+(m-1),e+(m-1))&\text{ if } |\{j_t|j_t<j\}| \text{ is odd};
\end{cases}$

\noindent$(c_{i,2j-1})_{i>k}=\begin{cases}
 (f,f-1,f-1,...,f-(m-1),f-(m-1),f-m)&\text{ if } |\{j_t|j_t<j\}| \text{ is even}\\
 (f,f,...,f-(m-1),f-(m-1))&\text{ if } |\{j_t|j_t<j\}| \text{ is odd}.
 \end{cases}$
\end{lem}
\begin{lem}\label{A3}
For $j=1,...,g-1$, $(c_{i,2j-1})$ is non-increasing and $(c_{i,2j})$ is non-decreasing.
\end{lem}
\begin{proof}
To verify this, notice that if all $(c_{i,2j-1})$ are non-increasing, then all $(c_{i,2j})$ are non-decreasing. So it suffices to check the former. For any $j$, by \textbf{Step 1}, $(c_{i,2j-1})_{i\le k}=(\hat{a}_{i,1})$, and is non-increasing; by Corollary \ref{AC1} and Lemma \ref{A2}, $(c_{i,2j-1})_{i>k}$ is non-increasing as well. So it further suffices to show $c_{k,2j-1}\ge c_{k+1,2j-1}$.

Before we check this, notice that by our construction $c_{k+1,2j-1}=c_{k+1,1}-(j-1)$, for $j=2,...,g$. Also, by \textbf{Step 2}, $c_{k,1}=c_{k+1,1}=(g+n-2)-k_1$.

By condition 5 in Definition \ref{stan}, $c_{k,2j_t-1}\ge \frac{1}{2}(d'_{j_t}-1)-1-(k_1+j_t-2)$. Since $d'_{j_t}\ge 2g'-3$, $c_{k,2j_t-1}\ge g'-1-k_1-j_t=c_{k+1,1}-(j_t-1)=c_{k+1,2j_t-1}$.

Now consider $j\neq j_1,...,j_{2T}$. Let $j_t$ be the largest index among $j_1,...,j_{2T}$ smaller than $j$. For all $j':j_t<j'\le j$, $c_{k,2j'-2}+c_{k,2j'-1}\ge g'-2=c_{k+1,2j'-2}+c_{k+1,2j'-1}$. Since we have already shown $c_{k,2j_t-1}\ge c_{k+1,2j_t-1}$, one can conclude $c_{k,2j-1}\ge c_{k+1,2j-1}$.

This shows $(c_{i,2j-1})$ is non-increasing for all $j$.
\end{proof}
We now finish checking conditions 3 and 4 using the next three lemmas.
\begin{lem}\label{A4}
  The entries of $c^{\Gamma}$ are all non-negative.
\end{lem}
\begin{proof}
From \textbf{Step 1}, it is clear that $c_{i,j}\ge0$, for $i=1,...,k$ and $j=1,...,2g-2$. In particular, $c_{i,2j-1}\le g'-1$, for $i=1,...,k$. If $0\le c_{i,2j-1}\le g'-1$, then from \textbf{Step 3}, the lemma follows.

By Lemma \ref{A3}, $\fa i>k$, $c_{i,2j-1}\le c_{1,2j-1}\le g'-1$ and $(c_{i,2j-1})$ is non-increasing. It then suffices to show $c_{k+n,2j-1}\ge 0$.

From \textbf{Step 2}, \textbf{Step 4.1}, \textbf{Step 4.2}, $c_{k+n,2j-1}\ge(g'-2)-(k_1+m)-(j-1)$ for all $j$.

Since $j\le g$ and $g'-g> k_1+m$, we have $c_{k+n,2j-1}>0$. So the lemma holds.
\end{proof}
\begin{lem}\label{A5}
If $\mE_j$ is semi-stable, then $(c_{i,2j-2})$ and $(c_{j,2j-1})$ satisfy Lemma \ref{one}, with $d=g'-1$.
\end{lem}
\begin{proof}
When $\mE_j$ is semi-stable, let $i_1<i_2$ be the special indices, i.e. $$a_{i_1,2j-2}+a_{i_1,2j-1}=a_{i_2,2j-2}+a_{i_2,2j-1}=g-1.$$
It then follows from \textbf{Step 1}, \textbf{Step 2}, \textbf{Step 4.1} and explicitly computing $(c_{i,2g-2})$ that for all $j$,
$$c_{i_1,2j-2}+c_{i_1,2j-1}=c_{i_2,2j-2}+c_{i_2,2j-1}=g'-1; c_{i,2j-2}+c_{i,2j-1}=g'-2, \text{ for }i\neq i_1,i_2.$$
\indent To verify this lemma, it remains to check the following two conditions:

\begin{enumerate}
  \item[(a)] Every integer appears in $(c_{i,2j-2})$ (resp. $(c_{i,2j-1})$) at most twice.
  \item[(b)] There does not exist $i\neq i_1,i_2$ such that $c_{i,2j-2}=c_{i_1,2j-2},c_{i,2j-1}=c_{i_2,2j-1}$.
\end{enumerate}

\indent We first check (a). Since $(c_{i,0})=a(k+n)$ (\textbf{Notation 12}) satisfies (a), it remains to look at $j=2,...,g$. Since $(c_{i,2j-2})_{i\le k}=\hat{a}^{\Gamma}$, every integer appears in $(c_{i,2j-2})_{i\le k}$ at most twice. By Lemma \ref{A2}, every integer appears in $(c_{i,2j-2})_{i> k}$ at most twice. Hence, it suffices to check the following:
$$\text{either }c_{k,2j-2}<c_{k+1,2j-2}, \text{ or }c_{k-1,2j-2}<c_{k,2j-2}=c_{k+1,2j-2}<c_{k+2,2j-2}.$$

By definition, $c_{k-1,0}<c_{k,0}=c_{k+1,0}$.

For any $j>2$ such that $\mE_j$ is semi-stable, either $\exists t$ such that $j_t<j<j_{t+1}$ or $j>j_{2T}$.

First, suppose $j$ is between some $j_t$ and $j_{t+1}$.

When $t$ is odd, by the $(g,k)$-standardness of $a^{\Gamma}$, one of the following holds:
\begin{enumerate}
  \item $c_{k,2j_t-2}<c_{k+1,2j_t-2}$;
  \item $c_{k-1,2j_t-2}<c_{k,2j_t-2}=c_{k+1,2j_t-2}<c_{k+2,2j_t-2}$ and $c_{k,2j_t-1}>c_{k+1,2j_t-1}$.
\end{enumerate}
In either case, $c_{k,2j_t-2}+c_{k,2j_t-1}=\frac{1}{2}(d'_{j_t}-1)$ and therefore $c_{k,2j_t-1}>c_{k+1,2j_t-1}$. $\fa j:j_t<j<j_{t+1}$, $$c_{k,2j-2}+c_{k,2j-1}\ge g'-2=c_{k+1,2j-2}+c_{k+1,2j-1}.$$
Thus we have $c_{k,2j-2}<c_{k+1,2j-2}$, for all $j:j_t<j<j_{t+1}$.

When $t$ is even, $c_{k,2j_t-2}<c_{k+1,2j_t-2}$. If $c_{k,2j_t-1}>c_{k+1,2j_t-1}$, we are done. Otherwise we have $c_{k,2j_t-1}=c_{k+1,2j_t-1}$ and $c_{k,2j_t-2}+c_{k,2j_t-1}=\frac{1}{2}(d'_{j_t}-1)-1$. By the monotonicity of $(c_{i,2j_t-2})$ and the fact that $c_{k-1,2j_t-2}+c_{k-1,2j_t-1}\ge\frac{1}{2}(d'_{j_t}-1)-1$, one must have $c_{k-1,2j_t-1}>c_{k,2j_t-1}$. Therefore, $c_{k-1,2j-2}<c_{k,2j-2}\le c_{k+1,2j-2}<c_{k+2,2j-2}$, for all $j:j_t<j<j_{t+1}$.

For $j>j_{2T}$, a similar argument applies.

Thus, we have checked that every integer appears in $(c_{i,2j-2})$ at most twice. For $j=1,...,g-1$, $(c_{i,2j-1})=(g'-1,....,g'-1)-(c_{i,2j})$, so every integer appears in $(c_{i,2j-1})$ at most twice. And $(c_{i,2g-1})$ visibly satisfies this condition. Hence, (a) holds for all $j$.

Lastly, (b) could only be violated by some $i$ between $i_1$ and $i_2$. But $(c_{i,j})_{i\le k}=\hat{a}^{\Gamma}$, so this is impossible.
\end{proof}
This verifies condition 3 in \ref{AM1}. Eventually, we check condition 4:
\begin{lem}\label{A6}
$\fa t$, $(c_{i,2j_t-2}), (c_{i,2j_t-1})$ satisfy Lemma \ref{one2} with $d=\frac{1}{2}(d'_{j_t}-1)$.
\end{lem}
\begin{proof}
It is obvious that $d+1\ge c_{i,2j_t-2}+c_{i,2j_t-1}\ge d-1$ and the special indices $\ell,i^*$ (Remark \ref{spid}) are inherited from $(a_{i,2j_t-2})$, $(a_{i,2j_t-1})$.

By Lemma \ref{A3}, $(c_{i,2j_t-2})$ is non-decreasing and $(c_{i,2j_t-1})$ is non-increasing. Hence, we are left to check the following:
\begin{enumerate}
  \item[(a)] Any integer appears in $(c_{i,2j_t-2})$ (resp. $(c_{i,2j_t-1})$) at most twice.
  \item[(b)] $\fa (a,c)\neq (c_{i^*,2j_t-2},c_{i^*,2j_t-1})$, $|\{i|(c_{i,2j_t-2},c_{i,2j_t-1})=(a,c)\}|\le 1$.
  \item[(c)] $\fa (a,d-a)\neq (c_{i^*,2j_t-2},c_{i^*,2j_t-1})$, $|\{i|(c_{i,2j_t-2},c_{i,2j_t-1})\ge(a,d-a)\}|\le 1$.
  \item[(d)] $\not\exists$ $i$ such that $c_{i,2j_t-2}+c_{i,2j_t-1}=c_{i+1,2j_t-2}+c_{i+1,2j_t-1}=d-1$ and $c_{i+1,2j_t-2}=c_{i,2j_t-2}+1$, where $c_{i,2j_t-2},c_{i+1,2j_t-2},c_{i,2j_t-1},c_{i+1,2j_t-1}$ are non-repeated vanishing orders.
\end{enumerate}

We verify (a) by induction on $t$. Recall $j_1=2$. By Lemma \ref{A5}, every integer appears in $(c_{i,1})$ at most twice, so every integer appears in $(c_{i,2})=(g'-1,...,g'-1)-(c_{i,1})$ at most twice. By \textbf{Step 4.2}, one can check directly that every integer appears in $(c_{i,3})$ at most twice. Now suppose every integer appears in $(c_{i,2j_t-2})$, $(c_{i,2j_t-1})$ at most twice. By Lemma \ref{A5}, every integer appears in $(c_{i,2j-2}),(c_{i,2j-1})$ at most twice, for any $j$ between $j_t$ and $j_{t+1}$; hence, the same holds for $(c_{i,2j_{t+1}-2})$.

As for $(c_{i,2j_{t+1}-1})$, it is non-increasing by Lemma \ref{A3}. By \textbf{Step 1} and \textbf{Step 4.2} resp., every integer appears in $(c_{i,2j_{t+1}-1})_{i\le k}$ at most twice, and in $(c_{i,2j_{t+1}-1})_{i>k}$ at most twice as well. Therefore, if $c_{k,2j_{t+1}-1}>c_{k+1,2j_{t+1}-1}$, we are done. Now suppose $c_{k,2j_{t+1}-1}=c_{k+1,2j_{t+1}-1}$. We claim $t$ must be even. First of all, $c_{k,3}>c_{k+1,3}$. When $t=2p+1$ is odd, by \textbf{Step 4.2} and $(g,k)$-standardness of $a^{\Gamma}$, there are at least $p+1$ many $j\in\{j_1,..,j_{2p+1}\}$ such that $c_{k,2j-2}+c_{k,2j-1}=\frac{1}{2}(d'_j-1)$, while there are at least $p+1$ many $j$'s such that $c_{k+1,2j-2}+c_{k+1,2j-1}=\frac{1}{2}(d'_j-1)-1$; moreover, $c_{k,2j-2}+c_{k,2j-1}\ge c_{k+1,2j-2}+c_{k+1,2j-1}$ for any $j\neq j_1,...,j_{2T}$. Therefore, when $t$ is odd, $c_{k,2j_t-1}>c_{k+1,2j_t-1}$. But given $t$ is even, (a) is a consequence of Corollary \ref{AC1}.

To check (b), it only remains to see $(c_{k,2j_t-2},c_{k,2j_t-1})\neq (c_{k+1,2j_t-2},c_{k+1,2j_t-1})$. This holds: by Lemma \ref{A1} and Corollary \ref{AC1}, either $c_{k+1,2j_t-2}=c_{k+2,2j_t-2}$ or $c_{k+1,2j_t-1}=c_{k+2,2j_t-1}$; if $(c_{k,2j_t-2},c_{k,2j_t-1})=(c_{k+1,2j_t-2},c_{k+1,2j_t-1})$, condition (a) would be violated.

For (c), it could only be violated if $$(c_{k,2j_t-2},c_{k,2j_t-1})=(c_{k+1,2j_t-2},c_{k+1,2j_t-1}+1)=(a,d+1-a).$$
This cannot happen: $c_{k,2j_t-1}=c_{k+1,2j_t-1}$ only if $t$ is odd, otherwise it would violate (a); but when $t$ is odd, by \textbf{Step 4.2}, $c_{k+1,2j_t-2}+c_{k+1,2j_t-1}=d-1$.

For (d), it suffices to see that the index $k$ does not violate this condition. When $t$ is odd, by $(g,k)$-standardness of $a^{\Gamma}$, either $c_{k,2j_t-2}=c_{k+1,2j_t-2}$ or $c_{k,2j_t-2}+c_{k,2j_t-1}=d$; when $t$ is even, $c_{k+1,2j_t-2}+c_{k+1,2j_t-1}=d$. In any case, (d) holds.

This proves the lemma.
\end{proof}
Thus, we have verified Proposition \ref{AM1}. Consequently, we have:
\begin{cor}\label{C2}
  All $K$-valued objects $((\mE'_i,V'_i),(\phi_i))$ in $\mathcal{G}^0$ (see Theorem \ref{Induction}) have same underlying vector bundles. Moreover, $((\mE'_i,V'_i),(\phi_i))$ is chain-adaptable.
\end{cor}
\subsection{Sub-bundles of Push-forwards}\label{secsub}
To discuss objects in \ma\ over an arbitrary $K$-scheme $T$, we quote the following notion of a generalized sub-bundle from \cite{Olls}:
\begin{defn}\label{subb}
  Let $\pi:X\to B$ be a morphism of schemes and $\mE$ be a vector bundle on $X$. A sub-sheaf $V\subset\pi_*\mE$ is said to be a \tb{sub-bundle} if it is locally-free and $V_S\to\pi_{S*}\mE_S$ is injective for all base-changes $S\to B$.\footnote{$V_S\to\pi_{S*}\mE_S$ is the composition $V_S\to(\pi_*\mE)_S\to\pi_{S*}\mE_S$.}
\end{defn}
In this subsection, we show some sub-sheaves of push-forwards of vector bundles are sub-bundles in the sense of \ref{subb}. We first recall some basic definitions:
\begin{defn}
  Let $f:\mE\to\mF$ be a morphism between two locally-free sheaves on a scheme $X$. We say that $f$ \it{has rank $r$} if $Z(\Lambda^{r+1}f)=X$ and $\supp(\Lambda^rf)=\emptyset$; in this case, we also denote $\rk(f)=r$.
\end{defn}
\begin{defn}\label{van}
  Let $\mE$ be a locally-free sheaf on $T\times C$, where $T$ is a $K$-scheme $C$ is a smooth projective curve and $P$ is a point on $C$. Suppose $V$ is a rank-$k$ sub-bundle of $\pi_*\mE$ in the sense of \ref{subb}. Denote $\gamma_m:V\to\pi_*(\mE|_{m(T\times P)})$ to be the obvious map factorizing through $\pi_*\mE$.

  We say a sequence of integers $(c_1\le...\le c_k)$ is \tb{the vanishing sequence of $V$ along $P$}, if the following conditions are satisfied:
  \begin{enumerate}
    \item $\fa i,\rk(\gamma_{c_i})=|\{t|c_t<c_i\}|$;
    \item $\fa i,\rk(\gamma_{c_i+1})-\rk(\gamma_{c_i})=|\{t|c_t=c_i\}|$.
  \end{enumerate}
\end{defn}
Obviously, sub-bundles in general may not have a vanishing sequence along $P$. Also, if $T$ is a point, this agrees with the usual notion of vanishing sequence of a linear system over $C$ at $P$.
\begin{lem}\label{case1}
Let $T$ be a $K$-scheme, $C$ be a smooth projective curve and $\mE$ be a rank-$r$ vector bundle on $T\times C$. Denote $\pi:T\times C\to C$ to be the projection map. Given a point $P\in C$, suppose $V$ is a sub-bundle of $\pi_*\mE$ which has a fixed vanishing sequence $(c_i)$ along $P$. Define $\gamma_m:V\to\pi_*(\mE|_{m(T\times P)})$ to be the obvious map factoring through $\pi_*\mE$. Then, $V(-mP):=\ker(\gamma_m)$ is a sub-bundle of $\pi_*\mE$.
\end{lem}
\begin{proof}
This essentially follows from Proposition 20.8 in \cite{EComm} unraveling our definition of the vanishing sequence of a sub-bundle.
\end{proof}
\begin{cor}\label{case2}
  Given the same setup as \ref{case1}, let $\mL$ be a sub-line bundle of $\pi_*(\mE(-m(T\times P))|_{T\times P})$ and define $\gamma_{m,\mL}$ to be the map $V(-mP)\to\pi_*(\mE(-m(T\times P))|_P)/\mL$. If $\gamma_{m,\mL}$ is surjective, then $\ker(\gamma_{m,\mL})$ is a sub-bundle of $\pi_*(\mE)$ of rank $\rk(V(-mP))-\rk(\mE)+1$.
\end{cor}
\begin{proof}
  Given Lemma \ref{case1}, when $\gamma_{m,\mL}$ is surjective, it is a surjective map from a locally-free sheaf to a locally-free sheaf of rank $\rk(\mE)-1$. Hence, $\ker(\gamma_{m,\mL})$ is a sub-bundle of $V(-mP)$ of rank $\rk(V(-mP))-\rk(\mE)+1$ in the usual sense. Since $V(-mP)$ is a sub-bundle of $\pi_*(\mE)$ in the sense of \ref{subb}, the result follows.
\end{proof}
\subsection{Some Locally Closed Points of the Moduli of Vector Bundles}
\begin{lem}\label{cl1}
    Let $C$ be a smooth projective curve and $X$ be the spectrum of a DVR, with $\eta$ being the generic point and $0$ being the closed point. Suppose $\mL_1,\mL_2$ are two line bundles on $C$ and $\mE$ is a rank two locally-free sheaf on $X\times C$. Denote $p_1:X\times C\to X$, $p_2:X\times C\to C$ be the projection maps. If $\mE_{\eta}\cong p_2^*\mL_1\op p_2^*\mL_2|_{\eta}$, then $h(\mL_i^{\vee}\ot\mE_0)\ge1$ ($i=1,2$). Moreover, if $\mL_1\cong\mL_2$, $h(\mL_i^{\vee}\ot\mE_0)\ge2$.
\end{lem}
\begin{proof}
  Denote $C_{\eta}$, $C_0$ to be the fibers of $X\times C\to X$ over the generic fiber and the special fiber respectively. Since $\mE_{\eta}\cong p_2^*\mL_1\op p_2^*\mL_2|_{\eta}$, $h(C_{\eta},p_2^*\mL_i|_{\eta}^{\vee}\ot\mE_{\eta})=h(C_{\eta},(p_2^*\mL_i^{\vee}\ot\mE)_{\eta})\ge1$.

  By the Semi-continuity Theorem, $h(C_0,(p_2^*\mL_i^{\vee}\ot\mE)_0)=h(C_0,p_2^*\mL_i|_0^{\vee}\ot\mE_0)\ge1$. In other words, $h(\mL_i^{\vee}\ot\mE_0)\ge1$ for $i=1,2$.

  If $\mL_1\cong\mL_2$, $\mE_{\eta}\cong p_2^*\mL_1\op p_2^*\mL_2|_{\eta}$ implies there are two linearly independent maps from $p_2^*\mL_1|_{\eta}$ to $\mE_{\eta}$, i.e. $h(C_{\eta},(p_2^*\mL_1^{\vee}\ot\mE)_{\eta})\ge2$. Hence, $h(\mL_1^{\vee}\ot\mE_0)\ge2$.
\end{proof}
\begin{lem}\label{cl2}
Same setup as in Lemma \ref{cl1}. Suppose that $0\le\deg(\mL_2)-\deg(\mL_1)\le 1$, $u(\mE_0)\le{1\over2}$ (see \tb{Convention}) and $\det(\mE_0)=\mL_1\ot\mL_2$. Then, $\mE_0\cong\mL_1\op\mL_2$.
\end{lem}
\begin{proof}
First, consider the case where $\deg(\mL_2)-\deg(\mL_1)=1$. By Lemma \ref{cl1}, take a non-zero map $\phi:\mL_2\to\mE_0$ be some non-zero morphism. Since an indecomposable vector bundle of negative degree has no sections, $\mE_0$ must be decomposable. By the determinant condition and that $u(\mE_0)\le{1\over2}$, conclude $\mE_0\cong\mL_1\op\mL_2$.

Now consider the case $\deg(\mL_2)=\deg(\mL_1)$ and $\mL_1\not\cong\mL_2$. By \ref{cl1}, we have a map $\mL_1\op\mL_2\to\mE_0$ whose restriction to every summand is non-zero. The image $Q$ of this map cannot be rank one, because otherwise $\deg(Q)\le\deg(\mL_1)$ by the bound on the unstability of $\mE_0$; since $\mL_1\not\cong\mL_2$, they cannot both map non-trivially to $Q$. Thus, $\rk(Q)=2$ and $\mL_1\op\mL_2$ injects into $\mE_0$. Since $\deg(\mE_0)=\deg(\mL_1\op\mL_2)$, $\mE_0\cong\mL_1\op\mL_2$.

Eventually, assume $\deg(\mL_2)=\deg(\mL_1)$ and $\mL_1\cong\mL_2$. Since $h^0(\mL_1^{\vee}\ot\mE_0)\ge2$, there exists an injective map $\mL_1^{\op2}\to\mE_0$. Since $\deg(\mL_1^{\op2})=\deg(\mE_0)$, we still get the claim.
\end{proof}
We now show certain points of $\prod_{j=1}^g\mathcal{M}_{2,\mL_j}(C_j)$ are locally-closed, where $\mathcal{M}_{2,\mL}(C_j)$ is the moduli stack of rank two vector bundles with fixed determinant $\mL_j$ on some elliptic curve $C_j$.
\begin{prop}\label{applc}
   For $j=1,...,g$, let $\mE_j$ be a rank two vector bundle on $C_j$ satisfying the following properties:
  $$\mE_j=\mL_{j,1}\op\mL_{j,2},\mL_{j,1}\ot\mL_{j,2}\cong\mL_j,0\le\deg(\mL_{j,2})-\deg(\mL_{j,1})\le1.$$
  Then, the point of $\mathcal{Z}$ of $\prod_{j=1}^g\mathcal{M}_{2,\mL_j}(C_j)$ corresponding to the tuple $(\mE_1,...,\mE_g)$ is locally-closed.
\end{prop}
\begin{proof}
  Let $\mathcal{U}_j$ be the sub-stack of $\mathcal{M}_{2,\mL_j}(C_j)$ of objects with bounded in-stability: $u(\mE)\le{1\over2}$. This is known to be an open condition (see \cite{GMS}).\footnote{This essentially follows from Proposition 2.3.1 in \cite{GMS}.} Hence, $\mathcal{U}:=\prod \mathcal{U}_j$ is an open sub-stack of $\prod\mathcal{M}_{2,\mL}(C_j)$. Suppose $x:\spec(K)\to\mathcal{U}$ represents a point in $|\mathcal{U}|$ which corresponds to a tuple of vector bundles $(\mE_1,...,\mE_g)$ satisfying the properties stated above. Claim: $x$ is a closed point in $|\mathcal{U}|$.

  By Proposition 7.2.2 in \cite{Champs}, suppose $x$ specializes to $y$ in $|\mathcal{U}|$, there is morphism $\spec(R)\to\mathcal{U}$, where $R$ is a discrete valuation ring whose closed point is mapped to $y$, fitting into the following 2-commutative diagram:$$\xymatrix{\spec(K')\ar[r]\ar[d]&\spec(K)\ar[r]^x&\mathcal{U}\ar[d]\\\spec(R)\ar[rr]&&\mathcal{U}}$$ where $K'$ is some field extension of $K$ containing $R$. Recall also that a morphism $\spec(R)\to\mathcal{U}$ corresponds to an object in $\mathcal{U}(\spec(R))$, i.e. a collection of (isomorphism classes of) vector bundles $\ti{\mE}_1,...,\ti{\mE}_g$, with $\ti{\mE}_j$ over $C_j\times\spec(R)$. Moreover, let $\eta$ be the generic point of $\spec(R)$, up to a base-change by $\spec(K')\to\eta$, $\ti{\mE}_j|_{\eta}$ is isomorphic to some decomposable vector bundle. We claim that $\ti{\mE}_j|_{\eta}$ is indeed decomposable. Note that $\mE_j|_{\eta}\ot K'\cong (\mL_{j,1}\op\mL_{j,2})\ot K'$, where $\mL_{j,1}$ and $\mL_{j,2}$ are line bundles over $K$. In particular, $\hom(\mL_{j,i}\ot K',\mE_j|_{\eta}\ot K')\ge 1$, and when $\mL_{j,1}\cong\mL_{j,2}$ $\hom(\mL_{j,i}\ot K',\mE_j|_{\eta}\ot K')\ge 2$. Denote $K''$ to be the fraction field of $R$ and we have $K'\supset K''\supset K$. By restriction of scalars, $\hom(\mL_{j,i}\ot K'',\mE_j|_{\eta})\ge 1$, and when $\mL_{j,1}\cong\mL_{j,2}$ $\hom(\mL_{j,i}\ot K'',\mE_j|_{\eta})\ge 2$. Using a similar argument as in the proof of Lemma \ref{cl2}, one gets the claim.

  It is clear that the proof of the closedness of $x$ in $\mathcal{U}$ reduces to the case $g=1$. In that case, Lemma \ref{cl2} implies that $y=x$ and hence we are done.
\end{proof}
\section{Technical Results in Induction Step 2}
\subsection{On the $(g,k)$-standardness Conditions}
\begin{lem}\label{standard}
  The vanishing condition $c^{\Gamma'}$ in section \ref{l1} is $(g',k')$-standard.
\end{lem}
\begin{proof}
In Appendix A we have checked that for $j\le g$, $(c_{i,2j-2}),(c_{i,2j-1})$ satisfy conditions 1-4 in Definition \ref{stan}. For $j>g$, the vanishing sequence are determined by the data given in \ref{l1}. Concretely, given $(c_{i,2j'-2}),(c_{i,2j'-1})$ for all $j'<j$, then $(c_{i,2j-2})$ is determined by condition 2 in \ref{stan} and $(c_{i,2j-1})$ is calculated following Lemma \ref{concise1} or \ref{concise2}. Thus, to verify 1-4 in \ref{stan} for $c^{\Gamma'}$, it remains to check $(c_{i,2j-2}), (c_{i,2j-1})$ satisfy the conditions in Lemma \ref{one} or Lemma \ref{one2}.

Inductively, we check that every integer appears in $(c_{i,2j-1})$ at most twice (which then implies the same property for $(c_{i,2j})$). Assume this is true for $(c_{i,2j-2})$. When $\mE_j=\oplus_{s=1}^2\mO(i_sP_j+(d-i_s)P_{j+1})$, by \ref{concise2} it suffices to see $c_{i_1,2j-2}-1$, $c_{i_2,2j-2}-1$ are non-repeated vanishing orders (if $c_{i_1,2j-2}=c_{i_2,2j-2}$, $c_{i_1,2j-2}-1$ should not be in $(c_{i,2j-2})$). When $\mE_j=\mO(c_{\ell,2j-2},(d+1)-c_{\ell,2j-2})\op\mO(c_{i,2j-2},d-c_{i,2j-2})$, it suffices to check $c_{\ell,2j-1}$ appears at most twice and $c_{i-1,2j-1}>c_{i,2j-1}$. We verify these by listing the special indices and comparing the relevant vanishing orders:\\
\ \\
\begin{tabular}{|c|c|c|}
  \hline
  $j$ & special indices &relevant vanishing orders\\
  \hline
$g+1$& $i_1=2,i_2=k+1$ & $c_{k,2g}\le c_{k+1,2g}-3<c_{k+2,2g}-3$\\

 $g+2$& $i_1=1,i_2=k+2$ & $c_{k+1,2g+2}= c_{k+2,2g+2}-2$\\

 $g+3$& $i_1=2,i_2=k+3$ & $c_{k+1,2g+4}< c_{k+2,2g+4}=c_{k+3,2g+4}$\\
 $g+4$& $\ell=3,i=k+2$ & $c_{1,2g+7}>c_{3,2g+7}$, $c_{k+1,2g+7}=c_{k+2,2g+7}+1$\\
 $g+5$& $i_1=4,i_2=k+2$ & $c_{3,2g+8}=c_{4,2g+8}-2$, $c_{k,2g+8}<c_{k+1,2g+8}<c_{k+2,2g+8}$\\
 $g+6$& $i_1=5,i_2=k+1$ &  $c_{3,2g+10}<c_{4,2g+10}<c_{5,2g+10}$\\
                          && $c_{k-1,2g+10}<c_{k,2g+10}<c_{k+1,2g+10}$\\
 $g+7$& $i_1=5,i_2=k+3$ & $c_{3,2g+12}<c_{4,2g+12}<c_{5,2g+12}<c_{6,2g+12}$\\
                          &&$c_{k+1,2g+12}<c_{k+2,2g+12}<c_{k+3,2g+12}$\\
 $g+8$& $\ell=k+4,i=2$ & $c_{1,2g+15}>c_{2,2g+15}$, $c_{k+3,2g+15}=c_{k+4,2g+15}+1$\\
 $g+9$& $i_1=5,i_2=k+4$ & $c_{3,2g+16}<c_{4,2g+16}<c_{5,2g+16}$\\
                          &&$c_{k+2,2g+16}<c_{k+3,2g+16}<c_{k+4,2g+16}$\\
                          \hline
\end{tabular}\\
For $s=0,...,q-2$,\footnote{This part of the construction is only relevant when $k_1\ge6$.} we have:\\
\ \\
\begin{tabular}{|c|c|c|}
  \hline
  $j$ & special indices &relevant vanishing orders\\
  \hline
$g+10+11s$& $\ell=k+1,i=6+6s$ & $c_{5+6s,2j-1}=c_{6+6s,2j-1}+2$, $c_{k,2j-1}>c_{k+1,2j-1}$\\

$g+11+11s$& $i_1=6+6s,i_2=k+2$ & $c_{5+6s,2j-2}=c_{6+6s,2j-2}+2$, $c_{k+1,2j-2}=c_{k+2,2j-2}+2$\\

$g+12+11s$& $i_1=7+6s,i_2=k+3$ & $c_{5+6s,2j-2}<c_{6+6s,2j-2}<c_{7+6s,2j-2}$ \\
&&$c_{k+1,2j-2}<c_{k+2,2j-2}<c_{k+3,2j-2}$\\

$g+13+11s$& $\ell=8+6s,i=k+2$ & $c_{6+6s,2j-1}>c_{8+6s,2j-1}$, $c_{k+1,2j-1}>c_{k+2,2j-1}$\\

$g+14+11s$& $i_1=6s+9,i_2=k+2$ & $c_{8+6s,2j-2}=c_{9+6s,2j-2}-2$,\\
&& $c_{k,2j-2}<c_{k+1,2j-2}<c_{k+2,2j-2}=c_{k+3,2j-2}$\\
$g+15+11s$& $i_1=6s+7,i_2=k+4$ & $c_{7+6s,2j-2}<c_{8+6s,2j-2}<c_{9+6s,2j-2}$,\\
&&$c_{k+3,2j-2}=c_{k+4,2j-2}-2$\\
$g+16+11s$& $i_1=6s+10,i_2=k+1$ & $c_{9+6s,2j-2}=c_{10+6s,2j-2}+2$, $c_{k,2j-2}\le c_{k+1,2j-2}-2$\\
$g+17+11s$& $i_1=6s+11,i_2=k+2$ & $c_{9+6s,2j-2}<c_{10+6s,2j-2}<c_{11+6s,2j-2}$ \\
&&$c_{k,2j-2}<c_{k+1,2j-2}<c_{k+2,2j-2}$\\
$g+18+11s$& $i_1=6s+9,i_2=k+4$ & $c_{7+6s,2j-2}<c_{8+6s,2j-2}<c_{9+6s,2j-2}$ \\
&&$c_{k+2,2j-2}<c_{k+3,2j-2}<c_{k+4,2j-2}$\\
$g+19+11s$& $i_1=6s+10,i_2=k+3$ & $c_{9+6s,2j-2}=c_{8+6s,2j-2}-2$, $c_{k+2,2j-2}=c_{k+3,2j-2}-2$\\
$g+20+11s$& $i_1=6s+11,i_2=k+4$ & $c_{9+6s,2j-2}<c_{10+6s,2j-2}<c_{11+6s,2j-2}$ \\
&&$c_{k+2,2j-2}<c_{k+3,2j-2}<c_{k+4,2j-2}$\\
\hline
\end{tabular}\\
Denote $h=g+11q-2$ and for $p=0,...,k_1-3q$,\footnote{This part of the construction is only relevant when $k_1\ge3$.} we have:\\
\ \\
\begin{tabular}{|c|c|c|}
  \hline
  $j$ & special indices &relevant vanishing orders\\
  \hline
 $h+4p+1$& $i_1=6q+2p,i_2=k+1$& $c_{6q+2p-1,2j-2}=c_{6q+2p,2j-2}-3$, $c_{k,2j-2}=c_{k+1,2j-2}-2$\\
 $h+4p+2$& $i_1=6q+2p+1,i_2=k+2$& $c_{i_1-2,2j-2}<c_{i_1-1,2j-2}<c_{i_1,2j-2}$ \\
 &&$c_{k,2j-2}<c_{k+1,2j-2}<c_{k+2,2j-2}$\\
 \hline
 \end{tabular}
 
 \begin{tabular}{|c|c|c|}
  \hline
 $h+4p+3$& $i_1=6q+2p,i_2=k+3$& $c_{6q+2p-1,2j-2}=c_{6q+2p,2j-2}-2$\\
 && $c_{k+2,2j-2}=c_{k+3,2j-2}-2$\\
 $h+4p+4$& $i_1=6q+2p+1,i_2=k+4$& $c_{i_1-2,2j-2}<c_{i_1-1,2j-2}<c_{i_1,2j-2}$ \\ &&$c_{k+2,2j-2}<c_{k+3,2j-2}<c_{k+4,2j-2}$\\
 $g'-3$& $i_1=k+1,i_2=k+2$& $c_{k,2g'-8}=c_{k+1,2g'-8}-3=c_{k+2,2g'-8}-3$\\
 $g'-2$& $i_1=k+1,i_2=k+3$& $c_{k,2g'-6}=c_{k+1,2g'-6}-2$, $c_{k+2,2g'-6}=c_{k+3,2g'-6}-2$\\
 $g'-1$& $i_1=k+2,i_2=k+4$& $c_{k,2g'-4}<c_{k+1,2g'-4}<c_{k+2,2g'-4}<c_{k+4,2g'-4}$\\
 $g'$& $i_1=k+3,i_2=k+4$& $c_{k+2,2g'-2}=c_{k+3,2g'-2}-2=c_{k+4,2g'-2}-2$\\
\hline
\end{tabular}\\
\ \\
\indent For all $j$ such that $\mE_j$ is semi-stable, one can check that either (1) $i_2-i_1>2$, or (2) $i_2-i_1=1$, or (3) $i_2-i_1=2$ but $c_{i_2,2j-2}-c_{i_1,2j-2}\ge2$. In particular, $\not\exists i\neq i_1,i_2$ such that $c_{i,2j-2}=c_{i_1,2j-2}$, $c_{i,2j-1}=c_{i_2,2j-1}$. So, all assumptions in \ref{one} are satisfied.

For all $j$ such that $\mE_j$ is unstable, conditions 1, 3, 4 in Lemma \ref{one2} are incorporated in the rules in Lemma \ref{concise2}, so it only remains to show there does not exist some $i$ such that $c_{i,2j-2}+c_{i,2j-1}=c_{i+1,2j-2}+c_{i+1,2j-1}=d-1$, $c_{i+1,2j-2}=c_{i,2j-1}+1$ and $c_{i,2j-2},c_{i+1,2j-2},c_{i,2j-1},c_{i+1,2j-1}$ are all non-repeated. To do so, we look at vanishing sequences in \ref{parta}, \ref{parte}:

For $j=g+4$, $\not\exists$ $i$ such that $c_{i,2j-2}$ is non-repeated and $c_{i,2j-2}+c_{i,2j-1}=d-1$.

For $j=g+8$, the $i$'s such that $c_{i,2j-2}$ is non-repeated and $c_{i,2j-2}+c_{i,2j-1}=d-1$ are $i=6,k+1$.

For $j=g+10+11s$ ($s=0,...,q-2$), $\not\exists$ $i$ such that $c_{i,2j-2}$ is non-repeated and $c_{i,2j-2}+c_{i,2j-1}=d-1$.

For $j=g+13+11s$ ($s=0,...,q-2$), the $i$'s such that $c_{i,2j-2}$ is non-repeated and $c_{i,2j-2}+c_{i,2j-1}=d-1$ are $i=1,2$; yet $c_{2,2j-1}$ is a repeated vanishing order.

We have thus verified conditions 1-4 in Definition \ref{stan}.

\indent To check condition 5 in \ref{stan}, notice that for $j_t\le g$, by \textbf{Step 4.2},
$$c_{k+4,2j_t-2}=\begin{cases}
(k_1+2)+j_t-1 \text{ when }t\text{ is odd}\\
(k_1+2)+j_t-2 \text{ when }t\text{ is even}.
\end{cases}$$
By Lemma \ref{A1}, when $t$ is odd, $c_{k+3,2j_t-2}<c_{k+4,2j_t-2}$. Thus, for $j_t\le g$, condition 5 holds. Given the vanishing sequences in \ref{parta}, \ref{parte}, one can compute:

For $j=g+4$, $c_{k+4,2j-2}=(k_1+2)+(g+4)-1$ and $c_{k+3,2j-2}<c_{k+4,2j-2}$.

For $j=g+8$, $c_{k+4,2j-2}=(k_1+2)+(g+8)-2$ and $g+8=j_{2t}$ for some $t$.

For $j=g+10+11s$ ($s=0,...,q-2$), $c_{k+4,2j-2}=g+k_1+8+8s<(k_1+2)+j-2$.

For $j=g+13+11s$ ($s=0,...,q-2$), $c_{k+4,2j-2}=g+k_1+11+8s<(k_1+2)+j-2$.

Hence, we have verified condition 5.

Eventually, condition 6 trivially holds since by \tb{Step 4.1} $c_{k+4,2j-2}+c_{k+4,2j-1}=g'-2$ for all $j:j_t<j_{t+1}$.

Therefore, $c^{\Gamma'}$ is $(g',k')$-standard.
\end{proof}
Also, for the second construction:
\begin{lem}\label{standard2}
The $c^{\Gamma'}_{g'}$ constructed in section \ref{l2} is $(g',k+2)$-standard.
\end{lem}
\begin{proof}
  This is simpler than proving \ref{standard}. For $j\le g$, by the same argument as in the proof of \ref{standard}, one verifies the needed conditions for data over $C_1,...,C_g$ using Lemma \ref{A1}-\ref{A6}. For $j>g$, we claim all $\mE_j$ are semi-stable and $(c_{i,2j-2}),(c_{i,2j-1})$ satisfy \ref{one} with $d=g'-1$. This can be checked in the same way as in \ref{standard}:

  For $j=g+1+t$ ($t=0,...,k_1-1$), the special indices are $i_1=2+2t$ and $i_2=k+1$, $c_{2t,2j-2}<c_{2t+1,2j-2}<c_{2t+2,2j-2}$, $c_{k,2j-2}\le c_{k+1,2j-2}-2$.

  For $j=g+1+t+k_1$ ($t=0,...,k_1$), $i_1=1+2t$, $i_2=k+2$, $c_{2t-1,2j-2}<c_{2t,2j-2}<c_{2t+1,2j-2}$ and $c_{k,2j-2}<c_{k+1,2j-2}<c_{k+2,2j-2}$.

  For $j=g'$, $c_{k,2g'-2}=c_{k+1,2g'-2}-2=c_{k+2,2g'-2}-2$.

 \indent It remains to check conditions 5, 6 in Definition \ref{stan}. But in this case $\mE_j$ is unstable only if $j<g$. Therefore, using the same argument as in Lemma \ref{standard} (for $j\le g$ ), we get the result.
\end{proof}
\begin{rem}
  For our third construction, we do not need the $c^{\Gamma'}_{g'}$ to be $(g',k')$-standard. Existence results for the cases where $k$ is even are deduced from cases where $k$ is odd. Moreover, the feasibility of the third construction follows easily:
\end{rem}
\begin{lem}\label{standard3}
 The vanishing sequences $(c_{i,2j-2}), (c_{i,2j-1})$ determined in \ref{l3} satisfy conditions in Lemma \ref{one} (resp. \ref{one2}) if $\mE_j$ is semi-stable (resp. unstable).
\end{lem}
\begin{proof}
 Notice that by our construction, $c^{\Gamma'}_{g'}$ consists of the first $k+3$ columns of a $(g',k+4)$-standard vanishing condition obtained by the first construction. All conditions in \ref{one} or \ref{one2} are inherited by such a sub-matrix except for 5 in \ref{one2}. Since either $c_{k+2,2j_t-2}=c_{k+3,2j_t-2}$ or $c_{k+2,2j_t-1}=c_{k+3,2j_t-1}$ always holds, the lemma follows.
\end{proof}
\subsection{On the Canonical Determinant Condition}
Secondly, we show that the canonical determinant condition is preserved under the inductions.
\begin{lem}\label{can1}
Any limit linear series obtained in the first construction satisfies the canonical determinant condition.
\end{lem}
\begin{proof}Notice that this is a condition on the vector bundles $\mE_1,...,\mE_{g'}$.

By Corollary \ref{candet}, the canonical determinant condition can be checked by looking at the coefficient of $P_j$ in $\det(\mE_j)=\mO(a_jP_j+b_jP_{j+1})$. For $j\le g$, $\mE_j=\mE'_j(NP_{j+1})$, where $\mE'_1,...,\mE'_g$ are the underlying vector bundles of objects in \ma.\ Therefore, $\det(\mE_j)$ satisfies the numerical condition in \ref{candet} for $j\le g$.

For $j>g$, one can compute directly that

$$\det(\mE_j)=\begin{cases}\mO((2j-3)P_j+(2g'-2j)P_{j+1})&\text{ for }j=g+8,g+13+11s\ (s=0,...,q-2),\\
\mO((2j-2)P_j+(2g'-2j+1)P_{j+1})&\text{ for }j=g+4,g+10+11s\ (s=0,...,q-2),\\
\mO((2j-3)P_j+(2g'-2j+1)P_{j+1})&\text{ for }j=g+5,g+6,g+7,g+11+11s,g+12+11s,\\
\mO((2j-2)P_j+(2g'-2j)P_{j+1})&\text{ otherwise}.
\end{cases}$$

This verifies the canonical determinant condition for our first construction.
\end{proof}
\begin{lem}\label{can2}
Any limit linear series obtained in the second or third construction satisfies the canonical determinant condition.
\end{lem}
\begin{proof}
The argument is the same as in \ref{can1} and hence omitted.
\end{proof}
\subsection{Non-emptiness and Sufficient Genericity Assumptions}\label{valid}
One of the most important tasks is to show the limit linear series stacks we constructed are non-empty. Moreover, the stacks in the first and second construction should contain sufficiently generic objects.

We now establish a non-emptiness result for the first induction. Due to the similarity in construction, it suffices to consider only the case $g'=g+4k_1+6-\left\lfloor{k_1\over3}\right\rfloor$.
\begin{prop}\label{nonempty}
 Given $\mathcal{G}^{k,\text{EHT}}_{2,\omega_g,d_{\bullet},a^{\Gamma}}(X_g)$ and $\mathcal{G}^{k+4,\text{EHT}}_{2,\omega_{g'},d'_{\bullet},c^{\Gamma'}_{g'}}(X_{g'})$ as in \ref{l1}. Suppose $\mathcal{G}^{k,\text{EHT}}_{2,\omega_g,d_{\bullet},a^{\Gamma}}(X_g)$ is non-empty with sufficiently generic objects, then $\mathcal{G}^{k+4,\text{EHT}}_{2,\omega_{g'},d'_{\bullet},c^{\Gamma'}_{g'}}(X_{g'})$ is also non-empty.
\end{prop}
\begin{proof}
  Denote $\GG=\mathcal{G}^{k+4,\text{EHT}}_{2,\omega_{g'},d'_{\bullet},c^{\Gamma'}_{g'}}(X_{g'})$ $\GG_r=\mathcal{G}^{k+4,\text{EHT}}_{2,\omega_r(A_rP_{r+1}),d^r_{\bullet},c^{\Gamma'}_r}(X_r)$ (\tb{Notation 17}).

  Denote $\mE_1,...,\mE_{g'}$ to be the underlying vector bundles of $\GG$.

  By Proposition \ref{nonempty3} and our assumption, $\GG_g$ is non-empty.

  Given $(c_{i,2g+17+11s})$ (\ref{parte}), $\tau(P_{g+10+11s})$ is trivial.

  Given $(c_{i,2h+8\ell-1})$ (\ref{partf}) and $\mE_{h+4p+1},...,\mE_{h+4p+4}$, one can inductively see that $\tau(P_j)$ is trivial for $j=h+4p+1,h+4p+3$; $\tau(P_j)$ is of the of the situation in Example \ref{config} (4) with only one other non-repeated vanishing order than $c_{i_1-1,2j-2}$, $c_{i_1,2j-2}$, $c_{i_2-1,2j-2}$, $c_{i_2,2j-2}$ ($i_1,i_2$ are the special indices), for $j=h+4p+2,h+4p+4$.

  Similarly, given $(c_{i,2g'-9})$ (take $\ell=k_1-3q+1$ in \ref{partf}) and $\mE_{g'-3},...,\mE_{g'}$, one sees that $\tau(P_j)$ is trivial for $j=g'-3,g'-2,g'$; $\tau(P_{g'-1})$ is of the of the situation in Example \ref{config} (4) with only one other non-repeated vanishing order than $c_{i_1-1,2g'-4}$, $c_{i_1,2g'-4}$, $c_{i_2-1,2g'-4}$, $c_{i_2,2g'-4}$.

  In these cases, if $\GG_{j-1}$ is non-empty, $\GG_j$ is always non-empty: given any $K$-valued object $((\mE_{j'},V_{j'})_{j'<j},(\phi_{j'})_{j'<j})$ in $\GG_{j-1}$, there is no obstruction in finding $\phi_j$, $V_j$ so that $((\mE_{j'},V_{j'})_{j'\le j},(\phi_{j'})_{j'\le j})$ is an object in $\GG_j$. \tb{Consequently, the proof reduces to showing $\GG_{g+9}$ and $\GG_{g+10+11s,g+20+11s}$ (Notation 17) are non-empty, for $s=0,...,q-2$.}

  \indent We first show $\GG_{g+9}$ is non-empty.

  Given $\mE_{g+8},\mE_{g+9}$, $(c_{i,2g+13}),(c_{i,2g+15})$ (\ref{parta}), $\tau(P_{g+9})=\{\{1,4,k+4\},\{5,k+3\}\}$. It is not hard to see that there exist $V_{g+8}\subset\Gamma(\mE_{g+8})$, $V_{g+9}\subset\Gamma(\mE_{g+9})$ with the desired vanishing sequences and this configuration at $P_{g+9}$. Moreover, $c_{5,2g+14},c_{k+3,2g+14}$ are repeated vanishing orders. So, if $\GG_{g+8}$ is non-empty, easy to see $\GG_{g+9}$ is non-empty.

  We claim a general object in $\GG_g$ must have configuration $\{\{1,k+1\},\{k+4\}\}$ at $P_{g+1}$: Let $\mE_{j_1},...,\mE_{j_{2p}}$ be the unstable bundles such that $j_t\le g$. Inducting from the base case in \ref{base}, one can show $c_{1,2j_{2p}-1},c_{k+1,2j_{2p}-1}$ obtained in \ref{C1} are always non-repeated and $c_{1,2j_{2p}-2}+c_{1,2j_{2p}-1}=c_{k+1,2j_{2p}-2}+c_{k+1,2j_{2p}-1}=\frac{1}{2}(d_{j_{2p}}-1)$. Therefore, there exists one set in $\tau(P_{j_{2p}+1})$ containing $1,k+1$. For $j=j_{2p}+1,...,g$, $c_{1,2j-2}+c_{k+1,2j-2}=c_{i_1,2j-2}+c_{i_2,2j-2}-1$. By applying Lemma \ref{aux}, \ref{sym} repeatedly, one gets $\tau(P_{j+1})$ contains a set containing $1,k+1$ ($j=j_{2p}+1,...,g$) and hence the claim. Given $\mE_{g+1}$ and $(c_{i,2g-1})$ (see Equation \ref{gvan}), it is clear that $\GG_{g+1}$ is non-empty and $\tau(P_{g+2})=\{\{3,k+1\},\{k+4\}\}$.

  By Lemma \ref{aux} again, one can show that $\GG_{g+2}$ is non-empty and a general object must have configuration $\{\{1,k+3\},\{2,k+2\},\{3,k+1\},\{k+4\}\}$ at $P_{g+3}$.

  Given $(c_{i,2g+5}),(c_{i,2g+7})$ (\ref{parta}), $\tau(P_{g+4})=\{\{3,k+1,k+4\}\}$. $\GG_{g+4}$ is non-empty given $\GG_{g+3}$ is non-empty with some object realizing this configuration at $P_{g+4}$.

  Furthermore, $c_{3,2g+6},c_{k+1,2g+6},c_{k+4,2g+6}$ are non-repeated and
  $$c_{3,2g+6}+c_{3,2g+7}=c_{k+1,2g+6}+c_{k+1,2g+7}=c_{k+4,2g+6}+c_{k+4,2g+7}=\frac{1}{2}(\deg(\mE_{g+4})-1).$$
  While there exists some $V_{g+3}$ for which $\tau^{V_{g+3}}(P_{g+4})=\{\{3,k+1,k+4\}\}$, by Theorem \ref{genericity0}, such $V_{g+3}$ has $\tau^{V_{g+3}}(P_{g+3})=\{\{1,k+3\},\{2,k+2\},\{3,k+1\},\{k+4\}\}$. To see $\GG_{g+3}$ is non-empty, one can choose $V_{g+3}$ and $\phi_{g+3}$ first and find appropriate $V_{g+2}\subset\Gamma(\mE_{g+2})$ in compatible with the choice of $\phi_{g+3}$. By Theorem \ref{genericity0} again, in general $\tau^{V_{g+2}}(P_{g+2})=\{\{3,k+1\},\{k+4\}\}$.

  Given $(c_{i,2g+5}),(c_{i,2g+7})$ (\ref{parta}), $\tau^{V_{g+4}}(P_{g+5})=\{\{1,4,k+1,k+4\},\{k\}\}$ for any feasible $V_{g+4}$. Given $\mE_{g+5}$, $4$ is a special index for $(c_{i,2g+8})$ and $(c_{i,2g+9})$. There exists $V_{g+5}\subset\Gamma(\mE_{g+5})$ satisfying the desired vanishing conditions which contains some $s_i\in\Gamma(\mO(c_{4,2g+8}P_{g+5}+c_{4,2g+9}P_{g+6}))$ for $i=1,4,k+1,k+4$ such that $\ord_{P_5}(s_i)=c_{i,2g+8}$, $\ord_{P_6}(s_i)=c_{i,2g+9}$. In particular, $\GG_{g+5}$ is non-empty. Moreover, $\tau^{V_{g+5}}(P_{g+6})=\{\{1,4,k+3,k+4\},\{k\}\}$.

  Similarly, one can show $\GG_{g+6}$ is non-empty. By Lemma \ref{aux} and Theorem \ref{genericity0}, conclude that $\tau^{V_{g+6}}(P_{g+7})=\{\{1,k+4\},\{4,k+3\},\{5,k+2\},\{6,k+1\},\{k\}\}$.

  Given $(c_{i,2g+13})$ (\ref{parta}) and $\mE_{g+7}$, one can calculate $$c_{6,2g+12}+c_{k+1,2g+12}=c_{5,2g+12}+c_{k+3,2g+12}-1.$$
  By Lemma \ref{aux}, $\tau^{V_{g+7}}(P_{g+8})$ contains $\{6,k+1\}$ in general.

  To show $\GG_8$ is non-empty, we choose $V_{g+6},V_{g+7},V_{g+8}$, $\phi_{g+7},\phi_{g+8}$ compatibly.

  Given $\mE_{g+8}$ and $c_{i,2g+15}$ (\ref{parta}), $c_{1,2g+14},c_{k,2g+14},c_{k+4,2g+14}$ are non-repeated and
  $$c_{1,2g+14}+c_{1,2g+15}=c_{k,2g+15}+c_{k,2g+15}=c_{k+4,2g+15}+c_{k+4,2g+15}=\frac{1}{2}(\deg(\mE_{g+8})-1),$$
  $\tau(P_{g+8})$ contains $\{1,k,k+4\}$. \footnote{When $k=5$, there are fewer non-repeated vanishing orders in $(c_{i,2g+14})$ and $\tau(P_{g+8})$ is different; the argument for non-emptiness of $\GG_{g+9}$ is even simpler in that case. See Lemma \ref{g6} also, for the derivation of dimension upper bound in that special case.} So, $\tau(P_{g+8})$ has to be $\{\{1,k,k+4\},\{6,k+1\}\}$.

  Recall $V_j$ is determined by its $(P_j,P_{j+1})$-adapted basis: $\{s^j_i\}_{i=1}^{k+4}$, where
  $$\ord_{P_j}(s_i)=c_{i,2j-2}, \ord_{P_{j+1}}(s^j_i)=c_{i,2j-1}.$$
  Denote $q$ to be the image of sections of $\mO(g-k_1+6,b'-(g-k_1+6))\subset\mE_{g+6}$ in $\PP\mE_{g+6}|_{P_{g+7}}$ and $q'$ to be the image of sections of $\mO(g-k_1+6,b'-(g-k_1+6))\subset\mE_{g+7}$ in $\PP\mE_{g+7}|_{P_{g+7}}$. Given the vanishing sequences, it is clear that $\phi_{g+7}(q)=q'$ must hold.

  Take any such $\phi_{g+7}$. Since $s^{g+7}_{k+3}$ is a canonical sections of $\mE_{g+7}$, $\phi_{g+7}$ determines $s^{g+6}_{k+3}$ and hence $s^{g+6}_4$. Given that $\tau(P_{g+6})=\{\{1,4,k+3,k+4\},\{k\}\}$, $s^{g+6}_1$, $s^{g+6}_{k+4}$ are also determined. Consequently, $s^{g+7}_1$, $s^{g+7}_{k+4}$ are determined. Similarly, $s^{g+7}_{6}, s^{g+7}_{k+1}$ are determined using the $\phi_{g+7}$ hereby chosen: $\phi_{g+7}(s^{g+7}_{k+1}|_{P_{g+7}})=s^{g+6}_{k+1}|_{P_{g+7}}$; yet $s^{g+6}_{k+1}$ is a canonical section of $\mE_{g+6}$.

  To determine $V_{g+6},V_{g+7}$, we are left with determining $s^{g+6}_k$ and $s^{g+7}_k$. Take any $V_{g+8}$ whose $(P_{g+8},P_{g+9})$-adapted basis satisfies
  $$s^{g+8}_6|_{P_{g+8}}=s^{g+8}_{k+1}|_{P_{g+8}},s^{g+8}_{5}|_{P_{g+9}}=s^{g+8}_{k+3}|_{P_{g+9}}$$
  and any $\phi_{g+8}$ such that $\phi_{g+8}(s^{g+7}_1|_{P_{g+8}})=s^{g+8}_1|_{P_{g+8}}$, $\phi_{g+8}(s^{g+7}_6|_{P_{g+8}})=s^{g+8}_6|_{P_{g+8}}$, one then determines $s^{g+7}_k$ via the condition $\phi_{g+8}(s^{g+7}_k|_{P_{g+8}})=s^{g+8}_k|_{P_{g+8}}=s^{g+8}_1|_{P_{g+8}}$. Correspondingly, $s^{g+6}_k$ is determined using $\phi_{g+7}$ chosen previously.

  For $V_{g+6},V_{g+7},V_{g+8}$ and $\phi_{g+7},\phi_{g+8}$ hereby determined and any $K$-valued object in $\GG_{g+5}$ one can find some $\phi_{g+6}$ to conclude that $\GG_{g+8}$ (and hence $\GG_9$) is non-empty.\\
  \indent Similarly, we show that $\GG_{g+10+11s,g+20+11s}$ (\tb{Notation 17}) is non-empty. Given $(c_{i,2g+39+11s})$ (\ref{parte}), one can inductively calculate the following:
  \begin{enumerate}
    \item $\tau(P_{g+20+11s})=\{\{6s+12,k+4\},\{6s+13,k+3\},\{1\}\}$ and is of the situation in \ref{config} (4).
    \item $\tau(P_{g+19+11s})$ is trivial.
    \item $\tau(P_{g+18+11s})=\{\{6s+8,k+4\},\{6s+9,k+3\},\{1\}\}$ and is of the situation in \ref{config} (4).
  \end{enumerate}
  Therefore, the question reduces to showing: (1) $\GG_{g+11+11s,g+17+11s}$ (\tb{Notation 17}) is non-empty; (2) a general object of it admits the desired configuration at $P_{g+18+11s}$.

  First of all, we choose $V_{g+11+11s},V_{g+12+11s}$, $\phi_{g+11+11s}$, $\phi_{g+12+11s}$ consistently.

  A general $V_{g+10+11s}$ satisfying the desired vanishing conditions must have $$\tau^{V_{g+10+11s}}(P_{g+11+11s})=\{\{1,2,6s+8,k+1\},\{6s+5,k+4\},\{k\}\}.$$
  \indent Given $\mE_{g+11+11s}$ and by \ref{aux}, \ref{genericity0}, there exists $V_{g+11+11s}$ such that
    $$\tau^{V_{g+11+11s}}(P_{g+11+11s})=\{\{1,2,6s+8,k+1\},\{6s+5,k+4\},\{k\}\},$$
    $$\tau^{V_{g+11+11s}}(P_{g+12+11s})=\{\{1\},\{2\},\{6s+5,k+4\},\{6s+8,k+1\},\{6s+9,k+3\},\{6s+10,k+2\},\{k\}\}.$$

  Similarly, given $\mE_{g+13+11s}$ and $(c_{i,2g+25+22s})$ (\ref{parte}), for a general $V_{g+13+11s}$ satisfying the vanishing condition,
  $$\tau^{V_{g+13+11s}}(P_{g+13+11s})=\{\{1\},\{2\},\{6s+5,6s+8,k,k+1,k+4\}\},$$
  $$\tau^{V_{g+13+11s}}(P_{g+14+11s})=\{\{1\},\{6s+6,6s+9,k,k+1,k+4\}\}.$$

  After deriving $(c_{i,2g+23+22s})$ from $(c_{i,2g+25+22s})$ (\ref{parte}) and $\mE_{g+13+11s}$, one can see that there exists $V_{g+12+11s}$ such that $\tau^{V_{g+12+11s}}(P_{g+13+11s})=\{\{1\},\{2\},\{6s+5,6s+8,k,k+1,k+4\}\}.$

  Still denote $\{s^j_i\}$ to be a $(P_j,P_{j+1})$-adapted basis. Choose $s^{g+11+11s}_{6s+5}$, $s^{g+12+11s}_{6s+5}$, which are not sections of one summand of the corresponding bundles. This immediately determines $\phi_{g+12+11s}$ since $\tau(P_{g+12+11s})$ is of the situation in Example \ref{config} (4). Consequently, $s_i^{g+11+11s}$, $s_i^{g+12+11s}$ are determined for $i=6s+8,k,k+1,k+4$. This further determines $\phi_{g+11+11s}$ as well as $s_1^{g+11+11s},s_2^{g+11+11s}$ (and hence $V_{g+11+11s}$); and even further, $s_1^{g+12+11s},s_2^{g+12+11s}$ (and hence $V_{g+12+11s}$). Eventually, one can find $\phi_{g+13+11s}$ according to $V_{g+12+11s}$ thus chosen and conclude that $\GG_{g+11+11s,g+13+11s}$ is non-empty.

  Fix some object of $\GG_{g+11+11s,g+13+11s}$. Given $(c_{i,2g+25+22s})$ (\ref{parte}) and $\mE_{g+14+11s}$, one can derive $(c_{i,2g+27+22s})$ and find $V_{g+14+11s}$, $\phi_{g+14+11s}$ in compatible with $V_{g+13+11s}$ determined by the fixed object; furthermore, $\tau^{V_{g+14+11s}}(P_{g+15+11s})=\{\{1\},\{6s+6,6s+9,k,k+3,k+4\}\}.$

  Similarly, after deriving $(c_{i,2g+29+22s})$ from $(c_{i,2g+27+22s})$, one can find $V_{g+15+11s}$, $\phi_{g+15+11s}$ in compatible with $V_{g+14+11s}$ previously chosen and
  $$\tau^{V_{g+15+11s}}(P_{g+16+11s})=\{\{1\},\{6s+8,6s+9,k,k+3,k+4\}\}.$$
  \indent Lastly, we choose $V_{g+16+11s},V_{g+17+11s}$ as follows. By referring to $\mE_{g+16+11s}$, $\mE_{g+17+11s}$, easy to see $\tau(P_{g+17+11s})$ is of the situation in \ref{config} (4). Therefore, fixing $s_{6s+8}^{g+16+11s}$, $s_{6s+8}^{g+17+11s}$ which are not sections of one summand of the corresponding bundles determines $\phi_{g+17+11s}$. Moreover, $s_i^{g+16+11s}$ ($i=6s+9,k,k+3,k+4$) is determined by the condition $s_i^{g+16+11s}|_{P_{g+16+11s}}=s_{6s+8}^{g+16+11s}|_{P_{g+16+11s}}.$

  Consequently, $s_i^{g+17+11s}$ ($i=6s+9,k,k+3,k+4$) is determined by the condition $$\phi_{g+17+11s}(s_i^{g+16+11s}|_{P_{g+17+11}})=s_i^{g+17+11s}|_{P_{g+17+11}}.$$
  Eventually, choose any $\phi_{g+16+11s}$ such that $\phi_{g+16+11s}(s_{6s+8}^{g+15+11s}|_{P_{g+16+11s}})=s_{6s+8}^{g+16+11s}|_{P_{g+16+11s}}.$ This will determine $s_1^{g+16+11s}$, $s_1^{g+17+11s}$ and hence $V_{g+16+11s}$, $V_{g+17+11s}$.

  Thus, $\GG_{g+11+11s,g+17+11s}$ is non-empty. So is $\mathcal{G}^{k+4,\text{EHT}}_{2,\omega_{g'},d'_{\bullet},c^{\Gamma'}_{g'}}(X_{g'})$.
  \end{proof}
Next, we verify the existence of sufficiently generic objects:
\begin{prop}\label{suffred}
Let $\GG=\mathcal{G}^{k+4,\text{EHT}}_{2,\omega_{g'},d'_{\bullet},c^{\Gamma'}_{g'}}(X_{g'})$ be as defined in Proposition \ref{nonempty}. If $\GG$ is non-empty, then it contains sufficiently generic objects ($N=g'-g$, $n=4$).
\end{prop}
\begin{proof}
Denote $\ti{c}^{\Gamma'}$ to be the $(k'+4)\times (2g'-2)$ matrix obtained from $c^{\Gamma'}_{g'}$ using Lemma \ref{C1}. For $j=j_t$ ($t=1,...,2T$)\footnote{Recall that $\mE_1,...,\mE_{j_{2T}}$ are the unstable bundles among $\mE_1,...,\mE_{g'}$.}, denote $q^t_1$ (resp. $q^t_2$) to be the image of sections of the destabilizing summand of $\mE_{j_t}$ in $\PP\mE_{j_t}|_{P_{j_t}}$ (resp. $\PP\mE_{j_t}|_{P_{j_t+1}}$).

Recall that being sufficiently generic means $q^t_2$ is not glued to $q^{t+1}_1$ via $\phi_{j_1+1},...,\phi_{j_{t+1}}$ in the $c_{i,2j_t-1}$-th order, where $i>k'$ and $c_{i,2j_t-1}$ is a non-repeated vanishing order. By Corollary \ref{AC1}, when $t$ is odd, all $c_{i,2j_t-1}$ ($i>k'$) are repeated vanishing orders. Hence, it suffices to consider cases where $t<2T$ is even. Moreover, when $t$ is even, at most $c_{k'+1,2j_t-1}$, $c_{k'+4,2j_t-1}$ are non-repeated.

Note that $\mathcal{G}^{5,\text{EHT}}_{2,\omega_6,d_{\bullet},a^{\Gamma}}(X_6)$ in Proposition \ref{base} has sufficiently generic objects for a trivial reason: there are only two unstable bundles among $\mE_1,...,\mE_6$. In general, suppose the moduli stack constructed for some pair $(g,k)$ has sufficiently generic objects. We claim the moduli stack constructed for $(g',k')=(g+4k_1+6-q,k+4)$ also has sufficiently generic objects. The statement fails only if the following happens: for some even $t$ smaller than $2T$, $q^t_2$ is glued to $q^{t+1}_1$ in $c_{i,2j_t-1}$-th order for some $i\le k'$ via $\phi_{j_t+1},...,\phi_{j_{t+1}}$ and ANY $(k'+4)$-dimensional space of sections $\ti{V}_{j_t+1}$ with the desired vanishing sequences at $P_{j_t+1},P_{j_t+2}$ has the configuration $\tau^{\ti{V}_{j_t+1}}(P_{j_t+1})$ containing a set $S$ such that $S\supset\{i,k'+1\}$ or $S\supset\{i,k'+4\}$.

We first consider all $j_t$ such that $j_t>g$ and $t<2T$ is even. They are $j_t=g+8,g+13+11s$ ($s<q-2$). For $j_t=g+8$, it suffices to notice that $j_{t+1}=g+10$ and $c_{1,2g+17}$ is the only non-repeated vanishing order in $(c_{i,2g+17})_{i\le k'}$. Using Lemma \ref{parta}, calculate $c_{i,2g+16}=g-k_1+5$; by Proposition \ref{AM1}, one can also compute that $c_{k'+1,2g+16}=g+k_1+12$. Since
$$c_{1,2g+16}+c_{k'+4,2g+16}>c_{1,2g+16}+c_{k'+1,2g+16}=2g+17>c_{i_1,2g+16}+c_{i_2,2g+16}=2g+16,$$
where $i_1,i_2$ are the special indices of $(c_{i,2g+16})$, $(c_{i,2g+17})$, by Lemma \ref{aux}, we can conclude that sufficient genericity will not be violated for $j_t=g+8$ in general.

Similarly, by Lemma \ref{parte}, $c_{1,2g+26+22s}=g-k_1+10+11s$; by Proposition \ref{AM1}, $c_{k'+1,2g+26+22s}=g-k_1+17+11s$. In particular, $c_{1,2g+26+11s}+c_{k'+1,2g+26+11s}>c_{i_1,2g+26+22s}+c_{i_2,2g+26+22s}$ and by Lemma \ref{aux}, Theorem \ref{genericity0}, we can conclude that sufficient genericity will not be violated for $j_t=g+13+11s$.

One still needs to consider $j_t<g$ where $t<2T$ is even. But based on our construction, the new non-repeated vanishing orders to consider are relatively bigger than the ones in the previous inductive step. Based on the previous argument, one can conclude immediately that sufficient genericity will not be violated for any $t$ such that $j_t<g$.

This concludes the proof.
\end{proof}
\begin{rem}
  Notice that the existence of sufficiently generic objects for $n=4$ follows from the fact that all the non-repeated vanishing orders are large enough so that no obstruction could occur. By exactly the same justification, one can establish the existence of sufficiently generic objects for $n=2$ as well as the existence of sufficiently generic objects satisfying the strengthened genericity assumption (Property \tb{A}) in Proposition \ref{even1}.

  Consequently, we can establish the non-emptiness result for the second induction.
\end{rem}
\begin{prop}\label{nonempty2}
 Let $\GG=\mathcal{G}^{k+2,\text{EHT}}_{2,\omega_{g'},d'_{\bullet},c^{\Gamma'}_{g'}}(X_{g'})$ be as given in \ref{l2}. Then $\GG$ is non-empty.
\end{prop}
\begin{proof}
  By the previous remark and Proposition \ref{Induction}, $\GG_g$ (\tb{Notation 17}) is non-empty.

  For all $j>g$, $\mE_j$ are semi-stable and decomposable. Given the special indices for $j=g+1,...,g+k_1$ (see the proof of \ref{twosp}), one can show that for general choices of $V_j$, the configurations at $P_j$ should be $\{\{2t-1,k+1\},\{k+2\}\}$ and the point associated to $\{2t-1,k+1\}$ is the image of a canonical section of $\mE_j$.

  By the same argument as in the proof of Proposition \ref{nonempty}, one can conclude that there exists a set in $\tau^{V_g}(P_{g+1})$ containing $\{1,k+1\}$. Since the associated sequence of points at each node contains only two points, compatible gluing isomorphisms exist. By induction, $\GG_{g+k_1}$ is non-empty.

  Consequently, $\tau^{V_{g+k_1}}(P_{g+k_1+1})=\{\{k,k+1\},\{k+2\}\}$. By direct computation, $$c_{k,2g+2k_1}+c_{k+1,2g+2k_1}=c_{i_1,2g+2k_1}+c_{i_2,2g+2k_1}-1.$$
  It follows from Lemma \ref{aux} that for $j:g'-1\ge j\ge g+k_1+1$,
  $$\tau^{V_j}(P_j)=\begin{cases}
  \{\{k,k+1\},\{k+2\}\} & j=g+k_1+1\\
  \{\{1\},\{k,k+1\},\{2t,k+2\}\}  &\text{ otherwise}
  \end{cases}$$ $$\tau^{V_j}(P_{j+1})=\begin{cases}
  \{\{1\}\{k,k+1\},\{2,k+2\}\} & j=g+k_1+1\\
  \{\{1\},\{k,k+1\},\{2t+2,k+2\}\}  &\text{ otherwise}
  \end{cases}$$
  It is then not hard to see $\GG_{g'-1}$ is non-empty.

  Since $\tau(P_{g'})$ is trivial, $\GG$ is also non-empty.
\end{proof}
Eventually, we establish the non-emptiness result for the third construction.
\begin{prop}\label{nonempty4}
Let $\GG=\GG^{k+3,\text{EHT}}_{2,\oo_{g'},d'_{\bullet},c^{\Gamma'}_{g'}}(X_{g'})$ be as given in \ref{l3}. Then $\GG$ is non-empty.
\end{prop}
\begin{proof}
  By Proposition \ref{even1}, $\GG_g$ is non-empty.

  Given $(c_{i,2g+2k_1+3})$ (\ref{evenonly}) and $\mE_{g+k_1+4}$, one can compute $(c_{i,2g+2k_1+5})$ to see that $\tau(P_{g+k_1+4})$ is trivial. Moreover, given $\mE_{j}$ for $j\ge g+k_1+4$, conclude by induction that $\tau(P_j)$ is either trivial or of the situation in Example \ref{config} (4) (with no other non-repeated vanishing orders). Therefore, the question reduces to showing $\GG_{g+k_1+4}$ is non-empty.

  Given $(c_{i,2g+9})$ (\ref{evenonly}) and $\mE_{g+5+j}$ ($j = 1,...,k_1-3$), one can inductively show that a general $V_{g+5+j}$ satisfying the desired vanishing conditions must have
  $$\tau^{V_{g+5+j}}(P_{g+5+j})=\{\{2j+3,k+1\},\{k-1\},\{k\}\},\tau^{V_{g+5+j}}(P_{g+6+j})=\{\{2j+5,k+1\},\{k-1\},\{k\}\}.$$
  Given $(c_{i,2g+2k_1+3})$ (\ref{evenonly}) and $\mE_{g+k_1+2}$, by \ref{genericity0} one can find $V_{g+k_1+2}$ such that
  $$\tau^{V_{g+k_1+2}}(P_{g+k_1+2})=\{\{k-2,k+1\},\{k-1\},\{k\}\},\tau^{V_{g+k_1+2}}(P_{g+k_1+3})=\{\{k-2,k+1\},\{k-1,k\}\}.$$
  Meanwhile, given $(c_{i,2g+2k_1+3})$ (\ref{evenonly}), any feasible $V_{g+k_1+3}$ must have
  $$\tau^{V_{g+k_1+3}}(P_{g+k_1+3})=\{\{k-2,k+1\},\{k-1,k\}\} \text{ and }\tau^{V_{g+k_1+3}}(P_{g+k_1+4}) \text{ is trivial}.$$
  It is then not hard to see that $\GG$ is non-empty if $\GG_{g+5}$ is non-empty and contains an object whose configuration at $P_{g+6}$ is $\{\{5,k+1\},\{k-1\},\{k\}\}$.

  By the same argument as in the proof of \ref{nonempty}, objects in $\GG_g$ must have configuration $\{\{1,k+1\}\}$ at $P_{g+1}$. Given $(c_{i,2g+1})$ (\ref{evenonly}), one can directly compute that a general $V_{g+1}$ would have:
  $$\tau^{V_{g+1}}(P_{g+2})=\{\{1,4,k+1,k+2\},\{k\},\{k+3\}\}.$$
  A general $V_{g+2}$ should have
  $$\tau^{V_{g+2}}(P_{g+2})=\{\{1,4,k+1,k+2\},\{k\},\{k+3\}\}, \tau^{V_{g+2}}(P_{g+3})=\{\{1,4,k+1\},\{k\}\}.$$
Deriving $(c_{i,2g+5})$ from $(c_{i,2g+1})$ (\ref{evenonly}) and applying \ref{aux}, \ref{genericity0}, one can see that there exists $V_{g+3}$ such that $\tau^{V_{g+3}}(P_{g+3})=\{\{1,4,k+1\},\{k\}\}$ and $\tau^{V_{g+3}}(P_{g+4})=\{\{1\},\{4,k+1\},\{2,k+3\},\{3,k+2\},\{k\}\}$. So $\GG_{g+3}$ is non-empty. Note also $\tau(P_{g+4})$ is of the situation in Example \ref{config} (4).

  Fix some object of $\GG_{g+3}$ and denote $\{{s^{j}_i}\}$ to be a $(P_{j},P_{j+1})$-adapted basis of $V_{j}$. We choose $V_{g+4}$ in the following way: choose $s^{g+4}_4$ (and hence $s^{g+4}_{k+1}$), not a section of one summand of $\mE_{g+4}$; this determines a unique feasible $\phi_{g+4}$. Consequently, $s^{g+4}_1$, $s^{g+4}_k$ are determined by the following conditions:
  $$\phi_{g+4}(s^{g+3}_1|_{P_{g+4}})=s^{g+4}_1|_{P_{g+4}},\phi_{g+4}(s^{g+3}_k|_{P_{g+4}})=s^{g+4}_k|_{P_{g+4}}.$$
  This determines $V_{g+4}$. Given $(c_{i,2g+7})$ (\ref{evenonly}) and $\mE_{g+4}$, one can compute $$c_{4,2g+6}+c_{k+1,2g+6}=2g+4=c_{i_1,2g+6}+c_{i_2,2g+6}-1$$ ($i_1,i_2$ are the special indices). By Lemma \ref{aux}, $\tau^{V_{g+4}}(P_{g+5})=\{\{4,k+1\},\{1\},\{k\}\}$.

  Eventually, $c_{4,2g+8}+c_{4,2g+9}=c_{k+1,2g+8}+c_{k+1,2g+9}={1\over2}(d'_{g+5}-1)$. Hence a general feasible $V_{g+5}$ must have $\tau^{V_{g+5}}(P_{g+5})=\{\{4,k+1\},\{1\},\{k\}\},\tau^{V_{g+5}}(P_{g+6})=\{\{5,k+1\},\{k-1\},\{k\}\}.$

  This shows that $\GG_{g+5}$ is non-empty and hence $\GG$ is non-empty.
\end{proof}
\subsection{Semi-stable Locus}
Eventually, we check that there is a non-empty open locus on which the semi-stability condition holds.
\begin{lem}\label{sst}
Starting from $\mathcal{G}^{5,\text{EHT}}_{2,\omega_6,d_{\bullet},a^{\Gamma}}(X_6)$, every moduli stack constructed in \ref{l1} contains a non-empty locus of objects satisfying the $\ell$-semi-stability condition.
\end{lem}
\begin{proof}
We first break $X_{g'}$ into various blocks $X^1,...,X^r$ satisfying the following properties:
\begin{enumerate}
  \item $X^i=C_{j^i}\cup...\cup C_{j^i+n_i}$ consists of some consecutive components of $X_{g'}$;
  \item $X^i$ and $X^{i+1}$ intersect at a nodal point;
  \item For every $i$ and $j=j^i,...,j^i+n_i$, the vector bundles $\mE_j$ are of one of the following situations:
  \begin{enumerate}
    \item $\mE_{j^i}$,...,$\mE_{j^i+n_i}$ are all semi-stable;
    \item $\mE_{j^i},\mE_{j^i+n_i}$ are unstable, and $\mE_j$ are semi-stable for $j:j^i<j<j^i+n_i$.
  \end{enumerate}
\end{enumerate}
We shall refer to a block $X^i$ as a \textit{basic block} of type (a) (or (b)) of $X_{g'}$.

To check a $K$-valued object $((\mE_j,V_j),(\phi_j))$ satisfies the $\ell$-semi-stability condition, by \ref{blocks}, it suffices to show that the bundle $\mE^i$ on $X^i$ determined by $((\mE_j)^{j^i+n_i}_{j=j^i},(\phi_j)_{j=j^i+1}^{j^i+n_i})$ is $\ell$-semi-stable for every $i$.\footnote{Technically, given $(\mE_j,\phi_j)$, one needs to take an appropriate twisting at the nodes to obtain some bundle $\mE''$ (see Remark \ref{twist}) and show $\mE''$ is $\ell$-semi-stable. However, by the observation in Example \ref{ex2}, it suffices to show $\mE^i$ is $\ell$-semi-stable.} Clearly, $((\mE_j)^{j^i+n_i}_{j=j^i},(\phi_j)_{j=j^i+1}^{j^i+n_i})$ satisfies the conditions in Situation \ref{s1}. By Proposition \ref{ssimple} it suffices to check there does not exist an invertible sub-sheaf $\mL$ of $\mE^i$ such that $\mL|_{C_{j^i}}$, $\mL|_{C_{j^i+n_i}}$ are the destabilizing summands of $\mE_{j^i}$ and $\mE_{j^i+n_i}$ resp. and $\mE|_{C_j}$ is one summand of $\mE_j$.

In \ref{base}, we have shown that the lemma holds for the base case. In each induction step there are some more basic blocks than in the preceding step. For the first construction, it suffices to look at the following basic blocks: $C_{g+4}\cup...\cup C_{g+8}$, $C_{g+10+11s}\cup...\cup C_{g+13+11s}$ ($0\le s\le q-2$), since all other basic blocks are of type (a).

For $C_{g+4}\cup...\cup C_{g+8}$, notice that $\mE_{g+6}$ is of the form $\mL_1\op\mL_2$ where $\mL_1\not\cong\mL_2$. Denote $q$ to be the image of sections of $\mO(g+k_1+5,b'-(g+k_1+3))\subset\mE_{g+7}$ in $\PP\mE_{g+7}|_{P_{g+7}}$. When choosing $\phi_{g+7}$, we can further require that $\phi_{g+7}^{-1}(q)$ is not the image of either of the canonical sections of $\mE_{g+6}$. This will guarantee the $\ell$-semi-stability of the bundle given by $((\mE_j)^{g+8}_{j=g+4},(\phi_j)_{j=g+5}^{g+8})$.

For $C_{g+10+11s}\cup...\cup C_{g+13+11s}$, note that $\mE_{g+11+11s}$ is of the form $\mL_1\op\mL_2$ where $\mL_1\not\cong\mL_2$. Denote $q$ to be the image of sections of the destabilizing summand of $\mE_{g+10+11s}$. In determining $\phi_{g+11+11s}$ we further require $\phi_{g+11+11s}(q)$ is not the image of either of the canonical sections of $\mE_{g+11+11s}$. This will guarantee the $\ell$-semi-stability of the bundle given by $((\mE_j)^{g+8}_{j=g+4},(\phi_j)_{j=g+5}^{g+8})$.
\end{proof}
For the second construction, $K$-valued objects in $\mathcal{G}^{k+2,\text{EHT}}_{2,\omega_{g'},d'_{\bullet},c^{\Gamma'}_{g'}}(X_{g'})$ satisfy the property that $\mE_j$ is semi-stable for all $j>g$. The non-emptiness of $\ell$-semi-stable follows trivially.

Lastly, we verify the semi-stability condition for our third construction:
\begin{lem}\label{sst2}
Every moduli stack constructed in \ref{l3} contains a non-empty locus of objects satisfying the $\ell$-semi-stability condition.
\end{lem}
\begin{proof}
Similar to the proof of Lemma \ref{sst}, it suffices to consider the basic block $C_{g+1}\cup...\cup C_{g+5}$ and justify that one can find $((\mE_j,V_j),(\phi_j))$ such that the bundle $\mE'$ given by $((\mE_j)_{j=g+1}^{g+5},(\phi_j)_{j=g+2}^{g+5})$ is $\ell$-semi-stable.

To see this, note that $\mE_{g+3}$ is of the form $\mL_1\op\mL_2$ where $\mL_1\not\cong\mL_2$. Denote $q$ to be the point associated to the set $\{1,4,k+1\}$ in $\tau(P_{g+3})$. In determining $\phi_{g+3}$ we further require $\phi_{g+3}(q)$ is not the image of either of the canonical sections of $\mE_{g+3}$. This will guarantee the $\ell$-semi-stability of the bundle given by $((\mE_j)^{g+1}_{j=g+5},(\phi_j)_{j=g+2}^{g+5})$ and hence proves the lemma.
\end{proof}
\bibliographystyle{amsalpha}
\bibliography{myrefs}
\end{document}